\documentclass{TPmod}
\usepackage{url}
\usepackage{setspace}
\usepackage{scrextend}

\usepackage{amsthm}
\usepackage{amssymb}
\usepackage{mathtools}
\usepackage{xy}
\usepackage{tikz}
\usepackage{tikz-cd}
\usetikzlibrary{shapes.geometric,arrows,tikzmark}

\usepackage{amsthm,amsfonts,amsmath,amscd,amssymb}
\usepackage{latexsym,bbm,epsfig,epic,eepic,oldgerm,psfrag}
\usepackage{mathrsfs}
\usepackage{graphics,graphicx}
\usepackage{caption}

\usepackage{xfrac} 
\usepackage{enumitem}

\usepackage{tikz-cd,amssymb}
\usetikzlibrary{tikzmark,quotes}

\theoremstyle{plain}
\newtheorem{lemma}[thm]{Lemma}

\theoremstyle{definition}

\newtheorem{definition}[thm]{Definition}
\newtheorem{remark}[thm]{Remark}

\theoremstyle{remark}

\numberwithin{equation}{section}

\newtheorem{proposition}[thm]{Proposition}

\newcommand{\m}{\mathfrak{m}}

\renewcommand{\P}{\mathbb{P}}

\newcommand{\DD}{\mathbb{D}}

\newcommand{\SL}{\mathscr{L}}
\renewcommand{\SS}{\mathscr{S}}
\newcommand{\SH}{\mathscr{H}}

\newcommand{\SP}{\mathscr{P}}
\newcommand{\SC}{\mathscr{C}}
\newcommand{\SB}{\mathscr{B}}

\newcommand{\dd}{\partial}

\newcommand{\sse}{\subseteq}
\newcommand{\lr}{\longrightarrow}

\newcommand{\pt}{\text{point}}
\newcommand{\conv}{\operatorname{conv}}

\newcommand{\GL}{\operatorname{GL}}

\newcommand{\ol}{\overline}

\newcommand{\st}{\text{st}}

\newcommand{\convex}{\operatorname{conv}}

\newcommand{\Vol}{\operatorname{Vol}}
\def\Op{{\mathcal O}{\it p}\,}
\newcounter{daggerfootnote}

\newcommand{\bD}{\mathbb{D}}

\newcommand{\bT}{\mathbb{T}}

\newcommand{\cB}{\mathcal{B}}

\newcommand{\cN}{\mathcal{N}}

\newcommand{\oX}{\overline{X}}

\newcommand{\mC}{\mathcal{C}}

\newcommand{\mH}{\mathcal{H}}

\newcommand{\mM}{\mathcal{M}}
\newcommand{\mN}{\mathcal{N}}
\newcommand{\fN}{\mathfrak{N}}
\newcommand{\mA}{\mathcal{A}}
\newcommand{\mB}{\mathcal{B}}
\newcommand{\mP}{\mathcal{P}}
\newcommand{\fB}{\mathfrak{B}}

\newcommand{\CP}{\mathbb{CP}}

\newcommand{\PxP}{\mathbb{CP}^1 \times \mathbb{CP}^1}

\newcommand{\BlIII}{Bl_3(\mathbb{CP}^2)}
\newcommand{\BlIV}{Bl_4(\mathbb{CP}^2)}

\newcommand{\Tpqr}{\Theta^{n_1,n_2,n_3}_{p,q,r}}

\newcommand{\del}{\partial}

\newcommand{\bv}{\textbf{v}}
\newcommand{\bw}{\textbf{w}}
\newcommand{\bbw}{\emph{\textbf{w}}}
\newcommand{\bbv}{\emph{\textbf{v}}}
\newcommand{\bbf}{\emph{\textbf{f}}}
\newcommand{\be}{\emph{\textbf{e}}}

\begin{document}

\title{Full Ellipsoid Embeddings and Toric Mutations\vspace{0.25cm}}

\author{Roger Casals}
\address{University of California Davis, Dept. of Mathematics, Shields Avenue, Davis, CA 95616, USA}
\email{casals@math.ucdavis.edu}

\author{Renato Vianna}
\address{Institute of Mathematics, Federal University of Rio de Janeiro (UFRJ), Rio de Janeiro, Brazil}
\email{renato@im.ufrj.br}

\begin{abstract} {\sc Abstract:} This article introduces a new method to construct volume-filling symplectic embeddings of 4-dimensional ellipsoids by employing polytope mutations in toric and almost toric varieties. The construction uniformly recovers the full sequences for the Fibonacci Staircase of McDuff-Schlenk, the Pell Staircase of Frenkel-M\"uller and the Cristofaro-Gardiner-Kleinman Staircase, and adds new infinite sequences of ellipsoid embeddings. In addition, we initiate the study of symplectic-tropical curves for almost toric fibrations and emphasize the connection to quiver combinatorics.
\end{abstract}
\thispagestyle{empty}

\maketitle

\section{Introduction}

The novel contribution of the article is the use of polytope mutations as a method for generating full ellipsoid embeddings in infinite staircases. In particular, we give a uniform toric mutation explanation for the full embeddings of ellipsoids, see \cite{LMS13}, in the two classical staircases: the McDuff-Schlenk Fibonacci Staircase \cite{McDuff1,McDuffSchlenk1} and the Frenkel-M\"uller Pell Staircase \cite{PellStaircase}, as well as the more recent discovery by Cristofaro-Gardiner-Kleinman \cite{DCG_Staircase}. In addition, we discuss nine 4-dimensional symplectic toric domains $(X,\omega)$ with a sharp infinite staircase in its symplectic ellipsoid embedding function. These are conjecturally all such domains with this property, as formulated in \cite[Conjecture 6.1]{DCGetAL}. This is the content of the first part of the article, in Sections \ref{sec:Preliminaries} and \ref{sec:proofmain}.

The manuscript also develops new techniques in the study of symplectic-tropical curves in almost toric fibrations, incorporating the works of M. Symington \cite{SymingtonLeung,Symington} and G. Mikhalkin \cite{Mikhalkin1,Mikhalkin2} into the study of symplectic ellipsoid embeddings, and the connections with the theory of cluster algebras and quiver mutations. This is the content of the second part of the article, developed in Section \ref{sec:Symp_Trop}.


\subsection{Context and Results} Let $(\R^4,\omega_\st)$ be standard symplectic 4-space, $a,b\in\R^{>0}$, and consider the symplectic ellipsoid $$E(a_1,a_2):=\left\{(x_1,y_1,x_2,y_2)\in\R^4:\frac{x_1^2+y_1^2}{a_1}+\frac{x_2^2+y_2^2}{a_2}\leq1\right\}\sse(\R^4,\omega_\st).$$
Let $(X,\omega_X)$ be a 4-dimensional symplectic almost toric domain, as introduced in Section \ref{sec:Preliminaries}. The first goal of this article is to discuss the existence of infinite staircases for the function
$$c_X(a):=\inf\{\sigma\mbox{ such that }\exists i:(E(1,a),\omega_\st)\hookrightarrow  (X,\sigma\cdot\omega_\st)\},$$
where $i$ denotes a {\it symplectic} embedding, i.e. $i^*\omega_X=\omega_\st$. A {\it symplectic} embedding is volume preserving, and thus the volume bound $\pi^2a\leq 2\sigma^2\cdot\Vol(X,\omega_\st)$ implies
$$\frac{\pi\sqrt{a}}{\sqrt{2\Vol(X,\omega_\st)}}\leq c_X(a).$$
The function $c_X(a)$ is non-decreasing and continuous. The symplectic non-squeezing phenomenon \cite{Gromov85} states the existence of values $a\in\R^{>0}$ for which the above inequality is actually strict. The ground-breaking work of D. McDuff and F. Schlenk \cite{McDuffSchlenk1} establishes for $(X,\omega_\st)=(E(1,1),\omega_\st)$ the existence of a non-empty interval $I_X=[\alpha_X,\Omega_X]\sse\R^{>0}$ and a convergent sequence $S=\{s_n\}_{n\in\N}\sse I_X$ such that $c_X(s_n)$ coincides with the volume lower bound and $c_X|_{I_X\setminus S}$ is strictly larger than the volume bound. The exact graph for the function $c_X$ in this case is depicted in \cite[Figure 1.1]{McDuffSchlenk1}.

\begin{definition}\label{def:infinitestairs} A 4-dimensional symplectic domain $(X,\omega_\st)$ is said to admit a {\it sharp infinite staircase} if there exists a non-zero interval $I_X=[\alpha_X,\Omega_X]\sse\R^{>0}$ and an infinite sequence $S=\{s_n\}_{n\in\N}\sse I_X$ of distinct points converging to $\Omega_X$, such that $c_X(s_n)$ coincides with the volume lower bound and $c_X|_{(I_X\setminus S)}$ is strictly greater than the volume bound. The {\it sharp points} of a sharp infinite staircase $(I_X,S)$ are the points $S=\{s_n\}_{n\in\N}$ where the volume bound for $c_X|_{I_X}$ is sharp.\hfill$\Box$
\end{definition}

The existence of a sharp infinite staircase for a 4-dimensional symplectic domain $(X,\omega_\st)$ has been a central question in the study of low-dimensional quantitative symplectic geometry, as beautifully developed by D. McDuff, R. Hind, M. Hutchings, F. Schlenk and many others \cite{DCG1,DCG2,McDuff1,McDuffSchlenk1,Schlenk1}. In the first crucial discovery \cite{McDuffSchlenk1}, it is shown that $(X,\omega_\st)=(\bD^4(1),\omega_\st)$, the standard unit ball, admits an infinite staircase. Two additional results were obtained for the polydisk $(X,\omega_\st)=(\bD^2(1)\times \bD^2(1),\omega_\st\oplus\omega_\st)$ by Frenkel-M\"uller \cite{PellStaircase}, and for the ellipsoid $(X,\omega_\st)=(E(2,3),\omega_\st)$ by Cristofaro-Gardiner-Kleinman \cite{DCG_Staircase}.

\begin{remark} Definition \ref{def:infinitestairs} is the notion we use in this manuscript, we refer to M. Usher's \cite[Section 1.3]{Usher_Staircases} for comparison. The article \cite{DCGetAL} also considers {\it non-sharp} staircases where the volume lower-bound is not sharp. This corresponds to their $J=3$ case, which would not abide by Definition \ref{def:infinitestairs}.\hfill$\Box$
\end{remark}

Let us consider the almost toric base diagrams in Figure \ref{fig:IntroTable}. Each such base diagram $B$ yields a closed monotone symplectic 4-manifold $(X_B,\omega_B)$, as explained in \cite[Section 3]{SymingtonLeung}, cf.~ also \cite[Section 2.3]{Vianna1} and \cite[Section 3]{Vianna2}. The almost toric diagrams in the first row yield $\C\P^2$ (left) and $\C\P^1\times\C\P^1$ (right), the four diagrams in the second row correspond to $Bl_3(\C\P^2)$ and the last two diagrams in the fourth row yield $Bl_4(\C\P^2)$. In this manuscript, we focus on 4-dimensional symplectic domains which are obtained by removing (neighborhoods of) symplectic divisors and Lagrangian 2-spheres from closed symplectic 4-manifolds. For each base diagram $B$ in Figure \ref{fig:IntroTable}, we choose the symplectic divisors associated to the {\it blue} sides of the diagram and the Lagrangian 2-spheres lying over {\it red} cuts, both according to the constructions in \cite{SymingtonLeung}, and see also \cite{Vianna1}. By definition, the symplectic domain associated to an almost toric base in Figure \ref{fig:IntroTable} endowed with the data of the blue sides and red cuts is the symplectic complement of the  corresponding  symplectic divisors and Lagrangian 2-spheres inside the closed symplectic 4-manifold $(X_B,\omega_B)$. See Section \ref{sec:Preliminaries} for definitions and further details.

Let $\SH$ be the set of the nine symplectic domains listed in Figure \ref{fig:IntroTable}, as described in the above paragraph. These domains $(X,\omega_X)\in\SH$ are obtained from one of the monotone closed symplectic 4-manifolds $\C\P^2,\C\P^1\times\C\P^1,Bl_3(\C\P^2)$ and $Bl_4(\C\P^2)$ by removing a configuration of surfaces. These surfaces depend on each base in Figure \ref{fig:IntroTable}: each of these configurations consists of a union of symplectic 2-spheres, lying above the {\it blue sides} of the almost toric base, and Lagrangian 2-spheres, located above the {\it red cuts}. See Section \ref{sec:proofmain} for a more detailed description of these symplectic domains.

\begin{example} The closed symplectic 4-manifold associated to the base in the leftmost diagram in the first row is $(\CP^2,\omega_{\st})$, and the symplectic surface associated to the blue side is a symplectic line $\CP^1\sse(\CP^2,\omega_{\st})$. Thus, the associated symplectic domain is the unit Darboux 4-ball $(\bD^4(1),\omega_{\st})$, presented as the complement $\CP^2\setminus\CP^1$. In contrast, the closed symplectic 4-manifold associated to the base in the rightmost diagram in the first row is $(\CP^1\times\CP^1,\omega_{\st}\oplus\omega_{\st})$. The symplectic surface associated to the blue hypotenuse is the symplectic diagonal $\CP^1\sse(\CP^1\times\CP^1,\omega_{\st}\oplus\omega_{\st})$ and the Lagrangian 2-sphere associated to the red cut is the Lagrangian anti-diagonal. The associated symplectic domain is the $E(1,2)$ ellipsoid. In general, the symplectic domains $\SH$ do not have a particular name, except for the cases of $\bD^4$ and $E(1,2)$ just discussed, and the polydisk $ \bD^2(1)\times\bD^2(1)$ and the ellipsoid $E(2,3)$, which correspond to the center base diagram in the first row and the rightmost diagram in the second row, respectively.\hfill$\Box$
\end{example}


First, let us state our main result on the construction of ellipsoid embeddings:
\begin{thm}\label{thm:main}
	Let $(X,\omega_X)\in\SH$ be a 4-dimensional symplectic domain. Then the non-decreasing function $c_{X}:\R\lr\R$ admits a sequence of sharp points $S\sse I_X$.\hfill$\Box$
\end{thm}

Theorem \ref{thm:main} establishes the existence of {\it sharp} points, which is the constructive ingredient towards infinite staircases. Indeed, the existence of sharp infinite staircases consists of two independent arguments. First, a {\it constructive} result showing the existence of a sequence of {\it sharp} points $S\sse I_X$. Second, an {\it obstructive} result stating that the volume bound is {\it not} an equality for the non-sharp points $I_X\setminus S$. Our Theorem \ref{thm:main} contributes to the constructive part, which is the part that the present manuscript geometrically establishes.

{\bf Important point:} It is crucial to emphasize that, in Theorem \ref{thm:main}, we are able to embed an infinite staircase of ellipsoids into the symplectic domains associated to Figure \ref{fig:IntroTable}, and not just into the closed (almost toric) symplectic 4-manifolds: we must show that the ellipsoid embeddings can be made to avoid the required configurations of symplectic 2-spheres and Lagrangian 2-spheres, and this is achieved explicitly by using symplectic-tropical curves, see Sections \ref{sec:proofmain} and \ref{sec:Symp_Trop}.

\begin{remark} \label{rmk:Volume_filling}
  
 Consider a 4-dimensional symplectic domain $(X,\omega_X)\in\SH$, and a correponding ellipsoid $E$ associated
 with a sharp point in $S\sse I_X$, from Theorem \ref{thm:main}. Let $E_k \sse E$ be the rescaling of $E$
 such that the volume of $E \setminus E_k$ is $1/k$, $k\in\N$. In this manuscript, we construct a family of almost toric fibrations
 $\m^k:(\overline{X},\omega_{\overline{X}}) \lr \R^2$ (see Sections
 \ref{subsec:ATFs}--\ref{ssec:PolytopeMutation}) on the compactification $\overline{X}$ of $X$, which will be either $\C\P^2,
 \C\P^1\times \C\P^1, Bl_3(\C\P^2)$ or $Bl_4(\C\P^2)$. This provides a sequence of embeddings
 $\Psi_k:E_k \hookrightarrow \overline{X}$. Since the nodal slide operation in almost toric
 fibrations (see \cite{SymingtonLeung,Symington}) can be taken to only modify the fibration in a neighbourhood of the cut where the slide happens, we can obtain a family of symplectomorphisms $\Phi_k: \overline{X} \lr
 \overline{X}$ such that $\Phi_l$ leaves $\Psi_k(E_k)$ invariant and $\Psi_k(E_k) \sse
 \Psi_l(E_l)$, if $k \le l$. This construction will yield an embedding $E \hookrightarrow \overline{X}$ as follows. From the almost toric fibrations
 $\m^k:(\overline{X},\omega_{\overline{X}}) \lr \R^2$, we get a sequence of embeddings $\varphi_k: X
 \hookrightarrow \oX$, by identifying $X$ with the complement of a configuration $\oX \setminus \varphi_k(X)$
 of symplectic-tropical curves and, in some cases, Lagrangian spheres. Taking the symplectomorphisms $\Phi_k: \overline{X} \lr
 \overline{X}$ into consideration, we get a family of symplectomorphisms $\phi_k: X \lr X$ such that $\Phi_k\circ\varphi_k=\varphi_{k+1}\circ\phi_k$. 
In the cases being discussed, the isotopy taking $\Phi_k (\oX \setminus \varphi_k(X))$ to $\oX \setminus \varphi_{k+1}(X)$ 
 can be taken in the complement of $\Phi_k(\Psi_k(E_k)) = \Psi_k(E_k)$. Hence, we may assume that $\phi_k$ is 
 identity in the complement of $\Psi_k(E_k)$, which gives the inclusion $\varphi_k^{-1}(\Psi_k(E_k))=\varphi_l^{-1}(\Psi_k(E_k)) \sse
 \varphi_l^{-1}(\Psi_l(E_l)) \sse X$, for all $k\leq l$, defining an embedding $E \hookrightarrow 
 X$.\hfill$\Box$

\end{remark}

\color{black}

\begin{remark} Even if we find the obstructive part equally interesting, our geometric argument is only constructive, not obstructive. The manuscript \cite{DCGetAL} uses the symplectic capacities from Embedded Contact Homology (ECH), as developed by M. Hutchings \cite{Hutchings1,Hutchings2}, to provide the desired obstructions. The combination of our Theorem \ref{thm:main}, being constructive, and the ECH obstructions in \cite{DCGetAL} imply that the sequence of sharp points $S\sse I_X$ are indeed part of an infinite staircase. We refer to \cite{DCGetAL} for a detailed discussion and computation of these ECH capacities.\hfill$\Box$
\end{remark}

\begin{example}\label{ex:FirstRow}
The three domains $(X,\omega_X)\in\SH$ in the first row of Figure \ref{fig:IntroTable} are the unit ball $\bD^4(1)$, presented as the complement of a symplectic sphere $\C\P^1\sse(\C\P^2,\omega_\st)$, the polydisk $\bD^2(1)\times\bD^2(1)$, arising as the complement of the two symplectic spheres $\C\P^1\times\{pt\},\{pt\}\times\C\P^1\sse(\C\P^1\times\C\P^1,\omega_\st\oplus\omega_\st)$, and $E(1,2)$, presented as the complement of a symplectic sphere $\C\P^1\sse(\C\P^1\times\C\P^1,\omega_\st\oplus\omega_\st)$ and a Lagrangian 2-sphere $S^2\sse(\C\P^1\times\C\P^1,\omega_\st\oplus\omega_\st)$ in the homology class of the anti-diagonal. Theorem \ref{thm:main} for these three domains recovers the Fibonacci stairs \cite{McDuffSchlenk1} and the Frenkel-M\"uller Pell stairs \cite{PellStaircase}. The Cristofaro-Gardiner-Kleinman staircase \cite{DCGetAL} for $E(2,3)$ correspond to the rightmost almost toric base in the second row.\hfill$\Box$
\end{example}

The symplectic domains $(X,\omega_X)\in\SH$ not discussed in Example
\ref{ex:FirstRow} do not have particular names, with the exception of the
rightmost domain in the second row. For instance, the leftmost domain
$(X,\omega_X)\in\SH$ in the second row is the complement in $Bl_3(\C\P^2)$,
$\C\P^2$ blown-up at three generic points $p_1,p_2,p_3\in\C\P^2$, of a
configuration of four symplectic 2-spheres: two of the exceptional divisors and
the proper transforms of the projective line through $p_1,p_2$, and the
projective line through $p_2,p_3$. These complements do not typically have a given name, except for $E(2,3)$, which appears as the
complement of an exceptional divisor and {\it three} Lagrangian 2-spheres in
$Bl_3(\C\P^2)$.

\begin{center} \begin{figure}[h!] \centering
\includegraphics[scale=0.75]{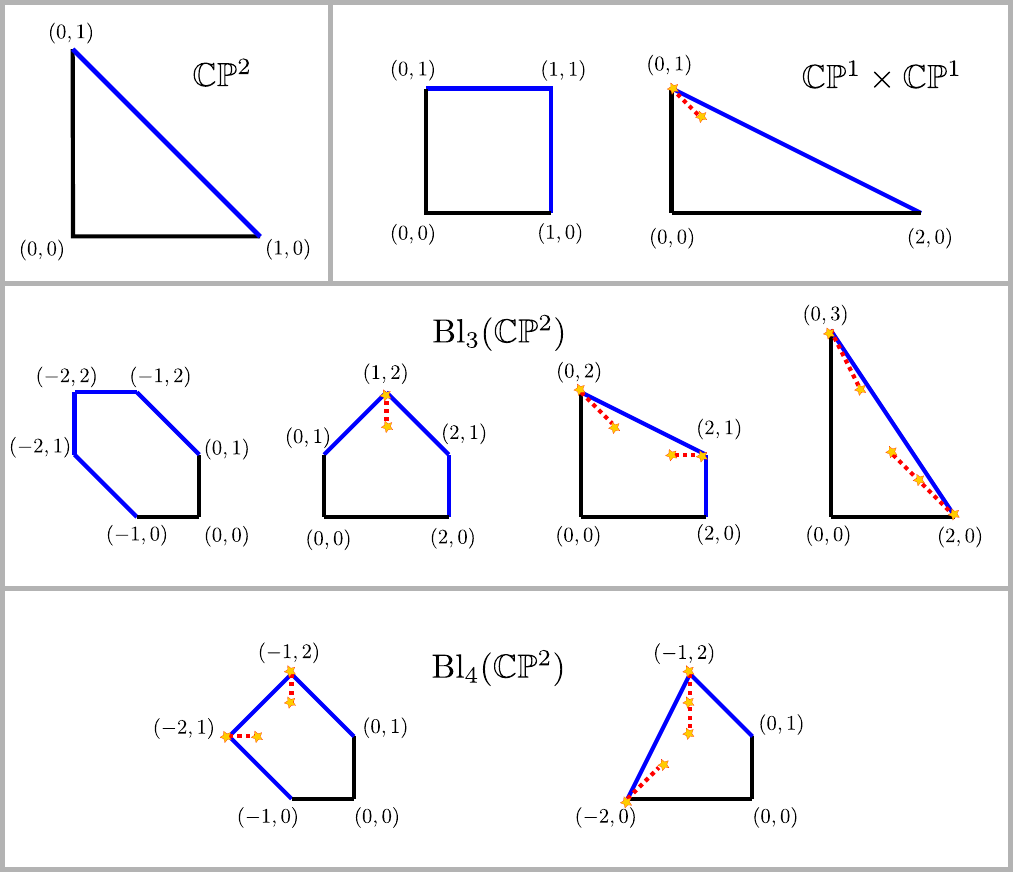}
\caption{The almost toric base
diagrams for the symplectic domains $(X,\omega_X)$ in Theorem \ref{thm:main}.
The domains $(X,\omega_X)$ are the complements of symplectic 2-spheres, in blue,
and Lagrangian 2-spheres, in red, in the closed almost toric symplectic manifold
corresponding to each diagram. Note that the same closed symplectic manifold
might yield different symplectic domains $(X,\omega_X)$, as the choices of
symplectic and Lagrangian 2-spheres depend on the almost toric base diagram. In each case, the polytope associated to the almost-base is given by the convex hull of the indicated vertices in the standard lattice $\Z\langle(1,0),(0,1)\rangle\cong\Z^2\sse\R^2$.}
\label{fig:IntroTable}
\end{figure} \end{center}

Second, our method for showing the existence of sharp
points requires applying symplectic techniques coming from {\it tropical} combinatorics in
almost toric diagrams \cite{SymingtonLeung,Symington}. In particular, our method -- including the proof for Theorem \ref{thm:main} -- requires the introduction and study of {\it symplectic-tropical curves}
in almost toric fibrations, which represent (configurations of) smooth symplectic curves in
4-dimensional almost toric symplectic manifolds. In a nutshell, our construction
in Section \ref{sec:Symp_Trop} yields the following result:

\begin{thm}\label{thm:main2}
Let $\pi:(X,\omega_{X})\lr B$ be an almost toric fibration and $\SC\sse B$ a symplectic-tropical curve.\footnote{See Definition \ref{dfn:Symp_comp_graph} for a precise description of symplectic-tropical curves. They are a generalization of the tropical diagrams in \cite{Mikhalkin1,Mikhalkin2} for symplectic surfaces and Lagrangian fibrations with singular nodal fibers.} Then there exists a symplectic curve $C\sse X$ with $\pi(C)=\SC$.
\end{thm}

The statement of Theorem \ref{thm:main2}, as well as the methods we introduce for its proof, are hopefully of interest on their own terms, as they extend the algebraic geometric tropical methods \cite{Mikhalkin1,Mikhalkin2} to the almost toric symplectic context. Symplectic tropical curves will be defined in Section \ref{sec:Symp_Trop}, where we shall develop the diagrammatics, arithmetic and symplectic geometry associated to sympletic-tropical diagrams in almost toric polytopes.

\begin{remark}
The tropical diagrams $\SC\sse B$ in Theorem \ref{thm:main2} also allow us to readily compute homological (and geometric) intersection numbers. The referee kindly informed us that homological intersection numbers can be defined in the general setting of tropical
cycles \cite{BIMS14} and our symplectic-tropical diagrams can be treated, away from the nodes, as a special kind of (1,1)-tropical cycles.\hfill$\Box$
\end{remark} 

The construction of the infinite staircases in Theorem \ref{thm:main} relies on our study of the symplectic isotopy classes of the symplectic curves $C_X\sse(\ol X,\omega_{\ol X})$ associated to specific tropical diagrams $\SC$. Theorem \ref{thm:main2} further develops the work of M. Symington \cite{Symington} in line with G. Mikhalkin's study of complex tropical geometry \cite{Mikhalkin1}. Section \ref{sec:Symp_Trop} contains a detailed study of the required local models, in Subsections \ref{ssec:Symp_Trop}-\ref{subsec:RequiredSTC}, as well as the construction of symplectic chains of embedded curves associated to symplectic-tropical diagrams, in Subsections \ref{subsec:Chains_STC_Prelim} through \ref{subsec:Chains_STC}.

Finally, the combinatorics and numerics appearing in the infinite staircases from Theorem \ref{thm:main} strongly intertwine with the recent developments in the study of cluster algebras \cite{Cluster2,Cluster1} and quiver mutations \cite{PolyMut2,PolyMut3}.

{\bf Organization.} The article is organized as follows. Section \ref{sec:Preliminaries} introduces basic ingredients of almost toric symplectic geometry. Section \ref{sec:proofmain} proves Theorem \ref{thm:main} by constructing the required infinite staircases with the results from Section \ref{sec:Symp_Trop}. Section \ref{sec:Symp_Trop} proves Theorem \ref{thm:main2}.
\hfill$\Box$\\

{\bf Acknowledgements.} We are grateful to Dan Cristofaro-Gardiner, Richard Hind, Liana Heuberger, Jeff Hicks, Tara Holm, Alessia
Mandini, Dusa McDuff, Navid Nabijou, Ana Rita Pires, Laura Starkston and Weiwei Wu for valuable discussions.

R.~Casals is supported by
the NSF grant DMS-1841913, the NSF CAREER Award DMS-1942363, an Alfred P. Sloan Fellowship and a BBVA Research Fellowship. R.~Vianna's participation at the Matrix program on
the Mirror Symmetry and Tropical Geometry at Creswick, Australia was significant
for the development of technical aspects of symplectic-tropical curves in
Section \ref{sec:Symp_Trop}. R.~Vianna is supported by Brazil's National Council
of scientific and technological development CNPq, via the research fellowships
405379/2018-8 and 306439/2018-2, and by the Serrapilheira Institute grant
Serra-R-1811-25965.\hfill$\Box$\\ \\
{\bf Relation to \cite{DCGetAL}.} This article has been posted in parallel with the
manuscript \cite{DCGetAL}. We are grateful to each of its authors, Dan
Cristofaro-Gardiner, Tara Holm, Alessia Mandini and Ana Rita Pires, for the
fluid and helpful communication with us. It is our understanding that both
groups of authors came to the study of this problem from different perspectives,
with the idea of using polytope mutations originating with the first author of
the present manuscript. Both collaborations have benefited from our exchanges of ideas.

We encourage the reader to study the manuscript \cite{DCGetAL}, which we find to be a very valuable contribution to the theory of symplectic ellipsoid embeddings as well. The results in our article, especially Theorem \ref{thm:main} are strengthened by their contributions to the {\it obstructive} side of the theory, as their manuscript \cite{DCGetAL} shows that ECH obstructions make the volume bound not optimal away from the required sequences. This clearly highlights the importance of the full ellipsoid embeddings we construct, and we gladly acknowledge the relevance and non-triviality of these ECH computations. Their manuscript \cite{DCGetAL} also proves Theorem \ref{thm:main}, equally based on polytope mutations but concluding in a more succinct abstract manner, and formulates a compelling conjecture regarding infinite staircases of rational convex toric domains. Each of the articles addresses the arithmetic of staircases from a different perspective: in this manuscript we directly use the implicit Diophantine equations, and the manuscript \cite{DCGetAL} proceeds parametrically, in terms of recursions. Even though both manuscripts
could potentially be joined, we find that having both articles available is
also enriching for the literature, as they discuss different techniques and
perspectives.\hfill$\Box$\\



\section{Preliminaries}\label{sec:Preliminaries}


In this section we develop notations for the base diagrams of almost toric fibrations, also known as ATFs \cite{Symington}, which we use in our description of symplectic ellipsoid embbedings. The present section is focused on understanding combinatorial mutations of polytopes that describe these almost toric fibrations -- particularly from the viewpoint of smoothing, and degenerating, toric orbifolds. The more technical aspects of the symplectic topology shall be presented in Subsection \ref{subsec:ATFs}.


\subsection{Almost-toric Fibrations}\label{subsec:ATFs}

We start by succinctly discussing \emph{almost toric fibrations},
henceforth abbreviated ATF \cite{Symington}. Almost-toric fibrations are generalizations of the toric ones, where we allow so-called nodal singularities, besides the usual toric 
singularities of a moment map. We refer the reader to \cite{SymingtonLeung,Symington} for additional details. The necessary definition, \cite[Definition~4.5]{Symington} and \cite[Definition~2.2]{SymingtonLeung}, reads as follows:

\begin{definition}[\cite{SymingtonLeung,Symington}] \label{dfn:ATF} 
 An \emph{almost toric fibration} of a symplectic 4-manifold $(X,\omega)$ is 
 a Lagrangian fibration $\pi: (X, \omega) \rightarrow B$ such that any point of 
 $(X, \omega)$ there exists a Darboux neighborhood $(x_1,y_1,x_2,y_2)\in(\bD^4,\omega_\st)$, with symplectic form $\omega_\st=dx_1\wedge dy_1 +
 dx_2\wedge dy_2$, in which the map $\pi$ has one of the following local normal forms:
\begin{eqnarray*}
 \pi(x,y) & = & (x_1, x_2),
\hspace{5,7cm} \text{regular point}, \\
 \pi(x,y) & = & (x_1, x_2^2 + y_2^2),
\hspace{4,85cm} \text{elliptic, corank one}, \\ 
 \pi(x,y) & = & (x_1^2 + y_1^2, x_2^2 + y_2^2),
\hspace{4cm} \text{ elliptic, corank two}, \\
 \pi(x,y) & = & (x_1y_1 + x_2y_2, x_1y_2 - x_2y_1),
\hspace{2,7cm} \text{   nodal or focus-focus},  
\end{eqnarray*}
with respect to some choice of coordinates near the image point in the almost toric base $B$.\hfill$\Box$
\end{definition}

For the nodal singularity, writing $x = x_1 + i x_2$ and $y = y_1 + i y_2$, 
the almost toric fibration $\pi$ reads $\pi(x,y)=\overline{x} y \in \C$. The following remark will
be relevant to Section~\ref{sec:Symp_Trop}.

\begin{remark} \label{rmk:Symp_Curves} By the Arnold-Liouville Theorem \cite{Ar_book}, given any point $q$ in a 
regular fibre, there is a neighborhood of the form $U \times T^2$, with action-angle coordinates $(p_1,p_2, \theta_1, \theta_2)$, 
$(p_1,p_2) \in \R^2$ and $(\theta_1, \theta_2) \in T^2 = \R^2/\Z^2$, 
where the symplectic for $\omega$ is $dp_1\wedge d\theta_1 + dp_2\wedge d\theta_2$. Here, 
we identify simultaneously $U$ with a neighborhood of $\pi(q)$ in $B$, as
well as, a neighborhood in $\R^2$. Consider then a 1-cycle $\gamma(t) = (\theta_1(t), \theta_2(t))
\in \pi^{-1}(q) = \R^2/\Z^2$ 
and move it along a curve $\sigma(s)$ in 
$U \subset \R^2$ so as to get the cylinder $\Gamma(s,t) = (\sigma(s), \gamma(t))$ 
in the $(p_1,p_2, \theta_1, \theta_2)$ coordinates. Up to a choice of orientation, $\Gamma(s,t)$ is symplectic iff 
$\langle {\sigma'}(s)| \gamma'(t) \rangle \ne 
0$, for all $s$ and $t$. In particular, if we fix the cycle with $\gamma'(t) =(a,b)$, $(a,b) \in \Z^2$, we are allowed to move in any 
direction in $U$, as long as we are never parallel to $(b, -a)$, in order 
to get a symplectic cylinder.\hfill$\Box$   
\end{remark}

The set $B_0$ of regular values of $\pi: X^{4} \to B^2$ carries naturally an
affine structure, and circling around a node, i.e. the image of a nodal critical point,
provides a monodromy for this affine structure, that is a shear with respect to
some eigendirection associated to the node -- see \cite{SymingtonLeung,Symington} for
details. In local action-angles coordinates (see Remark \ref{rmk:Symp_Curves}),
where we have $\pi_{|V} : V \cong U \times T^2 \to U \subset \R^2$, with $U
\hookrightarrow B$, this affine structure is identified with the standard lattice $\Z^2
\subset \R^2 \cong T_b^*U$ for each $b \in U$. In the sense of \cite[Definition
1.24]{Gross}, $B$ is viewed as an integral affine surface with singularities.
 
If $B$ is topologically a disk, an ATF can be described by an almost toric base diagram (ATBD)
\cite[Section~5.2]{Symington}. We choose a set of cuts $\mC$ consisting of rays on $B_0$ from the nodes to
$\del B$, so we don't have monodromy on $B_0 \setminus \mC$ and hence we can get an affine embedding of $B_0
\setminus \mC$ into $\R^2$. The closure of the image is the ATBD. The monodromy around the nodes informs us
how to glue the limits as we approach the cuts from both sides (for more details see \cite[Figure~7 and
Section~5.2]{Symington}). In particular, if we choose a cut associated to a node to be a ray in the direction of an
eigenvector of the monodromy associated to that node, as we approach a point in the cut from both sides, the
limit of the images under the embedding $B_0 \setminus \mC \to \R^2$ agree. In this paper we always assume that all
the cuts satisfy this condition and $B$ is topologically a disk. Hence, we can extend the affine embedding
$B_0 \setminus \mC \to \R^2$ to a continuous map $B \to \R^2$ and the ATBD is depicted as a polytope $P$, with
nodes (represented by $\times$ or $\star$) in the interior representing the nodal fibres, and cuts
(represented by dashed segments) towards the edge, that encode the monodromy around the singular fibres.

\begin{example} Figure \ref{fig:CP2} is an ATBD, with underlying polytope $P$, representing
an ATF of $(\CP^2,\omega_\st)$, see also Figure \ref{fig:FirstMutationP2} and Figure \ref{fig:IntroTable}, and \cite{Vianna3} for several ATBDs representing ATFs of
Del Pezzo surfaces. To each cut, we associate a monodromy
matrix in $SL(2, \Z)$, with eigenvector in the direction of the cut, representing, in the standard basis of $\R^2$, the
monodromy associated to a counter-clockwise loop around the node.\hfill$\Box$
\end{example} 
 
\begin{figure}[h!]   
  \begin{center} 
 \centerline{\includegraphics[scale=0.5]{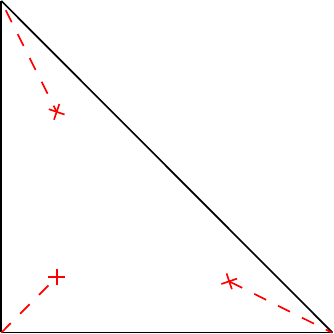}}

\caption{ATBD for $\CP^2$.}
\label{fig:CP2} 
\end{center} 
\end{figure}

\begin{remark} In an ATBD with underlying polytope $P$, if a vertex has no cut arriving into it, then it
represents a corank two elliptic singularity. Its pre-image is a point and a neighbourhood of the vertex
describes a toric structure on its pre-image, hence we call it a \textit{toric vertex}. Otherwise the vertex is inside
the cut and its pre-image is a circle of corank one elliptic singularities.\hfill$\Box$ \end{remark}

%

\subsection{Almost-Toric Base}\label{ssec:SymplecticAlmostToric}

In this section we develop a notation for base diagrams that is suitable for almost toric mutations; it is more in line
with the notions of mutations \cite{GalkinUsnich10,PolyMut1,PolyMut2} in algebraic geometry that arose 
from the pioneering work \cite{GalkinUsnich10}. 


\color{black}

\begin{definition}\label{def:ATBD}
An almost toric base content $\SP=(P,\SB)$ is a pair where 
$$P=\convex(v_1,\ldots,v_{|V(P)|})\sse\R^2$$
is the polytope defined by the convex hull of counter-clockwise cyclically ordered $|V(P)|$ vertices $v_i\in\R^2$, 
$1\leq i\leq |V(P)|$ and $\SB$ is the cut content of $\SP$. The cut content $\SB$ is defined as the union of pairs 
$B_i$, $1\leq i\leq |V(P)|$, where $B_i=(c_i,n_i)\in\R^2\times\N$ consists of a primitive vector $c_i\in \Z^2 \subset \R^2$,
pointing inside $P$ from $v_i$, 
and a positive number $n_i\in\N$, which accounts for the total number of nodes along $c_i$. Let us denote by $(v)_p$ the primitive vector pointing in the direction of $v$. The following consistency condition 
must be satisfied:

\begin{equation} \label{eq:Cut_content}
   (v_i - v_{i-1})_p = M_i^{n_i}(v_{i+1} - v_{i})_p,
\end{equation}

\color{black}
where we take the indices $i \in \Z/|V(P)|\Z $ and $M_i\in GL(2,\Z)$ denotes the shear in $\R^2$ with respect to 
$c_i$, given by $M_i v := v + \det[c_i,v] c_i$. For a matrix form of $M_i$ in terms of $c_i = (a_i,b_i)$ see \cite[(4.11)]{Symington}.

By definition, an almost toric base is an almost toric base content $\SP=(P,\SB)$ , together 
with choices of $r_{i,j} \in \R^{\geq0}$, for $j = 1 , \dots, n_i$, $r_{i,j} < r_{i,j-1}$, so that the cuts 
$v_i + tc_i$, $t\in [0,r_{i,1}]$ are disjoint and inside $P$. The $n_i$ nodes associated with $v_i$ are said to be positioned at $v_i + r_{i,j}c_i$.\hfill$\Box$
\end{definition}

Definition \ref{def:ATBD} is suitable for almost toric fibrations \cite{Symington,SymingtonLeung} whose base
is topologically a disk, as shall be the case in this manuscript. In the case that $(X,\omega)$ is a Del Pezzo
surface, which symplectically means that the symplectic structure is monotone, and all the cuts point towards the monotone fibre that, up to translation, 
we assume centered at $0$. In that case, we have that the cut at the vertex $v_i$ is $c_i = - (v_i)_p$ and the number $n_i\in\N$ is the number of singular nodes along the cut at $v_i$,
placed in the positions $v_i + r_{i,j}c_i$, for $j = 1 ,\dots, n_i$, with $r_{i,1} < |v_i|$. \color{black}  For instance, if $v_i$ is a toric vertex
then $B_i=(c_i,1)$ and $r_{i,1} = 0$. We have chosen the notation {\it cut content} in line with the {\it
singular content} as defined in \cite{PolyMut3}. From now onwards, we informally refer to $\SP=(P,\SB)$ as an
almost toric base, referring to an almost toric base with almost toric base content $\SP=(P,\SB)$. Given an
almost toric base $\SP=(P,\SB)$, we denote by $X(\SP)$ the almost toric symplectic 4-manifold associated to
the almost toric base diagram defined by the polytope $P$ and cut content $\SB$, as constructed in
\cite{SymingtonLeung} (cf.~ also \cite[Section 5]{Symington}, \cite[Section 3]{Vianna1}).

\begin{remark} \label{rmk:ToricATF}

If $P$ is a Delzant polytope, it represents a smooth toric symplectic manifold $X(P)$. In this case its cut content 
satisfies $n_i = 1$ and $c_i = (v_{i+1} - v_{i}) - (v_i - v_{i-1})$ (the Delzant condition implies that in the 
counter-clockwise orientation, $\det[c_i, (v_{i+1} - v_{i})_p] = -1$). Moreover, its toric base satisfies $r_{i,1} = 
0$, for all $1\leq i\leq |V(P)|$. We call that the $0$-content (or empty content) associated to the Delzant polytope $P$. 

It $P$ is not Delzant, than it represents a singular toric symplectic manifold $X(P)$. If all singularities
are $T$-singularities, in the sense of \cite{AkKa14}, then each singularity is of the form
$$\frac{1}{n_il_i^2}(1, n_il_ik_i - 1),$$
following the notation in \cite{AkKa14}, for some $n_i,l_i,k_i\in\N$. If we identify $(v_{i-1} -
v_{i})_p$ with $(0,1)$, and $(v_{i+1} -
v_{i})_p$ with $(n_il_i^2, n_il_ik_i - 1)$, then $c_i = (l_i,k_i)$. In consequence, the base content $(P,\SB)$ is extracted from the singularity types and the multiplicity
$n_i$ is such that equation \eqref{eq:Cut_content} is satisfied. Also, the weight $n_il_i^2$ of the singularity satisfies
$l_i = |\det[c_i, (v_{i+1} - v_{i})_p]|$.  
\color{black} In this case, $X(P)$ is smoothable and the smoothing has an
almost toric fibration with base content $(P,\SB)$, hence we denote the smoothing by $X(P,\SB)$.\hfill$\Box$

\end{remark}

  \begin{remark} \label{rmk:nodal_mod} 
    Symington defined modifications of a base diagram \cite[Section~6.1]{Symington} called \emph{nodal trade} and \emph{nodal slide} and 
    proved that they provide diagrams for another ATF of the same symplectic manifold. 
    Modifying the ATF by changing $r_{i,n_i} = 0$ to $r_{i,n_i} > 0$ corresponds to a 
   \emph{nodal trade} \cite[Section~6.1]{Symington}. Modifying the ATF by changing the value of $r_{i,j} > 0$ to another positive value 
  corresponds to a \emph{nodal slide} \cite[Section~6.1]{Symington}. A further explanation is provided in Remark \ref{rmk:tradeslide}, and we encourage the reader to study \cite{Symington} for more details. \hfill$\Box$
\end{remark}

\begin{remark} \label{rmk:Toric_vertex_Fig1}
  In some diagrams of Figure \ref{fig:IntroTable}, there are ``nodes'' with $r_{i,n_i} = 0$. 
  Though the vertex of the polytope is not Delzant if $n_i > 1$, when $r_{i,n_i} = 0$, over the vertex lies a corank two
  elliptic singularity and a neighbourhood of the vertex, \emph{excluding the other associated nodes}, is indeed toric.
  Though it represents a corank two elliptic singularity, we still depict it with $\times$ or $\star$ over the vertex $v_i$, 
  if $n_i > 1$. We can only make $r_{i,n_i} = 0$ and remain almost toric if the corresponding singularity for 
  the toric orbifold $X(P)$ (as we disregard the nodes and cuts) is a Du Val singularity of $A_{n_i - 1}$ type.
  Modifying $r_{i,n_i} > 0$ to $r_{i,n_i} = 0$ in this case can be regarded as the inverse of a nodal trade. \hfill$\Box$
  
  \end{remark}

\begin{center}
	\begin{figure}[h!]
		\centering
		\includegraphics[scale=0.75]{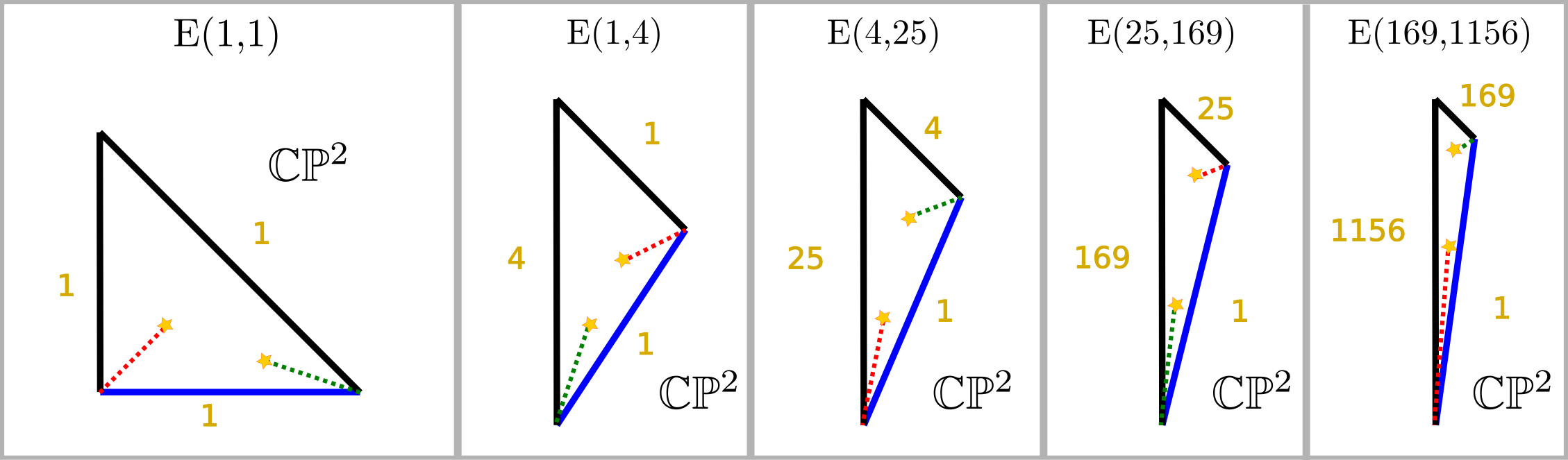}
		\caption{Almost-toric bases for rescalings of $(\C\P^2,\omega_\st)$.}
		\label{fig:FirstMutationP2}
	\end{figure}
\end{center}


\color{black}
\begin{example}\label{ex:FirstMutationP2}[\cite{Manetti91}, see also \cite{AG1}] Consider the five
almost toric bases $\SP_i=(P_i,\SB_i)$, $1\leq i\leq5$, depicted in Figure \ref{fig:FirstMutationP2},
increasingly ordered left to right. In this case, the five symplectic almost toric varieties $X(\SP)$ are
smooth 4-manifolds and, in fact, rescalings of $(\C\P^2,\omega_\st)$. The toric varieties
$X(P_i)$ are singular for $2\leq i\leq5$, respectively corresponding to the weighted
projective spaces $\C\P^2(1,1,4)$, $\C\P^2(1,4,25)$, $\C\P^2(1,25,169)$ and $\C\P^2(1,169,1156)$. These
correspond to the fact that $\C\P^2$ is a smoothing of each of these weighted projective spaces.\hfill$\Box$
\end{example}

Let us now define {\it mutation} of an almost toric base.\footnote{The mutation of an almost toric base is a different 
almost toric base with the property that, under the construction of \cite{SymingtonLeung} (or see Subsection \ref{ssec:PolytopeMutation}), the symplectic 4-manifolds associated to both bases will coincide -- for the basics on cut transfers, see \cite{Vianna2, Vianna3}.}

\begin{definition} \label{dfn:mutation}
 An almost toric base content $\mu_{v_i}^k \SP=(\mu_{v_i}^kP, \mu_{v_i}^k\SB)$ 
  is said to be \emph{mutated} from the almost toric base content $\SP=(P,\SB)$ 
  at the vertex $v_i$ with order $k \in \{1,\dots,n_i\}$ if it satisfies the 
  following:
  
  \begin{itemize}
    \item[-] Consider the ray $\gamma_i$ that leaves $v_i$ in the 
  direction of $c_i$, and $\tilde{v}$ the intersection with $\del P$. The 
  polytope is divided by $\gamma_i$ into $P = P^{(1)} \cup P^{(2)}$, with 
  $v_{i-1} \in P^{(1)}$ and $v_{i+1} \in P^{(2)}$. Then 
  $\mu_{v_i}^kP = P^{(1)} \cup M_i^{k}(P^{(2)})$ or $\mu_{v_i}^kP = M_i^{-k}(P^{(1)}) \cup P^{(2)}$. 
    
    \item[-]  The consistency condition for $\mu_{v_i}^k \SP$ implies that, 
  if $k < n_i$, $v_i$ is a vertex of $\mu_{v_i}^k \SP$, with
  content $(c_i, n_i - k)$. If $k = n_i$, $v_i$ is not
  a vertex of $\mu_{v_i}^k \SP$. If $\tilde{v} = v_j$ for some $j\ne i$, 
  we assume that $c_j = -c_i$, and the content of  
  $\tilde{v}$ is $(c_j, n_j + k)$, otherwise, the 
  content of $\tilde{v}$ is $(-c_i, k)$.
  \end{itemize}
  
  An almost toric base is said to be mutated from another at 
  the vertex $v_i$ with order $k \in \{0,\dots,n_i\}$, if their almost toric base contents are related as above and 
  the position of all the nodes in $\gamma_i$ are the same. A mutation of order is $1$ is said to be a \emph{single} mutation, and a {\it full} mutation will refer to a mutation of order $n_i$ at the vertex $v_i$.\hfill$\Box$
\end{definition}

%


\color{black}

The symplectic domains $(X,\omega_X)\in\SH$ featured in Theorem \ref{thm:main} are open subsets of the symplectic 4-manifolds $X(\SP)$ obtained by removing collections of symplectic and Lagrangian 2-spheres. (The collection of Lagrangian 2-spheres being removed might be empty in some cases.) Let us introduce a combinatorial structure to encode these surface configurations.

\begin{definition}
A relative almost toric base consists of a pair $(\SP,\SS)$ where $\SP=(P,\SB)$ is an almost toric base, and $\SS=S\cup \mathscr{L}$ is a set consisting of sides of the polytope $P$, forming a subset $S\sse\SS$, and segments $l\sse\SS$ within each cut, a segment being bounded by two consecutive nodes, forming a subset $\mathscr{L}\sse\SS$.\hfill$\Box$
\end{definition}

In the cases of Theorem \ref{thm:main}, the sides $S\sse\SS$ yield symplectic 2-spheres
\cite{SymingtonLeung,Symington} and the segments $\mathscr{L}\sse\SS$ give rise to Lagrangian 2-spheres in the
symplectic 4-manifold $X(\SP)$, as visible surfaces \cite[Section 7]{Symington}. Indeed, the cut
content $\SB$ of an almost toric base defines a sequence of Lagrangian 2-spheres $L^j(B_i)$, $1\leq j\leq
n_i - 1$, $1\leq i\leq |V(P)|$, whenever $n_i \ge 2$. \color{black} By definition, $L^j(B_i)$, $1\leq j\leq n_i - 1$, will be
the exact Lagrangian 2-sphere with matching cycle the segment in the direction $c_i\in\R^2$ from the $j$th
node to the $(j+1)$th node, where nodes are ordered increasingly from the vertex $v_i$ inwards.

Consider a relative almost toric base $(\SP,\SS)$. We denote by $D(\SS)\sse X(\SP)$ the configuration of symplectic divisors $D(S)$, associated to the pre-image of the sides of $S\sse\SS$ under the almost toric fibration, and by $L(\mathscr{L})$ the collection of Lagrangian spheres in the symplectic 4-manifold $X(\SP)$ associated to the segments in $\mathscr{L}\sse\SS$; cf.~ \cite{SymingtonLeung,Vianna3}. Note that for the same choice of sides $S$ in $P$, the topology of $D(S)$ typically depends on the cut content $\SB$ of $\SP$. We shall denote
$$X(\SP,\SS):=X(\SP)\setminus (D(S)\cup L(\SL)),$$ the symplectic complement of the symplectic divisors $D(S)$ and the Lagrangian spheres $L(\SL)$ in $X(\SP)$.

\begin{example} The unique cut in the rightmost relative almost toric base $(\SP,\SS)$ of the first row on Figure \ref{fig:IntroTable} yields a Lagrangian 2-sphere $S^2\sse X(\SP)\cong\C\P^1\times\C\P^1$. This is the unique Lagrangian sphere in $\C\P^1\times\C\P^1$, in the same Hamiltonian isotopy class as the anti-diagonal \cite{LagrangianSphere1}. The blue side $S$ in this almost toric base gives a symplectic divisor in the homology class of $D(S)=\C\P^1\times\{pt\}\sse X(\SP)$, and thus $X(\SP,S)\cong\C\P^1\times\bD^2(1)$. Finally, the complement of the anti-diagonal Lagrangian and the symplectic 2-sphere $D(S)$ is symplectomorphic to the ellipsoid $X(\SP,\SS)\cong E(1,2)$, as can be seen from the almost toric diagram. (E.g.~toric diagrams for ellipsoids $E(1,k)$ are provided in \cite[Section 2]{McDuff1} and then see \cite{Vianna3} for the almost toric modification where the anti-diagonal Lagrangian 2-sphere is isotoped to be above a vertex.)\hfill$\Box$
\end{example}


\subsection{Polytope Mutations in Almost-Toric Diagrams}\label{ssec:PolytopeMutation}
Let $(X,\omega_X)$ be a 4-dimensional closed {\it toric} symplectic manifold. The moment map $\m:X\lr\R^2$ associated to the Hamiltonian $T^2$-action yields a convex polytope $\m(X)=P_X\sse\R^2$. This polytope is the base $\SP=(P,\varnothing)$ of the Lagrangian toric fibration $\m:X\lr \m(X)$, which is a singular fibration at the boundary $\dd(P_X)=\m(X)\setminus\mbox{Int}(\m(X))$, with isotropic fibers. The standard affine structure $\Z^2\sse\R^2$ in the real plane endows the base of this fibration with an affine structure, which itself coincides with the affine structure associated to the Lagrangian fibration \cite{Ar_book}.

Recall from Definition \ref{dfn:ATF} that an {\it almost toric} fibration on a symplectic 4-manifold
$(X,\omega_X)$ contains the data of a smooth surface $B$ and a smooth map $\pi: X \to B$ with Lagrangian 
fibres (which here we assume compact), whose singular fibres can be toric 
(a corank 2 singularity or a circle of corank 1 singularities) or nodal. The regular locus $B_0$ is 
endowed with an integral affine structure. We assume $B$ is a disk, and a suitable choice
of cuts $\mC$ -- recall Section~\ref{subsec:ATFs} -- gives rise to a continuous map $B \to P_X \subset \R^2$, 
that is affine when restricted to $B_0 \setminus \mC$ (with respect to the standard affine structure $\Z^2 \subset 
\R^2$, generated by the unit vectors tangent to the coordinate axes \cite[Example 2.8]{Symington}). We encode the information of the nodes and cuts into an almost toric base $\SP=(P_X,\SB)$.
The composition $X \xrightarrow{\pi} B \to P_X$ can be regarded as an analogue of the moment map $\m:X\lr\R^2$ and
also yields a convex polytope $\m(X)=P_X\sse\R^2$, see \cite{Symington}. 

Let $(X,\omega_X)=X(P,\SB)$ be an almost toric symplectic variety with
almost toric base $(P,\SB)$, and $v\in P$ be a vertex of its polytope $P=\m(X)$, with content $(c_v,n_v)$. Let us construct a family of Lagrangian {\it almost toric} fibrations

$$\m^t:(X,\omega)\lr\R^2,\quad t\in[0,1],$$ with $\m^0=\m$ and the
image $\m^1(X)$ being the polytope $\mu_v^{1}(P_X)$ obtained by
combinatorially mutating $P_X$ at $v\in P_X$, as defined in \cite[Section
3]{PolyMut1}, which corresponds to an order one mutation in Definition
\ref{dfn:mutation} above. This shall be denoted $\m^1 = \mu_v^{1}(\m)$.

The crucial fact is that these are all fibrations of the
{\it same} symplectic variety $(X,\omega_X)$. It is only the presentation of
$(X,\omega_X)$ as an almost toric symplectic domain that varies. In a nutshell,
the geometric idea behind Theorem \ref{thm:main} is that different almost toric
fibrations for $(X,\omega_X)$ make different full embeddings of symplectic
ellipsoids in $(X,\omega_X)$ visible.
%

\begin{remark}\label{rmk:tradeslide} Following a suggestion by the referee, we recall the two notions of {\it nodal slide} and {\it nodal trade}, as introduced in \cite{Symington}, and refer also to \cite{SymingtonLeung,Vianna1,Vianna2,Vianna3} among others. Recall that two almost toric bases $(B,A_1,S_1)$ and $(B,A_2,S_2)$, $A_i$ integral affine structures on $B$, and $S_i$ the stratifications of the toric base (see \cite[Section 5]{Symington}), $i=1,2$, are said to be related by a nodal slide if there is a curve $\gamma$ in the base $B$ such that $(B\setminus\gamma,A_1,S_1)$ and $(B\setminus\gamma,A_2,S_2)$ are isomorphic and the curve $\gamma$ contains one node and belongs to the eigenline through that node.
	
For the definition of nodal trade, we follow \cite[Lemma 6.3]{Symington}. This result states that given an almost toric base $(B,A,S)$, $R$ an embedded eigenray that connects a node $s$ with a point $b$ in an edge $E$ of the reduced boundary of $B$ where there are no other nodes on $R$, and $v^*,w^*\in T^*_bB$ vanishing and collapsing covectors spanning the integral (dual) affine structure in $T^*_bB$, then there is an almost toric base $(B,A',S')$ such that

\begin{enumerate}
	\item $(B,A',S')$ contains one fewer node than $(B,A,S)$,
	\item $(B\setminus R,A',S')$ is isomorphic to $(B\setminus R,A,S)$,
	\item In $(B,A',S')$, the intersection of the ray $R$ and the reduced boundary is a vertex.
\end{enumerate}
By definition, two almost toric bases $(B,A,S)$ and $(B,A',S')$ are said to be related by a nodal trade if they satisfy these hypotheses.\hfill$\Box$
\end{remark}

\color{black}

Let us detail the steps producing the family of maps $\m^t$.


\begin{itemize}
	\item[1.] First, the introduction of a {\it nodal trade} at $v$ (if $n_v= 1$ and $r_{1,1} = 0$) and subsequent
	 {\it nodal slides} at the $n_v$-th node in direction $c_v$, as defined by M. Symington \cite{Symington}, produces 
	 a family of almost toric fibrations,	 described by $\m^t$, $t\in[0,1)$, 
	 with $\m^0=\m$, converging to ${\m'}^1$ as $t\to 1$, all with polytope $P_X$. 

	\item[2.] Second, apply a single mutation to the almost toric base for ${\m'}^1(X)$, 
	obtaining another map $\m^1:(X,\omega)\lr \R^2$, 
	by composing ${\m'}^1$ with the map $P_X \lr \mu_v^{1}(P_X)$ in Definition \ref{dfn:mutation}.
\end{itemize} 

We assume that we slide the $n_v$-th node long enough so that it is very close to
the (often new) vertex $\tilde{v}$ coming from mutation. In particular, in the case $X$ is Del Pezzo, we also assume that the node crossed $0 \in \R^2$, which represents the monotone 
fibre. For now, let us assume that $P_X$ is a triangle, which
will simplify our notation, and that we apply full mutations, so
we consider $\mu_v := \mu_v^1 \circ \cdots \circ \mu_v^1 = \mu_v^{n_v}$.

The process described in the steps above can be iterated for the initial symplectic domain $(X,\omega_X)$, as follows. Let $V=(v_n)_{n\in\N}$ be a sequence of vertices, with $v_1,v_2$ distinct, such that $$v_i\in\mu_{v_{i-1}}\mu_{v_{i-2}}\circ\ldots\circ\mu_{v_1}(P_X), \quad i\in\N,$$
and $v_{j}=\widetilde{v}_{j-2}$ for all $j\in\N$. In short, we choose two vertices $v_1,v_2\in P_X$ in a triangle and we first mutate at $v_1$ and then at $v_2$. Then we mutate at the new vertex $v_3=\widetilde{v}_1\in\mu_{v_2}\mu_{v_1}(P_X)$, which was first opposite to $v_1$ in $P_X$, and then at $v_4=\widetilde{v}_2\in\mu_{v_3}\mu_{v_2}\mu_{v_1}(P_X)$, which was initially opposite to $v_2$ in $\mu_{v_1}(P_X)$. Note that the triangle $P_X$ has a vertex $v_f\in P_X$ which remains fixed under these iterative mutations.

\begin{example}[\cite{GalkinUsnich10}]
Figure \ref{fig:FirstMutationP2} shows a sequence of rescaled almost toric fibrations $\{\m^t\}_{t\in[0,4]}$ at the values $t=0,1,2,3,4$, on the symplectic manifold $(\C\P^2,\omega_\st)$. First, there is a mutation from the toric moment polytope $P$ for $(\C\P^2,\omega_\st)$ to $\SP_1=(P_1,\SB_1)$, where $\mu^1(P)$ is the toric moment polytope for $\C\P^2(1,1,4)$. The second mutation moves from $\SP_1=(\mu^1(P),\SB_1)$ to $\SP_2=(\mu^2(P),\SB_2)$, where $\mu^2(P)$ is the toric moment polytope for $\C\P^2(1,4,25)$. The third mutation leads to the almost toric base with polyope $\mu^3(P)$, the toric moment polytope for $\C\P^2(1,25,169)$, and the fourth mutation arrives at $\mu^4(P)$, the toric moment polytope for $\C\P^2(1,169,1156)$. This procedure can be iterated indefinitely, yielding almost toric bases with their polytope $\mu^n(P)$ being the toric moment polytope for $\C\P^2(1,a_n,b_n)$, where $(a_n,b_n)$ will always be squares of consecutive odd Fibonacci numbers. A clear understanding of this sequence of almost toric mutations makes Theorem \ref{thm:main} for the unit 4-ball $\bD^4(1)$ much more intuitive. Note nevertheless that we require the symplectic ellipsoid embeddings to lie in $\bD^4(1)$ and not just a compactification.\hfill$\Box$
\end{example}


\section{Existence of Sharp Sequences of Ellipsoid Embeddings}\label{sec:proofmain}


The central idea that this manuscript introduces is the use of {\it almost toric mutations} in the study of symplectic ellipsoid embeddings. In this section we prove our main result Theorem \ref{thm:main}, assuming Theorem \ref{thm:ExistenceTropicalSymplectic}, whose proof will require the tropical symplectic techniques developed in Section \ref{sec:Symp_Trop}.

Let us start by explaining how to construct symplectic ellipsoid embeddings in an almost toric domain $X(\SP,\SS)$ with $(\SP,\SS)$ a relative almost toric base.

\begin{definition}\label{def:triangle}
	Let $P\sse\R^2$ be a convex polygon and $v_0,v_1,v_2\in V(P)$ three vertices of $P$. The {\it triangle} $\overline{T}_{v_0,v_1,v_2}\sse P$ is the convex hull of $v_0,v_1,v_2$. For $\varepsilon\in\R_{\geq0}$, the {\it open $\varepsilon$-triangle} $T^\varepsilon_{v_0,v_1,v_2}\sse P$ is the complement $\overline{T}_{v_0,v_1,v_2}\setminus t^\varepsilon$, where $t^\varepsilon\sse \overline{T}_{v_0,v_1,v_2}$ is an $\varepsilon$-neighborhood of the (open) side $\conv(v_0,v_2)$ connecting $v_0,v_2\in V(P)$.\hfill$\Box$
\end{definition}

Let $(X,\omega)=X(\SP,\SS)$ be symplectic almost toric manifold with polytope $P\sse\SP$, and assume that $v_1\in V(P)$ is a smooth toric vertex. Let $v_0,v_2$ be the two vertices in $P$ closest to $v_1$. If the open triangle $T^\varepsilon_{v_0,v_1,v_2}\sse P$ does not contain any critical values for the almost toric fibration $\m_X:X\lr P$, the pre-image of $T^\varepsilon_{v_0,v_1,v_2}$ under $\m$ yields a symplectic embedding
$$i_{v_1}:E(a,b)\lr (X,\omega),$$
where $a,b$ are the affine lengths of the two sides $\conv(v_0,v_1)\cap T^\varepsilon_{v_0,v_1,v_2}$, $\conv(v_1,v_2)\cap T^\varepsilon_{v_0,v_1,v_2}$ respectively. This construction is a potential method for construction symplectic ellipsoid embeddings, but it has the following two disadvantages.

First, for a given $\varepsilon\in\R^{>0}$, there are only finitely many open triangles in a moment polytope $P\sse\R^2$ for $(X,\omega)$ and thus, even if we considered the $\GL(2,\Z)$ action, this observation on its own is not sufficient to build an infinite sequence of ellipsoid embeddings. Second, unless $P$ is a triangle, i.e. $|V(P)|=3$, the symplectic ellipsoid embeddings of the form $i_{v_1}$ will not be volume-filling.

The first new geometric idea is that the symplectic manifold $(X,\omega)$ admits a sequence of {\it almost toric} fibrations $(\m_n)_{n\in\N}:X\lr\R^2$, as introduced in Section \ref{ssec:PolytopeMutation}, which themselves can be used to construct a sequence of symplectic ellipsoid embeddings. The images $\m_n(X)$ of these almost toric fibrations are also convex polytopes $P_n\sse\R^2$. In this almost toric case, the symplectic embeddings $i_{v_1}$ are built as above, where the main condition is that the open triangles $T^\varepsilon_{v_0,v_1,v_2}\sse P$ do not include any interior singular values of $\m_n$. (This condition is always satisfied in the toric case, as there are no interior singular values for the Lagrangian fibration.)

In Subsection \ref{ssec:SymingtonSeq}, we describe the properties for a sequence of almost toric fibrations $(\m_n)_{n\in\N}:X\lr\R^2$ such that each almost toric fibration admits a volume-filling open triangle $T^\varepsilon_{v_0,v_1,v_2}\sse P$, for arbitrarily small $\varepsilon\in\R^{>0}$. These open triangles give rise to symplectic embeddings
$$i^{(n)}_{v_1}:E(a_n,b_n)\lr (X,\omega).$$
The arithmetic of the sequence of pairs $(a_n,b_n)\in\N^2$ is governed by a diophantine equation which depends on the initial choice of $(\SP,\SS)$. These equations have featured prominently in birational geometry \cite{AG1,AG2,AG3} and the study of coherent sheaves \cite{Sheaf1,Sheaf2,Sheaf3}, and we present them systematically in Subsection \ref{ssec:ArithmeticSymington}.\\

The use of almost toric fibrations has many advantages, including the fact that we can build sequences of almost toric fibrations using the theory of {\it mutations} in algebraic geometry \cite{PolyMut1,PolyMut2,AG1}. Nevertheless, the standard techniques have the disadvantage that it requires our symplectic almost toric manifold $(X,\omega)$ to be closed. That said, the known tools for polytope mutations and algebraic degenerations do not include the relative case of symplectic divisors and Lagrangian submanifolds. This article starts developing techniques in this direction.

In Theorem \ref{thm:main}, we are interested in open symplectic toric domains $X(\SP,\SS)$, with $\SS\neq\varnothing$. In order to address this dissonance, we first compactify $X(\SP,\SS)$ to the closed symplectic toric manifold $X(\SP)$ by adding the surface configuration $D(\SS)\sse(X,\omega)$. Then we construct the sequence of almost toric fibrations $(\m_n)_{n\in\N}:X(\SP)\lr\R^2$ and, at the same time, keep track of the images of $D(\SS)\sse(X,\omega)$ under these almost toric fibrations. We achieve this latter step by developing a new theory of tropical symplectic curves, in Section \ref{sec:Symp_Trop},  and use the uniqueness of certain symplectic isotopy classes in symplectic 4-manifolds \cite{Ruled3,Ruled2,Ruled1} and the classification of Lagrangian 2-spheres in the monotone symplectic 4-manifolds $\C\P^1\times\C\P^1,Bl_3(\C\P^2),Bl_4(\C\P^2)$ \cite{LagSph3,LagrangianSphere1,LagSph4,LagSph2}.

In conclusion, the two main ingredients to execute the above scheme, and thus prove Theorem \ref{thm:main}, are:

\begin{itemize}
	\item[(i)] The construction of a sequence $(\m_n)_{n\in\N}:X\lr\R^2$ of almost toric fibrations, and its associated volume-filling symplectic ellipsoid embeddings $i^{(n)}:E(a_n,b_n)\lr (X,\omega)$. This is the content of Subsections \ref{ssec:SymingtonSeq} and \ref{ssec:ArithmeticSymington}.\\
	
	\item[(ii)] Showing that the images $i^{(n)}(E(a_n,b_n))$ of the symplectic embeddings in Part (i) lie in the complement of the compactifying divisor $D(\SS)\sse X(\SP,\SS)$. This will be shown in Subsection \ref{ssec:proofmain}.
\end{itemize}

The proof of Theorem \ref{thm:main} then proceeds in the following three steps. First, the existence and choice of a Symington Sequence, which addresses the first item above. Second, the construction of tropical configurations realizing $D(\SS)$ at each stage of the Symington Sequence, in Section \ref{sec:Symp_Trop}. Third, the study of the Hamiltonian isotopy class of these configurations in $X(\SP)$, which is the content of Subsection \ref{ssec:proofmain}. The second and third steps address the second item above.

\subsubsection{The symplectic domains in Figure \ref{fig:IntroTable}}
The combinatorial reason that sharp ellipsoid staircases can be constructed for the domains in Figure \ref{fig:IntroTable} is that they have almost toric basis with a smooth toric vertex, which admit infinitely many polytope mutations to triangular almost toric polytopes, with the smooth toric vertex fixed. The nine polytopes in Figure \ref{fig:IntroTable} are a subset of the 16 reflexive polytopes \cite{PolyMut3}.

These reflexive polytopes also yield one symplectic domain inside of $Bl_1(\C\P^2)$, and two symplectic domains inside of $Bl_2(\C\P^2)$, the three of them with infinitely many triangular polytope mutations. The arguments presented in this article provide infinitely many ellipsoid embeddings $E(1,a_n)$, with $(a_n)_{n\in\N}$ a convergent sequence, into $Bl_1(\C\P^2)$ and $Bl_2(\C\P^2)$. Nevertheless, these ellipsoid embeddings are {\it not} full. The twelve symplectic domains presented in \cite{DCG_Staircase} are precisely the nine domains in Figure \ref{fig:IntroTable} with the addition of these three domains. At this stage, it seems rather natural to ask about the remaining four, of the sixteen, reflexive polytopes. It might be possible to extract interesting staircases using our polytope mutation methods, however, the sequence of polytope mutations shall {\it not} be of triangular polytopes. The reader is referred to \cite{PolyMut1,PolyMut3} for a thorough study of the mutation classes of reflexive polytopes.\hfill$\Box$


\subsection{Realization of a Symington Sequence}\label{ssec:SymingtonSeq}

Given an almost toric symplectic 4-manifold $X(\SP)$, it is possible to mutate an almost toric fibration for $X(\SP)$ as in Subsection \ref{ssec:PolytopeMutation} in many ways. It is {\it not} true that {\it any} mutation, even if the underlying polytope is always triangular, will yield a full symplectic ellipsoid embedding.

\begin{example}\label{ex:NonEllipsoidMutation} We can mutate the triangular polytope $P_{\C\P^2}$ for $\C\P^2=X(P_{\C\P^2},\varnothing)$ to the toric moment polytope for the weighted projective space $\C\P^2(2,5,29)$. In fact, there are infinitely many ($\Q$-Gorenstein) degenerations of $\C\P^2$ obtained by mutating the triangular polytope $P_{\C\P^2}$, including $\C\P^2(5,29,433)$ and $\C\P^2(2,29,169)$ \cite{PolyMut2,AG1}. Neither of these weighted projective spaces readily admits a full symplectic ellipsoid embedding, since the three vertices of their toric moment polytopes are singular.\hfill$\Box$
\end{example}

Example \ref{ex:NonEllipsoidMutation} illustrates that a specific choice of sequence of mutations is needed in order to construct full ellipsoid embeddings with our method. This leads to the following definition, which we have named after M. Symington, after her exemplary article \cite{Symington}.

\begin{definition}\label{def:SymingtonSeq} A {\it Symington sequence} for an almost toric symplectic manifold $X(\SP,\SS)$ consists of a sequence of pairs $\{(\SP_n,v_n)\}_{n\in\N}$ such that:

\begin{itemize}
	\item[(S1)] $X(\SP_n)=X(\SP)$ for all $n\in\N$,
	
	\item[(S2)] $P_n$ is a triangle and $v_n\in V(P_n)$ is a vertex,
	
	\item[(S3)] $P_{n+1}=\mu_{v_n}(P_n)$, where $\mu_{v_n}(P_n)$ is the polytope mutation of $P_n$ at $v_n$,
	
	\item[(S4)] There exists $v_f\in\R^2$ such that $v_f\in V(P_n)$ is a smooth toric vertex, for all $n\in\N$.
\end{itemize}

The vertex $v_f$ in (S4) will be referred to as a {\it frozen} vertex, as it does not appear in the sequence $(v_n)_{n\in\N}$ specifying the sequence of mutations $(\mu_{v_n})_{n\in\N}$.\hfill$\Box$
\end{definition}

The crucial geometric properties in Definition \ref{def:SymingtonSeq} are (S1) and (S3), and the fact that we are interested in {\it ellipsoid} embeddings leads to requiring (S2) and (S4).

\begin{prop}\label{prop:ExistenceSymingtonSequence}
	Let $X(\SP,\SS)$ be a symplectic domain with relative almost toric base $(\SP,\SS)\in\SH$. Then there exists a Symington sequence for $X(\SP,\SS)$.
\end{prop}

\begin{proof}
	The statement readily holds for those $X(\SP,\SS)$ such that $P\in\SP$ is a triangle with a smooth toric vertex. Indeed, if $P$ is a triangle, a Symington sequence is obtained by choosing one of the two vertices in $P$ different from $v_f$, and iteratively mutating and choosing the only vertex in the mutated polytope which differs from $v_f$ and the newly created vertex. In case $P\in\SP$ is not a triangle, we must directly verify that the almost toric diagram $\SP$ can be mutated to triangular almost toric diagram with a smooth toric vertex. This is explicitly shown in \cite{Vianna3}, and Figure \ref{fig:IntroTable} for $Bl_3(\C\P^2)$.
\end{proof}

The arithmetic of the vertices of a Symington sequence for each $X(\SP,\SS)\in\SH$ is discussed in the subsequent Subsection \ref{ssec:ArithmeticSymington}. The usefulness of a Symington sequence in the study of the Ellipsoid Embedding function $c_X$ comes from the following:

\begin{prop}\label{prop:EllipsoidInTriangle} Let $P\sse\R^2$ be a triangular polytope with a smooth toric vertex $v_f\in V(P)$, and let $(a,b)\in\Z^2$ be the affine length of the two sides of $P$ incident to $v_f$. Then there exists a full symplectic embedding $i:E(a,b)\lr X(P)$.\hfill$\Box$
\end{prop}

See \cite{LMS13} for the notion of a full symplectic embedding. Proposition \ref{prop:EllipsoidInTriangle} follows by noticing that the complement of the symplectic divisor associated to the side in $P$ non-incident to $v_f$ is the symplectic domain $E(a,b)$ \cite{Toric1,Hutchings3,Schlenk1}. In particular, let $\{(\SP_n,v_n)\}_{n\in\N}$ be a Symington sequence and denote by $(a_n,b_n)\in\Z^2$ the affine lengths of the sides of $P_n$ incident to $v_f$. Then there exist a sequence of symplectic ellipsoid embeddings
$$i_n:E(a_n,b_n)\lr X(P_n).$$
This is at the core of the relation between the existence of infinite staircases and polytope mutations. It is a powerful starting point, but it not enough to conclude Theorem \ref{thm:main} since the symplectic varieties $X(P_n)$ are not isomorphic. As emphasized, $X(P_n)$ are typically singular algebraic varieties, and their algebraic isomorphism type strongly depends on $n\in\N$. That said, the fundamental defining property (S1) of a Symington sequence states that $X(\SP_n)=X(\SP)$ for all $n\in\N$.

The difference between the equalities $X(\SP_n)=X(\SP)$, for all $n\in\N$, and the sequence of algebraic varieties $X(P_n)$ is contained in the cut content $\SB_n$ in $\SP_n=(P_n,\SB_n)$. In particular, Proposition \ref{prop:EllipsoidInTriangle} implies the following:

\begin{prop}\label{prop:EllipsoidInSmoothing} Let $P\sse\R^2$ be a triangular polytope with a smooth toric vertex $v_f\in V(P)$, and let $(a,b)\in\Z^2$ be the affine length of the two sides of $P$ incident to $v_f$. Suppose that $\SP=(P,\SB)$ has empty cut content at $v_f$. Then there exists a full symplectic embedding $i:E(a,b)\lr X(\SP)$.\hfill$\Box$
\end{prop}

Note that introducing the cut content $\SB$ can be achieved by choosing arbitrarily short cuts at the vertices of
$P$. In this manner, we can choose $N\sse P$ to be a neighborhood of the side in $P$ opposite to $v_f$ containing the
cuts in $\SB$, so that the complement in $X(\SP)$ of the pre-image of $N\sse P$ is the symplectic ellipsoid
$E(a-\varepsilon,b-\varepsilon)$ for any $\varepsilon\in\R^{>0}$, depending on the choice of $N$; since $\epsilon > 0$ can be chosen arbitrarily small, the claim follows by taking into account Remark \ref{rmk:Volume_filling}. \color{black} 

Given a Symington sequence
$\{(\SP_n,v_n)\}_{n\in\N}$ for $X(\SP,\SS)$, Proposition \ref{prop:ExistenceSymingtonSequence} and Proposition
\ref{prop:EllipsoidInSmoothing} allow us to construct a sequence of full symplectic ellipsoid embeddings
$$i_n:E(a_n,b_n)\lr X(\SP).$$

The challenge is now upgrading these absolute ellipsoid embeddings
$(i_n)_{n\in\N}$ to a sequence of relative full symplectic ellipsoid embeddings $$\iota_n:E(a_n,b_n)\lr
X(\SP,\SS).$$ This would conclude the proof of Theorem \ref{thm:main} once we study the arithmetic properties
of the Symington sequence $(a_n,b_n)$ associated to $X(\SP,\SS)$ as in Proposition
\ref{prop:ExistenceSymingtonSequence}. The focus of Subsection \ref{ssec:proofmain} and Section
\ref{sec:Symp_Trop} is the construction of the relative embeddings $\iota_n$ into $X(\SP,\SS)$ from the
absolute embeddings $i_n$ into $X(\SP)$.

The central difficulty in showing the existence of $\iota_n$ is understanding the surface configuration $D(\SS)\sse X(\SP)$ under the identifications $X(\SP)\cong X(\SP_n)$, for each $n\in\N$. In order to achieve that, we will develop a diagrammatic tropical calculus for surface configurations $D(\SS)\sse X(\SP)$ in almost toric symplectic manifolds. In particular, such a tropical calculus starts with a configuration $D(\SS)\sse X(\SP_n)$ and describes a diagram $\Delta_n\sse\SP_n$ such that the Hamiltonian isotopy class of the configuration $D(\SS)$ admits a representative contained in the pre-image of this diagram $\Delta_n$.

\begin{remark}
	There is an alternative course of action to construct $\iota_n$. It should be possible to understand how the surface configuration $D(\SS)$ is explicitly carried along a polytope mutation through almost toric diagrams \cite{Symington,Vianna3}. Starting with $D(\SS)\sse X(\SP)$, this would yield an understanding of the inclusions $D(\SS)\sse X(\SP_n)$, upon identifying $X(\SP_n)\cong X(\SP)$ along the mutation sequence. This is an interesting line of research, but we shall not discuss it in the present manuscript.\hfill$\Box$
\end{remark}

Let us now provide a detailed and self-contained account of the numerics appearing in Symington sequences for $X(\SP,\SS)$ for $(\SP,\SS)\in\SH$.



\subsection{Arithmetic of Symington Sequences for $(\SP,\SS)\in\SH$}\label{ssec:ArithmeticSymington} The polytope $P\in\SP$ in an almost toric base contains all the arithmetic information for its mutations. The useful property of polytopes in Figure \ref{fig:IntroTable} is the following:

\begin{lemma}
Each of the polytopes in Figure \ref{fig:IntroTable} is mutation equivalent to a triangular polytope.
\end{lemma}
\begin{proof}
This is immediate for the first two rows: the toric moment polytope for $\C\P^2$ is triangular, the square toric moment polytope for the monotone $\C\P^1\times\C\P^1$ is mutation equivalent to the triangular moment polytope for $\C\P^2(1,1,2)$, and the hexagonal moment polytope for the monotone $Bl_3(\C\P^2)$ is mutation equivalent to the triangular moment polytope for $\C\P^2(1,2,3)$. These triangular polytopes are depicted in Figure \ref{fig:IntroTable} at the rightmost part for each region corresponding to each symplectic 4-manifold. Finally, the almost toric polytope $Bl_4(\C\P^2)$ can be mutated to $\C\P^2(1,4,5)$ as shown in \cite{Vianna3}.
\end{proof}

This allows us to reduce the arithmetic of Symington sequences for the polytopes in Figure \ref{fig:IntroTable} to those for the weighted projective spaces $\C\P^2(1,1,1),\C\P^2(1,1,2),\C\P^2(1,2,3)$ and $\C\P^2(1,4,5)$. The following result covers all the necessary arithmetic for our Theorem \ref{thm:main}:

\begin{proposition}\label{prop:arithmetic} Let $P(\alpha,\beta,\gamma)$ be the toric moment polytope for $\C\P^2(\alpha,\beta,\gamma)$.
\begin{itemize}
	\item[(i)] Suppose that $\alpha$ corresponds to a smoothable singularity
	of $\C\P^2(\alpha,\beta,\gamma)$ and divides $(\beta+\gamma)^2$, with $\alpha\cdot \delta=(\beta+\gamma)^2$. 
	Then the polytope $P(\alpha,\beta,\gamma)$ admits a polytope mutation to the triangle $P(\beta,\gamma,\delta)$.\\
	
	\item[(ii)] ({\it Vieta jumping}) Suppose that $(p,q,r)\in\Z^3$ solves the Diophantine equation
	$$C_0p^2+C_1q^2+C_2r^2=mpqr.$$
	Suppose that the $C_i$ divide $m$, $1\leq i\leq3$. Then either of the three triples $(p,q,mpq/C_2-r)$, $(p,mpr/C_1-q, r)$ and $(mqr/C_0-p,q,r)$
	solves this Diophantine equation.
	
	\hfill$\Box$
\end{itemize}
\end{proposition}	

\begin{remark}
A toric orbifold singularity being smoothable is equivalent to being a $T$-singularity \cite{PolyMut2,PolyMut3}. If $\C\P^2(\alpha,\beta,\gamma)$ is smoothable, then it is of the form
$\C\P^2(C_0p^2,C_1q^2,C_2r^2)$, where $(p,q,r)\in\Z^3$ solves the Diophantine
equation above. Vieta jumping is then equivalent to the change $\alpha \leftrightarrow \delta$ in 
Proposition \ref{prop:arithmetic}(i), with $\delta = C_0(mqr/C_0-p)$. Note that Proposition \ref{prop:arithmetic} is proven in \cite{PolyMut2} for more general toric orbifolds with triangular moment map and smoothable corner, and it appears in the context of Del Pezzo surfaces in \cite{Vianna3} as Lemma~4.2.\hfill$\Box$
\end{remark}

{\bf Arithmetic for $(\C\P^2,\omega_\st)$}. The toric moment polytope for $(\C\P^2,\omega_\st)$ is mutation equivalent to the toric moment polytope for $\C\P^2(p^2,q^2,r^2)$ if and only if $(p,q,r)\in\Z^3$ is a Markov triple, i.e.
$$p^2+q^2+r^2=3pqr.$$
This explains the numerics in Example \ref{ex:NonEllipsoidMutation}. An
arbitrary sequence of Markov triples, even if they differ only in one component,
does {\it not} yield a Symington Sequence, since condition (S4) in Definition
\ref{def:SymingtonSeq} imposes a non-trivial constraint. Indeed,
$\C\P^2(p^2,q^2,r^2)$ has a smooth toric vertex if and only if one of the
numbers $p,q,r\in\Z$ equals one. Thus a Symington Sequence for
$(\C\P^2,\omega_\st)$ can be obtained using the construction in the proof of
Proposition \ref{prop:ExistenceSymingtonSequence} yielding any of the toric
polytopes for $\C\P^2(1,q^2,r^2)$, if $(q,r)\in\Z^2$ satisfy 
$$1+q^2+r^2=3qr.$$
Let $F_n$ be the $n$th odd-index Fibonacci number, starting at $F_1=1$. The recursion
$F_{n+2}=3F_{n+1}-F_n$ implies that $(1,F_n,F_{n+1})$ is a Markov triple with
$p=1$, as required. Hence, a Symington Sequence for $(\C\P^2,\omega_\st)$ yields
the sequence of toric polytopes $\C\P^2(1,F^2_n,F^2_{n+1})$ for all $n\in\N$.
Conversely, all positive integral solutions $(1,q,r)$ of the Markov equation are
of the form $(1,F_n,F_{n+1})$. These Fibonacci solutions are directly obtained
by iteratively applying the Vieta jumping in Proposition \ref{prop:arithmetic}
starting with the minimal solution $(p,q,r)=(1,1,1)$. The 
ratio $F_{n+1}/F_{n}$ is known to converge to $1 + \varphi$, where $\varphi$ is the golden ratio. This is readily extracted from $1+q_n^2+r_n^2=3q_nr_n$, by dividing by $r_n^2$ and 
noticing that in the limit $1/r_n^2 \to 0$ as $n\to \infty$, since $1+\varphi$ is the root
of $1 -3x + x^2$, that is grater than $1$.

This is the explanation for the appearance of the odd Fibonacci numbers in the McDuff-Schlenk Fibonacci staircase from the viewpoint of polytope mutations. Indeed,
this computation shows that the $E(1,F^2_{n+1}/F^2_n)$ admits a full symplectic
embedding into $$\left(\C\P^2,\sqrt{\frac{2\Vol(E(1,F^2_{n+1}/F^2_n))}{\pi^2}}\cdot\omega_\st\right),$$
and that
any symplectic ellipsoid embedding obtained via toric mutations must be of this
form. \\

\begin{remark}\label{remark:volume}
 From the proof of Theorem~6.5 in \cite{LeeOhVianna}, see \cite[Figure~15]{LeeOhVianna}, we can
readily conclude that $E(1,F^2_{n+1}/F^2_n)$ embeds
into
$$\left(Bl_k(\C\P^2), \frac{F^2_{n+1}/F^2_n}{3F_{n+1}/F_n -
k} \sqrt{\frac{2\Vol(E(1,F^2_{n+1}/F^2_n))}{\pi^2}}\cdot\omega_\st\right),\quad k \le 5.$$ We learned about staircases in the toric
domains related to $Bl_1(\C\P^2)$ and $Bl_2(\C\P^2)$ from \cite{DCGetAL}.
\hfill$\Box$
\end{remark}

{\bf Arithmetic for $(\C\P^1\times\C\P^1,\omega_\st\oplus\omega_\st)$}. The toric moment polytope for $(\C\P^2(1,1,2),\omega_\st)$ is mutation equivalent to the toric moment polytope for $\C\P^2(p^2,q^2,2r^2)$ if and only if $(p,q,r)\in\Z^3$ satisfies
$$p^2+q^2+2r^2=4pqr.$$
As above, these $\C\P^2(p^2,q^2,2r^2)$ have a smooth toric vertex if and only if one of the numbers $p,q,r\in\Z$ equals one. Thus a Symington Sequence for $(\C\P^1\times\C\P^1,\omega_\st\oplus\omega_\st)$ exists and yields any of the toric polytopes for $\C\P^2(1,q^2,2r^2)$, if $(q,r)\in\Z^2$ satisfy
$$1+q^2+2r^2=4qr.$$
Let $P_n,H_n$ be the $n$th Pell number and the $n$th half-companion Pell number. They are defined by the initial values
$$P_0=0,P_1=1,\quad H_0=H_1=1,$$
and the recurrence relations
$$P_n=2P_{n-1}+P_{n-2} , \quad H_n=2H_{n-1}+H_{n-2}.$$
The triples $(1,q,r) = (1, H_{2n},P_{2n\pm 1})$ verify the above Diophantine equation. Indeed, rewriting the Diophantine equation as $1+2r^2= q(4r - q)$ and $1+q^2= 2r(2q - r)$, we see that 
mutations induce the recurrence relations: 
\[ q_0 = r_0 = 1, \ \ q_{2n+1} = 4r_{2n} - q_{2n}, \ \ 
q_{2n+2} = q_{2n+1}, \ \ r_{2n+1} = r_{2n}, \ \ r_{2n + 2} = 2q_{2n+1} - r_{2n+1}.\] 

Then, one shows by induction
that the Pell numbers satisfy $P_n = H_n - P_{n-1}$, and that the relations $H_{2n+2} = 4P_{2n+1} - H_{2n}$ 
and $P_{2n+1} = 2H_{2n} - P_{2n-1}$ follow from the recurrence relations and $P_n = H_n - P_{n-1}$. 
Therefore, setting $q_{2n} = q_{2n -1} = H_{2n}$ and $r_{2n+1} = r_{2n} = P_{2n+1}$, we conclude that the triples $(1,q_n,r_n)$ form solutions to the above Diophantine equation.


Rescaling the ellipsoid $E(q_n^2,2r_n^2)$ and setting
$$c_{2n -
1}=\max{\{2P_{2n-1}^2/H_{2n}^2,H_{2n}^2/(2P_{2n-1}^2)\}},\quad c_{2n}=\max{\{2P_{2n+1}^2/H_{2n}^2,H_{2n}^2/(2P_{2n+1}^2)\}},$$ we obtain that the symplectic ellipsoid
$E(1,c_n)$ symplectically embeds into
$(\C\P^1\times\C\P^1,\widetilde\Vol(E(1,c_n))\cdot(\omega_\st\oplus\omega_\st))$ for all $n\in\N$, where
$\widetilde\Vol$ denotes the rescaled volume, as in Remark \ref{remark:volume}. This explains the
numerics in the Frenkel-M\"uller Pell staircase via polytope mutations. The limit of $c_n$, for large enough $n\in\N$, can be also
extracted from $1+q_n^2+2r_n^2=4q_nr_n$: dividing by $r_n^2$ and taking the solution of $x^2 -4x + 2 = 0$
that satisfies $x^2/2 > 1$, which is $2 + \sqrt{2}$, one obtains that $c_n \to x^2/2 = 3 + 2\sqrt{2}$.\\

{\bf Arithmetic for $(Bl_3(\C\P^2),\omega_\st)$}. The toric moment polytope for
$(\C\P^2(1,2,3),\omega_\st)$ is mutation equivalent to the toric moment polytope
for $\C\P^2(p^2,2q^2,3r^2)$ if and only if $(p,q,r)\in\Z^3$ satisfies
$$p^2+2q^2+3r^2=6pqr.$$ 
A Symington Sequence for $(Bl_3(\C\P^2),\omega_\st)$
thus yields any of the toric polytopes for $\C\P^2(1,q,r)$, if
$(q,r)\in\Z$ satisfy 
$$1+2q^2+3r^2=6qr.$$ 
This Diophantine equation gives
rise to the numerics of the Cristofaro-Gardiner-Kleinman staircase. Indeed,
Vieta jumping alternately applied to the second and third components of
$(p,q,r)=(1,1,1)$ yields the sequence of triples
$$(1,1,1)\longmapsto(1,2,1)\longmapsto(1,2,3)\longmapsto(1,7,3)\longmapsto\ldots$$
which prove the existence of full ellipsoid embeddings into the appropriately
rescaled $(Bl_3(\C\P^2),\omega_\st)$, starting with the corresponding sequence
$$E(1,3/2),E(1,8/3),E(1,27/8),E(1,98/27),\ldots,$$ associated to the triples
above. The sequence of ellipsoids $E(1,k_n)$, where $\{k_n\}_{n\in\N}$ 
is associated to ratios of solutions for the
Diophantine equation obtained by Proposition \ref{prop:arithmetic} is a
convergent infinite sequence, limiting to $2x^2/3 = 2 + \sqrt{3}$,
where $ x = \frac{3+\sqrt{3}}{2}$ is the solution of
$2x^2 -6x + 3 =0$ with $2x^2/3 = 2 + \sqrt{3} > 1$, which is obtained from dividing $1+2q_n^2+3r_n^2=6q_nr_n$ by $r_n^2$
and taking the limit $n \to \infty$. \\

{\bf Arithmetic for $(Bl_4(\C\P^2),\omega_\st)$}. The toric moment polytope for
$(\C\P^2(1,4,5),\omega_\st)$ mutates to the toric moment polytope for
$\C\P^2(p^2,q^2,5r^2)$ if and only if $(p,q,r)\in\Z^3$ satisfies
$$p^2+q^2+5r^2=5pqr.$$ In conclusion, a Symington Sequence for
$(Bl_4(\C\P^2),\omega_\st)$ yields any of the toric polytopes for
$\C\P^2(1,q^2,5r^2)$, if $(q,r)\in\Z^2$ satisfy 
$$1+q^2+5r^2=5qr.$$ The
full ellipsoid embeddings come from the sequence of triples
$$(1,2,1)\longmapsto(1,3,1)\longmapsto(1,3,2)\longmapsto(1,7,2)\longmapsto(1,7,5)\longmapsto\ldots,$$
which yields the sequence of optimal embeddings of $E(1,5/4),E(1,9/5),E(1,20/9),E(1,49/20),E(1,125/49)\ldots$ into a rescaling of $(Bl_4(\C\P^2),\omega_\st)$. This is a new full ellipsoid
staircase $E(1,l_n)$, with sequence $\{l_n\}_{n\in\N}$ of ratios of solutions converging to 
$1+\varphi$, where $\varphi$ is again the golden ratio. Indeed, $1+\varphi= x^2/5$, where
$$x = \frac{5 + \sqrt{5}}{2},$$
is the solution of 
$x^2 - 5x + 5 = 0$, with $x^2/5 > 1$, obtained from dividing $1+q_n^2+5r_n^2=5q_nr_n$
by $r_n^2$ and taking $n \to \infty$.\\


\subsection{Proof of Theorem \ref{thm:main}}\label{ssec:proofmain}

Let us now prove Theorem \ref{thm:main}, building on the Symington sequences discussed in Subsections \ref{ssec:SymingtonSeq} and \ref{ssec:ArithmeticSymington}. For that, let us first state the result that we need from the theory of symplectic-tropical curves, as developed in Section \ref{sec:Symp_Trop}. The following statement is the only tropical ingredient for our proof of Theorem \ref{thm:main}, its proof will be the content of Section \ref{sec:Symp_Trop}.

\begin{thm}\label{thm:ExistenceTropicalSymplectic}
	Let $X(\SP,\SS)$ be a symplectic toric domain, with $(\SP,\SS)\in\SH$, let $(\SP_n,\SS_n)$ be the associated Symington sequence, and let $$i_n:E(a_n,b_n)\lr X(\SP)$$
	be a full symplectic ellipsoid embedding. For any $\varepsilon\in\R^{>0}$, there exists a tropical symplectic curve $\widetilde{S}_n\sse\SP_n$ such that
	\begin{itemize}
		\item[(i)] The tropical symplectic curve $\widetilde{S}_n\sse\SP_n$ represents any embedded (configuration of) symplectic surface(s) $D(\widetilde{S}_n)\sse X(\SP_n)$ in the same homology class as the (configuration of) embedded symplectic curve(s) in $D(S)$,\\
		
		\item[(ii)] There exists a neighborhood $\Op(\widetilde{S}_n)\sse P_n$ and a volume-filling symplectic embedding
		$$\iota_n:E(a_n,b_n)\lr X(\SP_n)\setminus \Op(\widetilde{S}_n).$$\hfill$\Box$
	\end{itemize}
\end{thm}

The inclusion $\widetilde{S}_n\sse\SP_n$ in Theorem \ref{thm:ExistenceTropicalSymplectic} is to be understood as $\widetilde{S}_n\sse P_n$ for a realization of the cut content $\SB_n$ of $\SP_n$. The exact realization of the cut content $\SB_n$, as cuts in $P_n$, is a choice and, in Theorem \ref{thm:ExistenceTropicalSymplectic}, this choice depends on the initial value of $\varepsilon\in\R^{>0}$. Let us now apply Theorem \ref{thm:ExistenceTropicalSymplectic} and conclude Theorem \ref{thm:main} in each of the cases.

\begin{remark}
The exact configurations of symplectic curves that we shall use in Theorem \ref{thm:ExistenceTropicalSymplectic} are specifically constructed in Proposition \ref{prp:PxP_H1H2} for $\CP^1\times\CP^1$, Proposition \ref{prp:3Blup_A1A2A5} for $\BlIII$ and Proposition \ref{prp:4Blup_A1A2A5} for $\BlIV$.\hfill$\Box$
\end{remark}

{\bf Complex Projective Plane $\C\P^2$}. We start from the triangular toric fibration of $\CP^2$ in Figure 1 with $\SB=\varnothing$. The relative almost toric base $(\SP,\SS)$ has $D(\SS)=\C\P^1$, the complex projective line, as its unique symplectic divisor. This relative almost toric base leads to the symplectic toric domain of the standard symplectic 4-ball $$X(\SP,\SS)=(\bD^4(1),\omega_\st)\cong(\C\P^2\setminus\C\P^1,\omega_\st).$$
Let $c_n=\Vol(E(1,F^2_{n+1}/F^2_n))$, with $\{F_n\}_{n\in\N}$ as in Subsection \ref{ssec:ArithmeticSymington}, and let us show that there exists a full symplectic embedding
$$\iota_n:E(1,F^2_{n+1}/F^2_n)\lr (\bD^4(1),c_n\cdot\omega_\st).$$
The sequence of polytope mutations, as discussed in Subsection \ref{ssec:PolytopeMutation}, associated to the Symington sequence in Subsection \ref{ssec:SymingtonSeq} yields a sequence of full symplectic embeddings
$$i_n:E(1,F^2_{n+1}/F^2_n)\lr (\C\P^2,c_n\cdot\omega_\st).$$
This sequence follows the arithmetic in Subsection \ref{ssec:ArithmeticSymington} stemming from the Markov Equation with $a=1$. In order to guarantee an embedding into $X(\SP,\SS)$, it suffices to show that $D(\SS)=\C\P^1$ can be symplectically isotoped to lie above an arbitrarily small $\varepsilon$-neighborhood of the side opposite to the frozen vertex $v_f\in P_n$, $n\in\N$. This is achieved by first constructing a symplectic $\C\P^1$ above this $\varepsilon$-neighborhood, and then arguing that this symplectic $\C\P^1$ can be symplectically isotoped to the standard $D(\SS)=\C\P^1\sse X(\SP)$. The former part is achieved by Theorem \ref{thm:ExistenceTropicalSymplectic}, the symplectic-tropical curves of which are depicted explicitly in Figure \ref{fig:SympTropMutationP2}, for the first five mutations in the Fibonacci-Symington Sequence. The latter part, constructing a symplectic isotopy from the lift of the symplectic-tropical curve to the standard complex line $D(\SS)$, is now achieved by using M. Gromov's \cite[Section 0.2.B]{Gromov85}, which shows that the symplectic isotopy class of the complex line is unique.

\begin{center}
	\begin{figure}[h!]
		\centering
		\includegraphics[scale=0.8]{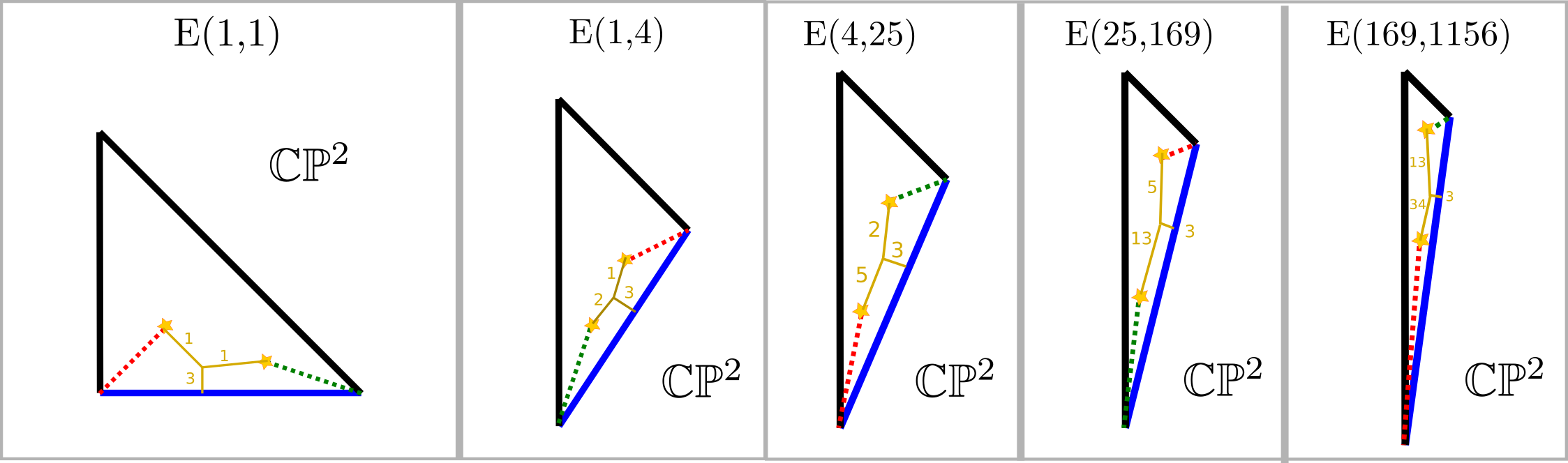}
		\caption{Symplectic-tropical curves for $(\C\P^2,\omega_\st)$. The weights of the tripods are, from left to right, $(1,1,3)$ for the toric moment polytope, $(1,2,3),(5,2,3),(13,5,3)$ and $(34,13,3)$. These five tropical symplectic curves represent symplectic embedded lines $\C\P^1\sse(\C\P^2,\omega_\st)$.}
		\label{fig:SympTropMutationP2}
	\end{figure}
\end{center}

{\bf The Ruled Surface $(\C\P^1\times\C\P^1,\omega_\st\oplus\omega_\st)$}. There are two (relative) almost toric bases $(\SP,\SS)$ in Figure \ref{fig:IntroTable}, associated to the symplectic domains $\bD^2(1)\times\bD^2(1)$, the polydisk, and the ellipsoid $E(1,2)$.

The first almost toric base $(\SP,\SS)$, on the right, has underlying polytope $P\in\SP$ the toric moment polytope for the weighted projective plane $\C\P^2(1,1,2)$. The cut content $\SB\in\SP$ is non-trivial at one of the vertices, where a Lagrangian 2-sphere $L(\SB)$ lies above the segment uniting the two nodal points. This Lagrangian 2-sphere is in the homology class of the Lagrangian anti-diagonal in $X(\SP)\cong(\C\P^1\times\C\P^1,\omega_\st\oplus\omega_\st)$, and it is in fact Hamiltonian isotopic to the anti-diagonal. The surface configuration $D(\SS)$ consists of two surfaces: a symplectic 2-sphere, in the homology class of the complex line $\{pt\}\times\C\P^1$, and the Lagrangian 2-sphere $L(\SB)$. The sequence of polytope mutations constructed in Subsection \ref{ssec:SymingtonSeq} yields a sequence of embeddings
$$i_n:E(1,p_n)\lr (\C\P^2(1,1,2),\Vol(E(1,p_n))\cdot(\omega_\st\oplus\omega_\st)),$$
where $p_n=\max{\{2P_n^2/H_{n+1}^2,H_{n+1}^2/(2P_n^2)\}}$, according to the arithmetics discussed in Subsection \ref{ssec:ArithmeticSymington}, and we are denoting $\Vol=\widetilde\Vol$ onwards to ease notation.

In order for these embeddings to be into $X(\SP)=E(1,2)$, instead of $X(P)=\C\P^2(1,1,2)$, we first apply Theorem \ref{thm:ExistenceTropicalSymplectic} and construct a symplectic-tropical curve inside an $\varepsilon$-neighborhood of the edge opposite to the frozen vertex, which includes the cuts in $\SB$. This tropical curve $\widetilde{S}_n$ is chosen such that $D(\widetilde{S}_n)$ is a symplectic 2-spheres in the homology class $[\{pt\}\times\C\P^1]$. By \cite[Section 2.4.C]{Gromov85}, the isotopy class of $\{pt\}\times\C\P^1\sse \C\P^1\times\C\P^1$, through embedded symplectic surfaces, is unique. Thus, there exists a symplectic isotopy from $D(\widetilde{S}_n)$ to $\{pt\}\times\C\P^1$. The image of the Lagrangian sphere $L(\SB_n)$ under this symplectic isotopy might not coincide with $L(\SB)$. Nevertheless, since $L(\SB_n)$ and $L(\SB)$ are both Lagrangian 2-spheres in $\C\P^1\times\C\P^1$ in the class of the anti-diagonal, there exists a Lagrangian isotopy connecting them, relative to the divisor $\{pt\}\times\C\P^1$, as shown in \cite{LagrangianSphere1}. Since an exact Lagrangian isotopy can be realized by a Hamiltonian isotopy, this yields a symplectic isotopy from our surface configuration $D(\widetilde{S}_n)$ and $L(\SB_n)$ to the standard $D(S)$ and $L(\SB)$. This constructs the full symplectic embeddings
$$i_n:E(1,p_n)\lr (E(1,2),\Vol(E(1,p_n))\cdot\omega_{\st}).$$

The second almost toric base, on the left, corresponds to the toric moment polytope for $\C\P^1\times\C\P^1$ and has $\SB=\varnothing$. This data has the polydisk $X(\SP)=\bD^2(1)\times\bD^2(1)$ as its associated domain. Indeed, the surface configuration $D(\SS)$ for this first $(\SP,\SS)$, depicted in blue, consists of a chain of two symplectic embedded 2-spheres, corresponding to $H_1=\C\P^1\times\{pt\}$ and $H_2=\{pt\}\times\C\P^1$ respectively. Thus the divisor $D=H_1+H_2$ is defined as $(\C\P^1\times\{pt\})\cup(\{pt\}\times\C\P^1)\sse X$, which represents the element $(1,1)\in H^2(\C\P^1\times\C\P^1,\Z)\cong\Z^2$. At this stage, we can directly invoke \cite[Corollary 1.6]{PellStaircase}, which states that a symplectic ellipsoid embeds fully in $E(1,2)$ if and only if it embeds fully in $\bD^2(1)\times\bD^2(1)$. This concludes Theorem \ref{thm:main} for these two almost toric bases.\\

{\bf Thrice Blown-up $Bl_3(\C\P^2)$}. There are four relative almost toric basis $(\SP,\SS)$ with $X(\SP)\cong
Bl_3(\C\P^2)$ symplectomorphic to the blow up of $(\C\P^2,\omega_\st)$ at three non-collinear points. In line
with the previous case of symplectic domains in $\C\P^1\times\C\P^1$, there exists an infinite sharp sequence
of ellipsoids for one of these four symplectic domains $X(\SP,\SS)$ if and only if it exists for one of them. See for instance \cite[Theorem 1.2]{CG19}, as an ellipsoid is a concave (and convex) domain (see also \cite{Hutchings1,ECHnotes}). The numerics of these staircases
will also coincide. Let us focus on the leftmost relative almost toric base $(\SP,\SS)$ in Figure
\ref{fig:IntroTable}. In this case, $\SS$ is a unique symplectic configuration realized by a linear plumbing
of four symplectic 2-spheres each with self-intersection $(-1)$, and representing the classes $E_1$, $H - E_1
- E_3$, $H - E_1 - E_2$ and $E_2$.

Theorem \ref{thm:ExistenceTropicalSymplectic} applied to this singular configuration -- see Proposition \ref{prp:3Blup_A1A2A5} -- yields a tropical curve in an $\varepsilon$-neighborhood of the side opposite to the frozen vertex $(0,0)$. Then we need to prove that any such symplectic configuration $S_1,S_2,S_3,S_4\sse Bl_3(\C\P^2)$ is equisingularly symplectic isotopic to a given such configuration $D$ in the same homology class. For that, blow down the second and fourth spheres $S_2,S_4\sse Bl_3(\C\P^2)$. The blow down of $S_1$ becomes a $0$-self-intersection 2-sphere, and the blow down of $S_3$ is a $(+1)$-self-intersection 2-sphere in $Bl_1(\C\P^2)$. The classification of ruled symplectic surfaces \cite{Ruled3,Ruled1} gives a symplectomorphism of $Bl_1(\C\P^2)$ that sends the symplectic configuration $S_1\cup S_3$ to the union of the proper transform of the complex line $\C\P^1\sse\C\P^2$ and the $0$-self-intersection fiber of the linear pencil of $\C\P^2$. Since the symplectomorphism group of $Bl_1(\C\P^2)$ is connected \cite{LagSph4}, this can be achieved via a symplectic isotopy. This same symplectic isotopy, upon blowing up twice, yields the required symplectic isotopy.

\begin{remark}\label{rmk:Bl3Divisors}
The argument above requires the construction of the 4-chain of symplectic spheres, provided by Theorem \ref{thm:ExistenceTropicalSymplectic}. Nevertheless, it is possible to instead argue with the rightmost relative almost toric base in Figure \ref{fig:IntroTable}. This alternative argument requires an understanding of the Lagrangian spheres in $Bl_3(\C\P^2)$ -- which fortunately exists -- and reads as follows.

The rightmost relative almost toric base yields the symplectic ellipsoid $E(2,3)\cong X(\SP,\SS)$, with the unique symplectic 2-sphere in $D(\SS)$ in the homology class of an exceptional divisor $E\sse Bl_3(\C\P^2)$. Consider the Symington sequence $\SP_n$, constructed in Subsection \ref{ssec:ArithmeticSymington}, associated to $\C\P^2(1,2,3)$, which yields a sequence of full ellipsoid embeddings
$$i_n:E(1,k_n)\lr (Bl_3(\C\P^2),\Vol(E(1,k_n))\cdot\omega_\st).$$
Theorem \ref{thm:ExistenceTropicalSymplectic} now constructs an almost-tropical curve $\widetilde{S}_n\sse P_n$ such that $D(\widetilde{S}_n)$ is an embedded symplectic 2-sphere in the homology class $[E]$. At this stage we proceed as before, by constructing a symplectic isotopy from $D(\widetilde{S}_n)$ to the original exceptional divisor $E\sse Bl_3(\C\P^2)$, which is itself given by $D(S)$, where $S$ is the side of the toric moment polytope of $\C\P^2(1,2,3)$ opposite to the frozen smooth vertex. The symplectic isotopy embeds the complement of $D(\widetilde{S}_n)$ into the complement of $E$, by upgrading it to an ambient symplectic isotopy, which is possible since the symplectic 2-spheres are embedded. As in the case of $E(1,2)$, the symplectic ellipsoid $E(2,3)=X(\SP,\SS)$ arises as the complement of $E$ and three Lagrangian 2-spheres; this is the geometric incarnation of $\SS$ containing surfaces above the cuts. Similar to the case of $\C\P^1\times\C\P^1$, the Hamiltonian isotopy classes of these Lagrangian 2-spheres in $Bl_3(\C\P^2)$ are known to be unique \cite{LagSph3,LagSph2} and the Lagrangian isotopy can be taken to be in the complement of a stable symplectic sphere configuration \cite{LagSph4}. The composition of our symplectic isotopy with a compactly supported Hamiltonian isotopy, bringing these three Lagrangian 2-spheres to our standard configuration in $X(\SP,\SS)$, yields a full symplectic embedding
$$\iota_n:E(1,k_n)\lr (E(2,3),\Vol(E(1,k_n))\cdot\omega_\st).$$\hfill$\Box$
\end{remark}
{\bf Four Times Blown-up $Bl_4(\C\P^2)$}. The argument follows the same pattern as the previous three cases. Consider generators $\langle H,E_1,E_2,E_3,E_4\rangle\in H^2(Bl_4(\C\P^2),\Z)$ given by the proper transform of the complex line in $\C\P^2$, away from the blow up points, and $E_i$ the exceptional divisors, $1\leq i\leq4$. Let us analyze its leftmost relative almost toric base $(\SP,\SS)$ in Figure \ref{fig:IntroTable}, with $D(\SS)$ a surface configuration consisting of a symplectic linear chain $C$ of three self-intersection-($-1$) spheres, in the ordered homology classes $[E_1],[H]-[E_1]-[E_4],[E_4]$, and two Lagrangian spheres in the homology classes $[E_3]-[E_1]$ and $[E_4]-[E_2]$. In the same scheme as before, the Symington sequence $\SP_n$ associated to $Bl_4(\C\P^2)$, with arithmetic as in Subsection \ref{ssec:ArithmeticSymington}, yields full symplectic embeddings
$$i_n:E(1,l_n)\lr (Bl_4(\C\P^2),\Vol(E(1,l_n))\cdot\omega_\st).$$
Theorem \ref{thm:ExistenceTropicalSymplectic} yields an almost toric tropical curve $\widetilde{S}_n$ for each of these almost toric bases $\SP_n$, which yields the required symplectic configuration $C\sse Bl_4(\C\P^2)$ and with each 2-sphere in the same homology class as the 2-spheres in $C$. See Proposition \ref{prp:4Blup_A1A2A5} for the explicit construction of $\widetilde{S}_n$. Now, the equisingular symplectic isotopy class of these configuration is unique, by the same type of argument as in the cases of $Bl_3(\C\P^2)$ and $\C\P^1\times\C\P^1$. Indeed, blowing-down the two extremal 2-spheres in the classes $[E_1]$ and $[E_4]$ gives a $(+1)$-symplectic 2-sphere in $Bl_2(\C\P^2)$ disjoint from the remaining two exceptional divisors. This provides a symplectic isotopy in $Bl_2(\C\P^2)$ from our configuration $D(\widetilde{S}_n)$ to the standard configuration in Figure \ref{fig:IntroTable}. Since Lagrangian 2-spheres in $Bl_4(\C\P^2)$ are unknotted \cite{LagSph4}, and unique in their homology classes (up to Hamiltonian isotopy), the two Lagrangian 2-spheres can also be Hamiltonian isotoped to the standard configuration. This yields the required full embeddings
$$\iota_n:E(1,l_n)\lr (Bl_4(\C\P^2),\Vol(E(1,l_n))\cdot\omega_\st).$$

This concludes Theorem \ref{thm:main}, where we have assumed Theorem \ref{thm:ExistenceTropicalSymplectic}. Let us now construct the symplectic-tropical curves needed for Theorem \ref{thm:ExistenceTropicalSymplectic}. This is the content of Section \ref{sec:Symp_Trop}.


\section{Symplectic Tropical Curves in Almost Toric Fibrations}\label{sec:Symp_Trop}

This section develops the construction of configurations of symplectic curves in terms of almost toric tropical diagrams. Theorem \ref{thm:ExistenceTropicalSymplectic} is the result from the present section used in the proof of Theorem \ref{thm:main}. In a nutshell, tropical curves in almost toric diagrams consist of two pieces: tropical curves in toric diagrams, as developed by G. Mikhalkin \cite{Mikhalkin1,Mikhalkin2}, and tropical local models near the cut singularities of the affine structure. These new tropical local models are discussed in this section. There are two useful perspectives: from the viewpoint of contact 3-manifolds or directly from the perspective of symplectic 4-manifolds; we will present the latter.\footnote{See a previous version of this manuscript on the arXiv for the former.}

\subsection{Symplectic-tropical curves} \label{ssec:Symp_Trop}

The central ingredient in Theorem \ref{thm:ExistenceTropicalSymplectic} and in our argument in Subsection \ref{ssec:proofmain} is the notion of a {\it symplectic-tropical curve} in an almost toric diagram, which we abbreviate {\it STC}. This is the content of the following definition.

Let $P_X \subset \R^2$ be an ATBD representing an ATF $\pi: X \lr B$ of a symplectic 4-manifold $X$. Let $\Gamma$ be an oriented graph, with edges decorated by primitive $\Z^2$ vectors and a multiplicity in $\Z_{> 0}$. Given an oriented edge $\gamma$ from a vertex $b$ to a vertex $c$, $b$ is said 
to be negative and $c$ positive, with respect to $\gamma$. 

\begin{definition} \label{dfn:Symp_comp_graph}
A symplectic-tropical curve $\SC : \Gamma \lr P_X$ is a $C^0$-embedding which satisfies the following conditions:  
  
\begin{enumerate}[label=(\roman*)]
  
\item Vertices are either univalent {\it (boundary)}, bivalent {\it (bending)}  or trivalent {\it 
(interior)}. The edges associated with boundary vertices shall be called {\it leaves};

\item All boundary vertices are negative;

\item Images under $\SC$ of boundary vertices are either on the boundary of 
the polytope $P_X$ or on a node. The images of bending vertices belong to the cuts (hence come endowed with an associated monodromy matrix). The images of interior vertices
belong to the complement of cuts, nodes, and boundary;

\item $\SC$ restricted to the (interior of) the edges is a $C^{\infty}$-embedding and 
tangent lines have lateral limits at each vertex, which are oriented according to the orientation
of the corresponding edge. We call a vector on this limit tangent lines a {\it limit} vector;

\item If $\bv$ is a positively oriented vector tangent to the image of an edge under $\SC$, with associated primitive 
vector $\bw \in \Z^2$, then $\langle \bv | \bw \rangle > 0$;

\item For a boundary vertex over the boundary of $P_X$, the primitive $\Z^2$-vector 
associated to its corresponding leaf must be orthogonal to the boundary
and pointing towards the interior of $P_X$. (The multiplicity of the edge can be arbitrary); 

\item For a boundary vertex over a node of $P_X$, the primitive $\Z^2$-vector
      $\bw$ associated to its corresponding leaf must be orthogonal to the cut, 
      and its orientation is determined by (ii) and (v). (The multiplicity can be arbitrary);
   
\item Let $\gamma_1$ and $\gamma_2$ be two edges meeting at a bending vertex,
      with monodromy matrix $M$. First, the bending vertex must be positive
      w.r.t. one edge and negative w.r.t. the other. Assume that we go
      counter-clockwise from $\SC(\gamma_1)$ to $\SC(\gamma_2)$. If $\bv$ is
      a positively oriented limit vector for $\SC(\gamma_1)$ at this bending
      vertex, then $M \bv$ is a positively oriented limit vector for
      $\SC(\gamma_2)$. Moreover, if $\bw$ is the primitive vector associated to
      $\gamma_1$, then the primitive vector associated to $\gamma_2$ must be
      $(M^T)^{-1} \bw$. In this case, the multiplicity of $\gamma_1$ and $\gamma_2$ must
      be the same;

\item For the three edges $\gamma_j$, $j = 1,2,3$, associated to an interior vertex $b$, 
with associated vectors $\bw_j$ and multiplicity $m_j$, the following balancing 
condition must be satisfied:

\begin{equation} \label{eq:BalCond}
  \varepsilon_1 m_1 \bw_1 + \varepsilon_2 m_2 \bw_2 + \varepsilon_3 m_3 \bw_3 = 
0, 
\end{equation} 
where $\varepsilon_j = \pm 1$ according to $b$ being positive or negative with 
respect to $\gamma_j$, $j = 1,2,3$.\hfill$\Box$     
\end{enumerate}
  \end{definition}

\begin{figure}[h!]   
  \begin{center} 
 \centerline{\includegraphics[scale= 0.6]{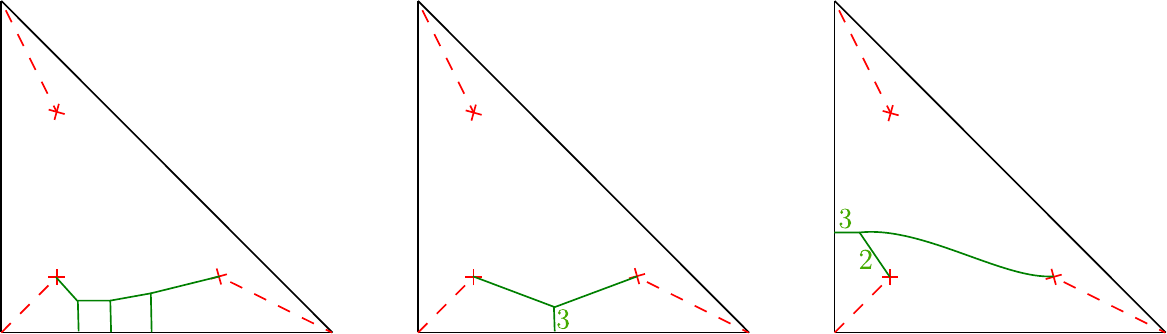}}

\caption{Examples of symplectic-tropical curves in an ATBD of $\CP^2$.}
\label{fig:Curves_CP2} 
\end{center} 
\end{figure}

There are many conditions in Definition
\ref{dfn:Symp_comp_graph}, but the idea is the following. We travel
along an edge with cycles, according to the multiplicity, represented by its
associated vector, as in Remark \ref{rmk:Symp_Curves}. Condition (v) will
guarantee that the surfaces are symplectic as we travel along the edges; condition (ix)
implies that the sum of these cycles arriving at a vertex are null-homologous, which will
allow us to glue them together (by a bounding null-homology); conditions (vi) and (vii) guarantee that the
corresponding cycles collapse as we arrive at an edge, or can collapse at the
nodal fiber; condition (viii) says that the map $\Gamma \to B$ given by
composing with the inverse of the homeomorphism $B \to P_X$, is actually
smooth at the bending vertex, and the associated vector changes accordingly. The remaining pieces of notation are required to write the balancing condition \eqref{eq:BalCond} in
a precise and consistent way.    

We represent a symplectic-tropical curve by just drawing its image $\SC(\Gamma)$, and labeling the multiplicity 
of each edge. The label is implicitly $1$ if the edge is drawn unlabeled. 
Conditions (vi)--(ix) determine the associated vectors and multiplicities, 
up to the ambiguity given by the sign and orientation of the edges
that are not leaves, which we basically ignore since they are
just an artifact to write (ix) consistently. For instance, Figure \ref{fig:Curves_CP2} represents 3 different symplectic
lines in an ATF of $\CP^2$. In the leftmost picture, 
the leaves are decorated with the associated vectors $(1,-1)$, $(0,1)$ and $(-1,-2)$, 
accordingly,
while the other two edges have labels $\pm (1,0)$ and $\pm (1,1)$, 
the signs being determined by (ix) and by how one decides to orient the edges. 

The core constructive result in this section, which justifies Definition \ref{dfn:Symp_comp_graph}, is the following:

\begin{thm} \label{prp:symp_trop}  Let $\SC : \Gamma \lr P_X$ be a symplectic-tropical curve as above, and $\cN\sse P_X$ a neighborhood of $\SC(\Gamma)$. Then there exists a closed symplectic curve\footnote{A {\it real} symplectic surface, i.e. two real dimensions.} $C$ embedded in $X$, projecting to $\cN$ under $\pi: X \lr B$. In addition, the intersection of $C$ with the anti-canonical divisor $K_X\sse X$
defined by the boundary of $P_X$ is 
given by the sum of the multiplicities of the corresponding boundary vertices.
\hfill$\Box$
\end{thm}
 
In this context, we say that the symplectic surface $C$ is represented by the symplectic-tropical curve $\SC (\Gamma)$. See \cite[Proposition~8.2]{Symington} for a discussion on the anti-canonical divisor $K_X\sse X$. In order to prove Theorem \ref{prp:symp_trop}, we need to construct local 
models near the nodes and the trivalent vertices. This is the content of the following subsections.
 
\begin{remark} We could consider embeddings of graphs with interior vertices having
valency greater than $3$, and write a rather involved definition for
symplectic-tropical curves. In this case, these graphs could be viewed as a
limit of graphs as in Definition \ref{dfn:Symp_comp_graph}, and a result similar to Theorem \ref{prp:symp_trop} holds.\hfill$\Box$
\end{remark}


\subsubsection{Local model near the nodes} \label{subsec:Local_Nodes}

The interesting case near the nodes is that we may arrive at a node with multiplicity $k\in\N$ greater than $1$.
 We shall now argue directly in the 4-dimensional symplectic domain and monitor the projections onto the ATBD. In principle, we would need to get $k$ disjoint capping 2-disks with boundary in $k$ copies of the collapsing cycle nearby a nodal singularity. That is not
possible if we force the boundary of these 2-disks to be entirely contained in a torus fibre. 
In consequence, for $k >1$, our 2-disks cannot project exactly over a segment under the projection $\pi$, 
but rather onto a 2-dimensional thickening of a segment in $P_X$.
 
Let us start with the local model for $k = 1$, where we are to collapse only one cycle through a symplectic 2-disk to the singular point. Consider the local model of a nodal fibre 
as in Definition \ref{dfn:ATF}, and the complex notation $\pi(x,y)=\overline{x}{y}$ for the almost toric fibration the map $\pi$. Choose $\varepsilon\in\R^{>0}$ sufficiently small. In this case, the 2-disk
$$\sigma_1 = (re^{i\theta},-ire^{i\theta}),\qquad 0 \le |r| \le \varepsilon,\quad\theta \in [0, 2\pi]$$
is a symplectic 2-disk with boundary
$c(\theta) = (\varepsilon e^{i\theta}, -i\varepsilon e^{i\theta})$, which is part of the 
symplectic line $y = -ix$ and projects via $\pi = \overline{x} y$ to the half-line $i\R_{\le 0}$. Let us now discuss the case of higher $k\in\N$.

Let us introduce a second symplectic 2-disk in this neighborhood, disjoint from 
the above, so that its boundary is isotopic, and arbitrarily close, to the collapsing boundary cycle 
$c(\theta)$ above. Fix a value $\delta_2\in\R^{>0}$, thought to be small with respect to $\varepsilon^2$, and 
a monotone non-increasing $C^{\infty}$-bump function $\Psi_2(s)$ such that
$$\Psi_2(s)\equiv\delta_2\mbox{ for }s\approx 0,\qquad \Psi_2(s)\equiv0\mbox{ for }s \ge 1,\qquad \Psi_2(s) > 0\mbox{ for }s \in [0, \varepsilon^2].$$

By taking $\delta_2$ sufficiently small, 
we can take $\Psi$ as $C^1$-close to $0$ as necessary. We
can then guarantee that the line $l_2 := \{y = -ix + \Psi_2(|x|^2)\}$ is still symplectic.
Then our second disk $\sigma_2$ is taken to be the intersection of the symplectic line $l_2$ with 
the half-space defined by $\im (\overline{x} y) \ge - \varepsilon^2$, where $\im$ denotes the imaginary part.
Note that, by construction, $\sigma_2 \cap \sigma_1 = \emptyset$.

Now, in order to then construct $k\in\N$ mutually disjoint symplectic 2-disks, 
with their boundary being isotopic and close to the collapsing boundary cycle 
$c(\theta)$, we only need to take $\delta_k > \delta_{k-1} > \cdots >
\delta_2$, $\delta_i\in\R^{>0}$, $2\leq i\leq k$, and the corresponding monotone non-increasing $C^{\infty}$-
bump functions $\Psi_k(s)> \Psi_{k-1}(s) > \cdots >
\Psi_2(s)$, all sufficiently $C^1$-close to $0$ so that 
the lines $l_j :=\{y = -ix + \Psi_j(|x|^2)\}$ are
symplectic for $j= 2, \dots, k$. The symplectic 2-disk $\sigma_j$ is
then taken to be the intersection 
of the symplectic line $l_j$ with the half-space $\im (\overline{x} y) \ge - \varepsilon^2$. 
See Figure \ref{fig:Loc_disks} (Left) for
a depiction of the projections $\pi(\sigma_j)$ of such symplectic 2-disks.

\begin{remark}\label{rmk:disjoint}
The equations for the symplectic lines $\l_j$ imply
  that each 2-disk $\sigma_j$, $1\leq j\leq k$, intersects the Lagrangian plane $x=y$ at a point where $x = -ix + \Psi_j(|x|^2)$, since by Pythagoras Theorem, we need 
  $\Psi_j(|x|^2) = \sqrt{2}|x|$, which occurs since we have chosen $\delta_j\ll 
  \varepsilon^2$.\hfill$\Box$
\end{remark}

Note that the boundary $\del \sigma_j$ projects under $\pi$ to a segment $I_j$ normal to the half-line
$i\R_{\le 0}$. Now, using the fact that the lines $l_j$ are isotopic to the line $y = -ix$, which
topologically self-intersects once (e.g. relative to the boundary at infinity),
we can deduce that the boundary $\del \sigma_j$ links $\del \sigma_1$ exactly once in the thickened
annulus $\pi^{-1}(I_k)$, as illustrated in the rightmost picture of Figure \ref{fig:Loc_disks}. 

\begin{figure}[h!]   
  \begin{center} 
\begin{minipage}{0.4\textwidth}
\centerline{\includegraphics[scale=0.6]{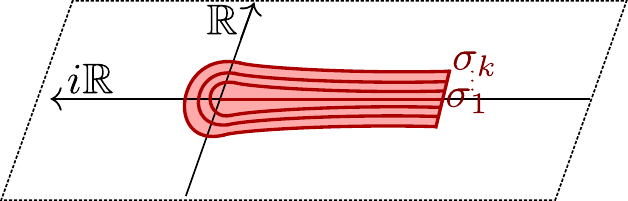}}
\end{minipage}
\hspace{0.5cm}
\begin{minipage}{0.4\textwidth}
\centerline{\includegraphics[scale=0.6]{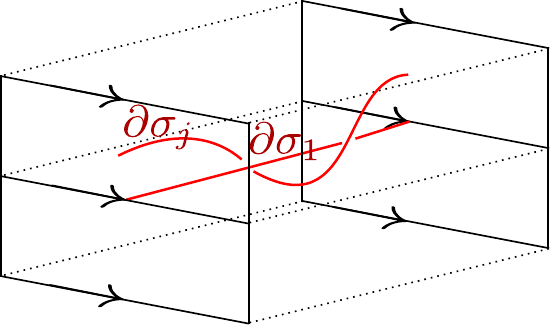}}
\end{minipage}

\caption{The left picture illustrates the projection of the $\sigma_j$ disks. The right picture 
illustrates how  $\del \sigma_j$ and $\del \sigma_1$ are linked in the thickened
annulus $\pi^{-1}(I_k)$. }
\label{fig:Loc_disks} 
\end{center} 
\end{figure}

In proving Theorem \ref{prp:symp_trop}, we may assume that the image of the leaves arrive
at a node in a segment, which we can identify in our local model with a segment 
in $i \R_{\le 0}$. The above discussion is then summarized in the following proposition:

\begin{proposition} \label{prp:Local_node}{\it  Let $\pi:X \lr B$ be an ATF, represented by
an ATBD $P_X$, where we homeomorphically identify $B$ with $P_X$. Let $\gamma$
be a leaf of a symplectic-tropical curve
$\SC: \Gamma \lr P_X$, with boundary vertex over a node, multiplicity $k \in \Z_{>0}$, and $\cN$ a neighborhood of $\SC(\gamma)\sse B$. Fix a point $p \in \SC(\gamma)$ close to the node with collapsing 
cycle $\alpha \subset \pi^{-1}(p)$. 

Then we can associate $k$ disjoint symplectic 2-disks 
$\sigma_j$, $1\leq j\leq k$, such that $\pi(\sigma_j) \subset \cN$, 
with their boundary $\del \sigma_j$ arbitrarily close 
to $\del \sigma_1 = \alpha$.}
\end{proposition}


\subsubsection{Local modal near interior vertices} \label{subsec:Local_Int}

Condition (v) in Definition \ref{dfn:Symp_comp_graph} and Remark \ref{rmk:Symp_Curves} allows us to transport the boundary of the symplectic 2-disk
$\sigma_1$ in Proposition \ref{prp:Local_node} along the image of the
corresponding leaf until it is close to an interior node. (Note that if we hit a bending
vertex, condition (viii) guarantees that we can keep moving the same
cycle $\del \sigma_1$, whose class in the first homology of the fibre is then
represented by a different $\Z^2$ vector according to the monodromy.) In addition, the boundaries $\del \sigma_j$ project to a segment normal to the leaf and, because of its
closeness to $\del \sigma_1$, we can also transport it using cycles
projecting under $\pi$ to small segments normal to the leaf. In line with Remark \ref{rmk:Symp_Curves}, we can ensure that these surfaces remain symplectic.   

At this stage, we need to construct a local model for a symplectic surface near a trivalent vertex
$b$, using the data of (ix), that can be made to project to a given
neighborhood of $\SC(b)$. Moreover, this model needs to glue with 
prescribed incoming cycles. For that, a first option is to rely on the article \cite{Mi04a}, where G. Mikhalkin uses O. Viro's patchworking
ideas to construct families of hypersurfaces in $(\C^*)^n$, whose amoebae
converge to a given tropical hypersurface in $\R^n$, see
\cite[Remark~5.2]{Mi04a}. In that manuscript, G. Mikhalkin actually views $(\C^*)^n$ as the open stratum of a
closed toric variety, in particular having finite volume, so we can
symplectically assume that it is indeed a local model. Given the nature
of Theorem \ref{prp:symp_trop}, we can assume that 
the symplectic-tropical curve arrives at the interior vertex in a tropical way,
i.e., locally as segments $\bbv_j = \bbw_j$ satisfying the balancing condition
\eqref{eq:BalCond} of Definition \ref{dfn:Symp_comp_graph}.

\begin{figure}[h!]   
  \begin{center} 
 \centerline{\includegraphics[scale= 0.6]{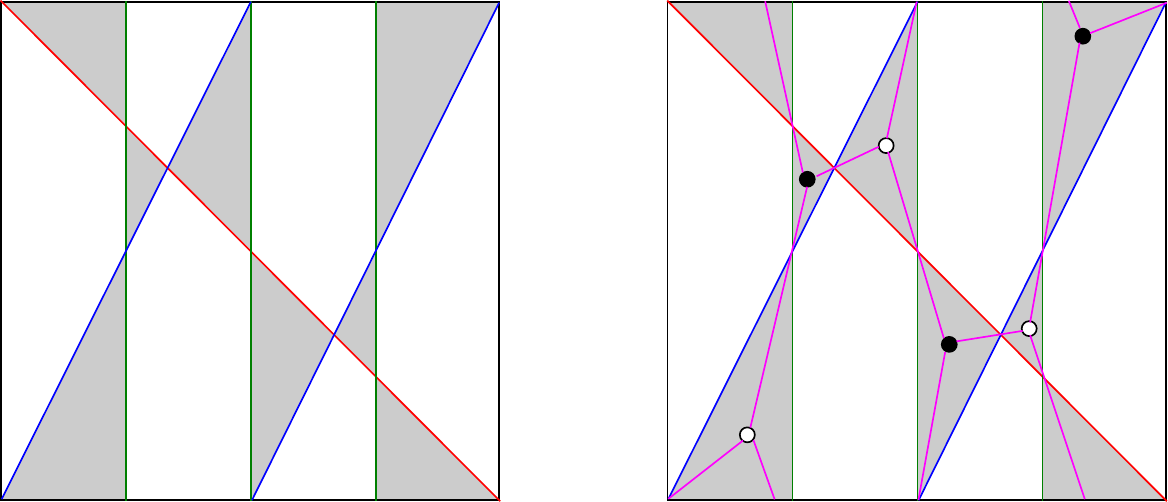}}

\caption{Collection of cycles in $T^2$ satisfying the balancing condition \eqref{eq:BalCond} and its associated dimer. These cycles 
correspond to the ones arriving at the interior vertex of the symplectic-tropical curve in the middle diagram
of Figure \ref{fig:Curves_CP2}. It represents the coamoeba in $\cB$ of the symplectic surface we will construct.}
\label{fig:Dimer} 
\end{center} 
\end{figure}

A second option, independent of \cite{Mi04a}, is using the following explicit local model. Over a small disk $\cB\sse B$
centered at an interior vertex $b \in B$, we have $\pi^{-1}(\cB) \cong \cB \times
T^2$, and the fibration is given by the projection onto the first factor. The projection of the
symplectic curve to the $T^2$ factor is known as the coamoeba of that curve. We
aim at first constructing what will be the coamoeba out of the balancing
condition \eqref{eq:BalCond} and -- out of that data -- then building our local model for the
symplectic curve. We want the surface to be so that its boundary projects to
straight cycles, having only double crossings. Moreover, away from the pre-image
of the double crossings, the rest of the surface will project injectively into
polygons divided by the straight cycles. The homology classes of the boundary
are represented by $m_j$ disjoint copies of $\varepsilon_j\bw_j$, $j = 1,2,3$,
where $\bw_j$ are the vectors associated with the interior vertex $b$ and
hence satisfy the balancing condition \eqref{eq:BalCond} from Definition
\ref{dfn:Symp_comp_graph} (ix). Figure \ref{fig:Dimer} (Left) illustrates the 
coamoeba of the local model we will build for the neighborhood of the interior 
vertex in the ATBD from Figure \ref{fig:Curves_CP2} (Center). The balancing condition 
associated with the interior vertex is $ (1,-1) + 3(0,1) + (-1,-2) = 0$.

The existence of the above mentioned configuration for the coamoeba is equivalent to the 
existence of a \emph{dimer model}\footnote{A dimer model \cite{Kas67} is a bipartite graph, and we name 
	half of the vertices black and the other half white, so edges connect a black 
	vertex to a white one.} embedded in $T^2$. We label each convex polygon of our coamoeba
black or white, where a black polygon can only share a vertex with a 
white one. The vertices of the dimer model are then placed in the interior
of the polygons according to their colours, and for each intersection
of the boundary cycles we associate an edge, projecting inside the coamoeba. See Figure \ref{fig:Dimer} (Right). 
Then the straight cycles are taken to be the collection 
of zigzag paths associated to the dimer. For concepts related to dimer models, including zigzag paths, and its
relationship with coamoebas, we refer the reader to
the recent works \cite{FHKV08,Fo19,Gul08,Hi19}. Thus, from the discussion above, our aim is then to provide a dimer model with prescribed set of homology classes for its
collection of zigzag paths. Moreover, we want the zigzag paths to be straight. This is achieved in Proposition \ref{prp:Dimer3} below.

\begin{remark} If one does not require the paths to be straight, then the article \cite[Section~6]{Gul08} constructs
an algorithm to build a dimer model out of the prescribed classes for the
collection of zigzag paths. (Again, with non-straight cycles.) Now, note that in \cite[Example~4.1]{Fo19}, an example of a collection of $5$
classes in $H_1(T^2)$ that cannot be realized by straight cycles, that
are the zigzag paths of a dimer model, is given. Nevertheless, in case the collection 
of classes are given by copies of only $3$ primitive classes, 
one can in fact construct a dimer model with straight set of zigzag paths,
as shown in the upcoming Proposition \ref{prp:Dimer3}. We believe this is likely well-known
to experts but we did not find references in the literature.\hfill$\Box$
\end{remark}

\begin{figure}[h!]   
  \begin{center} 
 \centerline{\includegraphics[scale= 0.6]{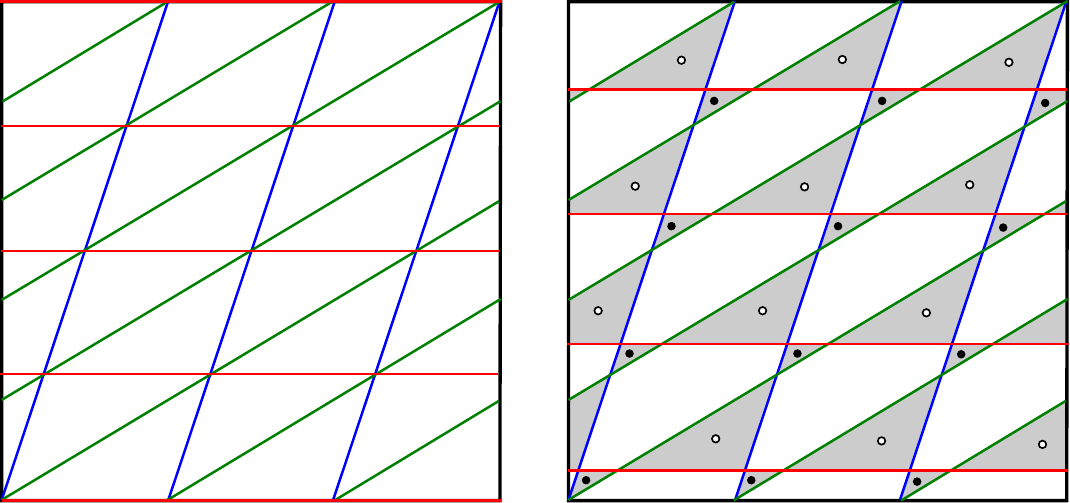}}

\caption{Construction of a dimer model with the prescribed collection 
$4(1,0) + (1,3) + (-5,-3) = 0$ of zigzag paths.}
\label{fig:Dimer5} 
\end{center} 
\end{figure}

\begin{proposition} \label{prp:Dimer3}
  Given $\bbw_1,\bbw_2,\bbw_3 \in H_1(T^2;\Z)$, primitive classes satisfying 
  $$m_1\bbw_1 + m_2\bbw_2 +m_3\bbw_3 = 0,\qquad m_1,m_2,m_3\in\N,$$
  there exists a dimer model 
  in $T^2$ with zigzag paths realised by $m_j$ straight lines in
  the classes $\bbw_j$, $j= 1,2,3$. Moreover, the components containing the vertices
  of the dimer can be taken to be triangles.  
\end{proposition}

\begin{proof} Figure \ref{fig:Dimer5} essentially provides a proof by drawing. In detail, let us assume, without loss of generality, that $\bbw_1 = (1,0)$, and let $\bbw_2 =
(a_2,b_2)$, $\bbw_3 = (a_3,b_3)$. Take the straight cycle $[0,1]\times\{0\}$ in
$T^2 = \R^2/\Z^2$, in the class $\bbw_1 = (1,0)$, and consider $|m_2b_2| =
|m_3b_3|$ equidistant points in $[0,1]$. For $j = 2,3$, take $m_j$ straight lines
with slope $\bbw_j$, passing through the first $m_j$ points in the segment $[0,1]$,
and consider them as cycles in $T^2 = \R^2/\Z^2$.  
It can be seen that these lines intersect at heights which are multiples of $1/m_1 \mod \Z$.  
For $k = 0, \dots, m_1 - 1$, take the cycles  $[0,1]\times\{k/m_1\}$, 
each of them intersecting the other straight lines in $|m_2b_2| =
|m_3b_3|$ triple points. In that way, $T^2 $ is divided into triangles, 
as illustrated in Figure \ref{fig:Dimer5} (Left). Then, moving the 
horizontal cycles slightly up, we build the required dimer, as shown in Figure 
\ref{fig:Dimer5}.
 \end{proof}

From this dimer model, constructed from the data of the balancing condition
\eqref{eq:BalCond}, we will now build a smooth surface in $I \times T^2$,
where $I = [-\epsilon,\epsilon]$, with boundary $m_j$ copies of cycles in class $\bbw_j =
\varepsilon_j\bw_j$ in $H_1(I \times T^2; \Z) \cong H_1(T^2; \Z)$, living
at heights $-\epsilon, 0, \epsilon$ for $j = 2,1,3$, respectively.

\begin{figure}[h!]   
  \begin{center} 
 \centerline{\includegraphics[scale= 0.8]{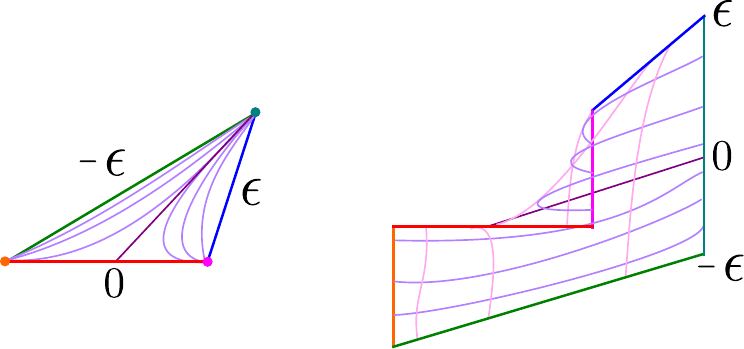}}

\caption{The smooth surface in $I \times T^2$ associated to one of the triangles
in the dimer model of Proposition \ref{prp:Dimer3}. The coamoeba in the 2-torus $T^2$ is depicted on the left, with the described foliation associated to $\bbf$.}
\label{fig:SmoothSurf} 
\end{center} 
\end{figure} 
 
Consider the congruent white triangles and, in each, take a segment from its
vertex given by the intersection of the cycles in the classes $\bbw_2$ and
$\bbw_3$, to the midpoint of the opposite edge. We name the primitive
direction of this segment $\bbf = (\alpha, \beta)$, and abuse notation by calling $\bbf$ the segment itself. Consider then a smooth foliation of the
triangle minus the vertices, so that following it gives an isotopy from the edge
in the $\bw_2$ cycle to half of the edge in the $\bbw_1$ cycle union the segment
$\bbf$, and then from the other half of the edge in the $\bbw_1$ cycle union the
segment $\bbf$ to the the edge in the $\bbw_3$ cycle, as illustrated in Figure \ref{fig:SmoothSurf} (Left). Considering each leaf of the
foliation as level sets of a smooth function $\rho_2$ from the triangle minus
vertices to $I$, its graph embeds into $I \times T^2$, as in Figure \ref{fig:SmoothSurf} (Right). Now, taking a symmetric version
of the foliation and function $\rho_2$ on the black triangles of the dimer model, this ensures that the compactification of the union of the graphs is a smooth surface
in $I \times T^2$, with the desired boundaries.

Note that this embedding can be made symplectic into the region $\mB \times T^2$, 
the symplectic neighborhood of the pre-image of the interior vertex $b$, by embedding the segment $I$ into $\mB$ in an appropriate direction. Nonetheless, 
we can also build our surface to project on the $(p_1,p_2)$ coordinates on $\mB$, much like in the work of Mikhalkin \cite{Mi04b,Mi04a}. 

\begin{figure}[h!]   
  \begin{center} 
 \centerline{\includegraphics[scale= 1.2]{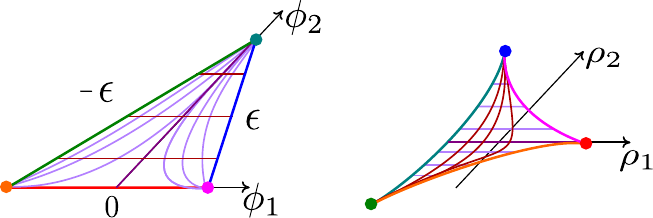}}

\caption{The depiction of a triangular piece of the coamoeba (left) and of the amoeba (right) of a symplectic surface near
the interior vertex.}
\label{fig:Amoeba} 
\end{center} 
\end{figure} 

\begin{remark} For convenience, we shall now perform a change of local coordinates, from $(p_1, \theta_1, p_2, \theta_2)$ 
to new coordinates $(\rho_1, \phi_1, \rho_2, \phi_2)$, where the symplectic form will be
$$dp_1\wedge d\theta_1 + dp_2\wedge d\theta_2 =
d\rho_1\wedge d\phi_1 + d\rho_2\wedge d\phi_2.$$ 
This change of the basis is performed in the $T^2$ chart, going from $\be_1=(1,0)$,
and $\be_2 = (0,1)$, to $\be_1$ and $\bbf = (\alpha, \beta)$. 
This corresponds to setting $\phi_1 = \theta_1 - 
(\alpha/\beta)\cdot\theta_2$, and $\phi_2 = \theta_2/\beta$, 
so $\theta_1 \be_1 + \theta_2 \be_2 = \phi_1 \be_1 + \phi_2 \bbf$. Setting
$\rho_1 = p_1$ and $\rho_2 = \alpha p_1 + \beta p_2$ ensures that  
$dp_1\wedge d\theta_1 + dp_2\wedge d\theta_2 =
d\rho_1\wedge d\phi_1 + d\rho_2\wedge d\phi_2$.
These coordinates $(\rho_1, \phi_1, \rho_2, \phi_2)$ are used in the following proposition.\hfill$\Box$
\end{remark}

We are now ready to construct all the required local models for the symplectic surfaces in Theorem \ref{prp:symp_trop}. (These local models are used in Subsection \ref{subsec:Conn_local_models}.) This is the content of the following proposition:

\begin{proposition} \label{prp:local interior}{\it
 Let $\SC: \Gamma \lr P_X$ be a symplectic-tropical curve in an ATF
 $\pi:X \lr B$, represented by
an ATBD $P_X$, $b$ be an interior vertex and $\mB$ a small 
disk centered at $b$, whose boundary intersects $\SC(\Gamma)$ in
the points $p_1$, $p_2$, $p_3$. Let $\bw_j$ be the associated 
vector corresponding to the edge containing $p_j$, for $1\leq j\leq 3$,
hence satisfying the balancing condition \eqref{eq:BalCond}.

Then there exists a symplectic curve in $X$, projecting to $\mB$,
whose boundary projects to the points $p_j$, and represents 
cycles whose classes in $H_1(\pi^{-1}(p_j);\Z)$ are 
given by $\bbw_j = \varepsilon_j\bw_j$, $1\leq j\leq 3$.}
\end{proposition}

\begin{proof}
  
  Figure \ref{fig:Amoeba} illustrates how the embedding will look like in each of the (white) triangles, by
  describing the amoeba and coamoeba together. The interior vertex corresponds to $(\rho_1,\rho_2) = (0,0)$,
  and we vary $\rho_1$ in a small enough interval $[-\delta, \delta]$. First, we associate 
  to the coamoeba in Figure \ref{fig:Amoeba} (left) a smooth surface inside $I\times T^2$, as explained before. (Recall Figure \ref{fig:SmoothSurf}.) This smooth surface will be the projection of the corresponding piece of our symplectic surface into the $(\rho_2,\phi_1,\phi_2)$
  coordinates. Hence, our foliation in the triangle corresponds to the level sets of the coordinate $\rho_2$,
  i.e. $\rho_2$ constant. It is left to us to determine the $\rho_1$ coordinate of each point,  ensuring the symplectic condition. Each edge of the triangle in the coamoeba has constant $(\rho_1,\rho_2)$ coordinates
  corresponding to the vertices of the amoeba. Using the same name for the edges as for their homology
  classes, we choose $\rho_1$ so that $\rho_1(\bbw_3) = \rho_1(\bbw_2) < 0 \le \rho_1(\bbw_1)$. The $\rho_1$
  coordinate decreases as we move along the $\rho_2$ level sets from bottom to top. Also, as we vary $\phi_1$
  positively in the horizontal segments ($\phi_2$ constant) in the triangle, we choose $\rho_1$ so that its
  variation is non-negative for $ \rho_2 \le 0$ and non-positive for $\rho_2 \ge 0$; also the
  $\rho_2$-coordinate varies positively. We also note that the segments in our smooth surface that projected
  to the vertices of the triangle in $T^2$, will project to the boundary of the amoeba in $\mB$.
  
 Now, in order to check the symplectic condition for the above surface, we probe with the following paths $\xi^1$, $\xi^2$. The path $\xi^1$ is given by following the horizontal segment ($\phi_2$ constant)
 in the coamoeba, and $\xi^2$ is given by following the $\rho_2$ level sets
 in the coamoeba (as depicted in Figure \ref{fig:Amoeba} and oriented towards the positive direction of $\phi_2$). Denote by $d\rho_i^j$ and $d\phi_i^j$ the coordinates
 of $d\rho_i$ and $d\phi_i$, along the $j$-th path, for $i,j =1,2$.
 We have that,  
 
 \[ d\rho_1^1 = \begin{cases} \ge 0 &  \text{if} \ \rho_2 \le 0 \\
  \le 0 &  \text{if} \ \rho_2 \ge 0 \end{cases}  \ ,  \ d\phi_1^1 \ge 0\ , \ d\rho_2^1 > 0\ , \  d\phi_2^1 = 0 \ , \]

Note that we oriented the path $\xi^2$ so that $ d\phi_2^2 > 0$. For $\rho_2 < 0$, the path $\xi^2$  goes from the left-most (orange) vertex, to the 
  top (teal) vertex, and hence $d\phi_1^2 > 0$ there. For the $\rho_2 > 0$ part, the path goes
 from the right-most (pink) vertex, to the top (teal) vertex, and thus $d\phi_1^2 < 0$. Note also that the curve in the amoeba picture have $d\rho_1^2 < 0$. In conclusion, we obtain:
    
   \[d\rho_1^2 < 0\ ,  \ d\phi_1^2 = \begin{cases} \ge 0 &  \text{if} \ \rho_2 \le 0 \\
  \le 0 &  \text{if} \ \rho_2 \ge 0 \end{cases} \ , \ d\rho_2^2 = 0\ , \  d\phi_2^2 > 0\ . \]

  In particular, we get that:
  
 \begin{equation} \label{eq:Symp_Amoeba}
    \omega (\del \xi^1, \del \xi^2) = \begin{vmatrix} 
 d\rho_1^1 & d\phi_1^1 \\ d\rho_1^2 & d\phi_1^2 \end{vmatrix}
 +\begin{vmatrix} 
 d\rho_2^1 & d\phi_2^1 \\ d\rho_2^2 & d\phi_2^2 \end{vmatrix} > 0
 \end{equation}
 
as desired, where $\del \xi^i$ is the tangent vector to the $\xi^i$ curve. An analogous embedding is defined for the black triangles, with the
same amoeba image on the $(\rho_1,\rho_2)$ projection, and we obtain a smooth
embedding of the surface as required. 
  
   \begin{rmk} \label{rmk:Straight_Amoeba}
Note that we could simply take $\rho_1 \equiv 0$ over the surface, and it would 
still satisfy \eqref{eq:Symp_Amoeba}. Hence we can indeed get a symplectic 
embedding of our surface over a fixed interval in $\mB$. 
   \end{rmk} 
  
Finally, notice that we chose our amoeba (with small $\delta$) so that the tangency
 of the amoeba at the vertices are far from being orthogonal to the
 corresponding $\bbw_j$ direction. Let $\mH_j$ be the half-plane whose boundary
 line is normal to $\bbw_j$, passes through $b$, and contains the vertex of the
 amoeba corresponding to $\bbw_j$. The symplectic conditions
 \eqref{eq:Symp_Amoeba} and $(v)$ from Definition \ref{dfn:Symp_comp_graph},
 ensure that the path in $\SC(\Gamma)$ associated to $\bbw_j$ is in $\mH_j$.
Hence, we can connect this path to the corresponding $p_j$ via a path still satisfying 
the condition of never being orthogonal to $\bbw_j$.
\end{proof}


\subsubsection{Connecting the local models} \label{subsec:Conn_local_models}
In the subsection we conclude Theorem \ref{prp:symp_trop} by gluing together the local models provided in Proposition \ref{prp:Local_node} and Proposition \ref{prp:local interior} above.

The local model at an interior vertex
gives us $m_j$ copies of the boundary corresponding to $\bw_j$, equidistant
in our parametrisation of $T^2$. We now follow the path that passes through $p_j$,
as in Proposition \ref{prp:local interior}:

\begin{itemize}
	\item[-] If it hits another interior vertex, we just follow these constant cycles and the surfaces glue naturally.
	\item[-] If it hits a boundary vertex over the boundary of $P_X$, we build the surface	again by keeping the cycles constant, and the cycles smoothly collapses to a point since near the boundary of $P_X$ we have a toric model.
\end{itemize}
 
Now, due to the nature of the boundaries resulting from Proposition \ref{prp:Local_node},  we need to do additional work in order to glue to the local models given by Proposition \ref{prp:Local_node}, as follows (recall Figure \ref{fig:Loc_disks}).

 \begin{figure}[h!]   
  \begin{center} 
 \centerline{\includegraphics[scale= 0.6]{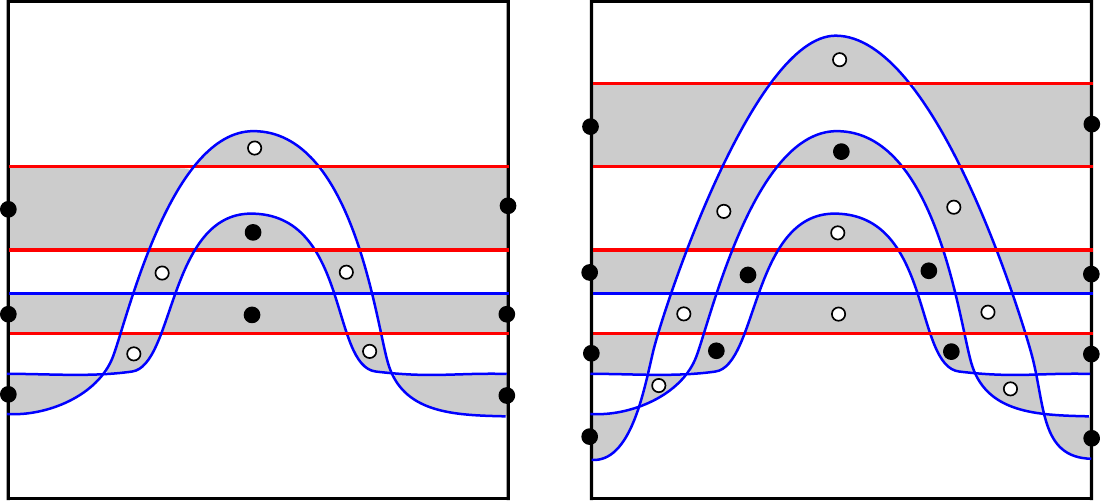}}

\caption{Dimer models for connecting the boundaries of the disks, for the local model of a boundary vertex at a node, to straight cycles.}
\label{fig:Conn_Dimer} 
\end{center} 
\end{figure} 

\begin{figure}[h!]   
  \begin{center} 
 \centerline{\includegraphics[scale= 0.85]{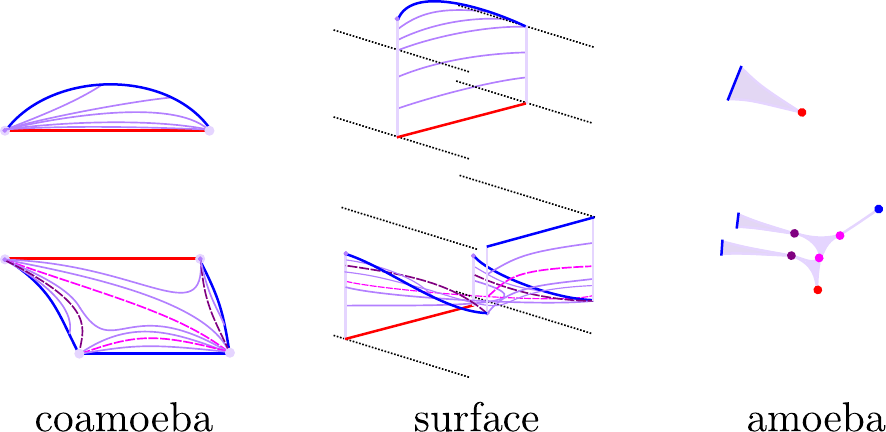}}

\caption{Coamoebas, surfaces and amoebas corresponding to different pieces of the dimer model 
illustrated in Figure \ref{fig:Conn_Dimer}. The 4-gon is subdivided into two 3-gons and two bigons 
indicated by the dashed curves.}
\label{fig:VariousSurf} 
\end{center} 
\end{figure} 

To connect $k$ straight cycles, which we color in red, to $k$ curling cycles, which we color blue -- see again Figure \ref{fig:Loc_disks} -- we consider dimers as the ones
illustrated in Figure \ref{fig:Conn_Dimer}. The figures depicts the cases $k = 3$ and $4$, and readily generalizes for any $k \ge 1$. Independent of $k$, the components
containing the vertices of these dimer models are either a bi-gon, a $3$-gon or
a $4$-gon. The analysis for getting a symplectic embedding for each of these
pieces, especially the $3$-gons and the bi-gons, is similar to the one we made
in the proof of Proposition \ref{prp:local interior}, and we employ the
analogous terminology now.

\begin{remark} The main difference is that the $\rho_2$
coordinate at most of the pieces of the blue cycles is not constant, but this
was not crucial for getting the symplectic curve in Proposition \ref{prp:local
interior}. This makes the boundary of the surface project into segments, rather
than points.\hfill$\Box$
\end{remark}

To replicate the analysis in the proof of Proposition
\ref{prp:local interior} to the $4$-gons, it is better to subdivide it into two 3-gons and two bigons. Figure \ref{fig:VariousSurf} illustrates an example of a coamoeba (with some $\rho_2$ level sets depicted), the corresponding piece of
the smooth surface, and the amoeba. Remark \ref{rmk:Straight_Amoeba}
still holds and we could actually view the amoeba picture as straight segments, 
but the blue boundaries would still project to sub-segments.     

Now, we do not know that the boundaries $\del \sigma_j$, who live in the product of an
interval with $T^2$, project exactly in the pattern as depicted in Figure
\ref{fig:Conn_Dimer}, for $2\leq j\leq k$. The only information we have is
that they are mutually linked and that the projections intersect each horizontal
cycle generically twice. Nevertheless, this is in fact not a problem, since for each one of them we
can draw the actual projection and the corresponding generic curve as in Figure
\ref{fig:Conn_Dimer}.

These curves will generate a dimer, either 
one annulus or two bi-gons. Hence, we can connect them with a smooth symplectic 
surface. The condition $\Psi_k(s)> \Psi_{k-1}(s) > \cdots >
\Psi_2(s)$ we required in the discussion before the statement of 
Proposition \ref{prp:Local_node}, ensures that we can place these surfaces
inside mutually disjoint thickened tori for $2\leq j\leq k$.
This explains how to patch the boundaries $\del \sigma_j$ to straight
cycles for our $(\theta_1,\theta_2)$ coordinates on the torus. 

Finally, the cycles built in Proposition \ref{prp:local interior} are
equidistant in our coordinates for $T^2$, but the straight cycles we just built
are sufficiently close to $\del \sigma_1$. (In this metric sense, Figure
\ref{fig:Conn_Dimer} is misleading for visual purposes.) Using the notation of
Proposition \ref{prp:local interior}, we promptly see that this is not a
problem, since we can move apart these close cycles -- thought to be given by
$\phi_2$ constant -- until their projections to $T^2$ become equidistant, by moving
only in the corresponding $\rho_2$ direction. The symplectic condition
\eqref{eq:Symp_Amoeba} is readily checked, by taking the first curve the
horizontal ones in the coamoeba part, and the second curves vertical ones, so
$d\rho_1^1 = d\rho_1^2 = 0$, $d\rho_2^1 = d\rho_2^2 > 0$, $d\phi_2^2 = 0$. 

 All these connecting surfaces can be made to project into sufficiently small regions in the almost toric fibrations, in particular inside the neighborhood $\mN$
 of our given symplectic-tropical curve $\SC(\Gamma)$. This concludes the proof of Theorem \ref{prp:symp_trop}. \qed

From now on, we will ease notation by calling $\SC$ the symplectic-tropical curve obtained in Theorem \ref{prp:symp_trop} 
from $\SC: \Gamma \to P_X$. At this stage, Theorem \ref{prp:symp_trop} allows us to construct symplectic surfaces $C(\SC)\sse X$ associated to symplectic-tropical curves $\SC\sse B$, for an almost toric fibration $\pi:X\lr B$. The upcoming Subsection \ref{subsec:Combinatorics} shall now address the general combinatorics appearing in the ATBD associated to Del Pezzo surfaces, which are crucial for the construction of the required symplectic-tropical curves used in our proof of Theorem \ref{thm:main} in Section \ref{sec:proofmain}.


\subsection{Combinatorial background for triangular shaped ATFs} \label{subsec:Combinatorics}

As in \cite{Vianna3} consider the ATBD of triangular shape for a Del Pezzo
surface containing the monotone Lagrangian torus $\Tpqr$ as a visible fibre,
i.e., not inside a cut. This ATBD is related to the Markov type equation:

\begin{equation} \label{eq:Markov}
  n_1p^2 + n_2q^2 + n_3r^2 = G pqr
\end{equation}

where $G = \sqrt{dn_1n_2n_3}$ and $d$ is the degree of the corresponding Del Pezzo. These equations yield the Diophantine equations in Subsection \ref{ssec:ArithmeticSymington}. In \cite{Vianna3}, it is shown that $n_1p^2$, $n_2q^2$ and $n_3r^2$ correspond to the 
determinant of the primitive vectors associated with the corners of the
corresponding ATBD. 

\begin{figure}[h!]   
  \begin{center} 
 \centerline{\includegraphics[scale= 0.8]{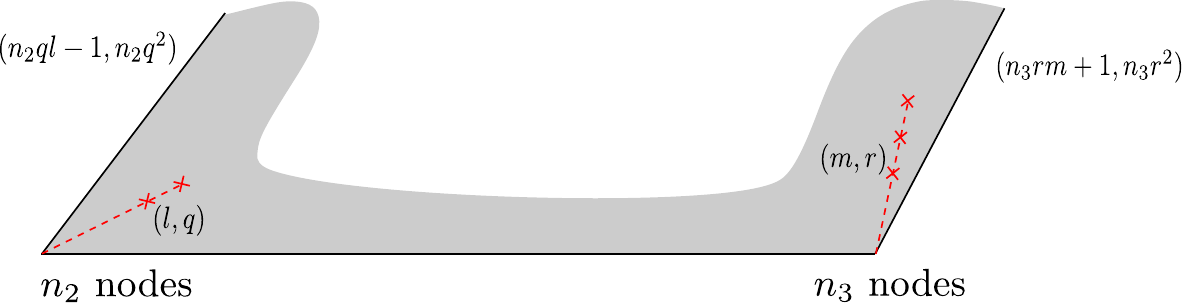}}

\caption{Corners of ATBDs}
\label{fig:CornATBD} 
\end{center} 
\end{figure}

Forcing the common edge of the corners corresponding to $n_2q^2$ and $n_3r^2$ to be 
horizontal, we get that the cuts and the primitive vectors of the remaining 
edges are as illustrated in Figure \ref{fig:CornATBD}, where we name
$(l,q)$ and $(m,r)$ the direction of the cuts -- compare with 
\cite[Figure~13]{Symington} and \cite[Proposition~2.2, Figure~2]{Vianna2}. The condition that the third determinant is $n_1p^2$ then becomes:

\begin{equation} \label{eq:det}
   \begin{vmatrix} n_2ql -1 & n_3rm +1 \\ n_2q^2 & n_3r^2 \end{vmatrix} = n_2n_3r^2ql - n_3r^2 - 
     n_2n_3q^2rm - n_2q^2 = n_1p^2
\end{equation}
This yields the equalities
\[n_2n_3r^2ql - n_2n_3q^2rm = n_3r^2 + n_2q^2 + n_1p^2 = Gpqr \] 
and upon dividing by $n_2n_3qr$, we obtain

\begin{equation} \label{eq:lm}
  rl - qm = \frac{Gp}{n_2n_3}.
\end{equation}

Equation \eqref{eq:lm} is used to build symplectic-tropical curves in the 
$\varepsilon$-neighborhood depicted in Figure \ref{fig:CornATBD}. 


\subsubsection{Symplectic-Tropical Curves in the Edge Neighborhood $\fN$} \label{subsubsec:Neigh_Edge}

Let $\fB$ be a neighborhood of an edge in an ATBD, containing its associated cuts, as
illustrated in Figure \ref{fig:CornATBD}. Let $\fN$ denote the preimage of $\fB$
in $X$. By studying the combinatorics of the ATF, we show that we can construct
a symplectic-tropical curve (Definition \ref{dfn:Symp_comp_graph}) inside
$\fB$, and by Theorem \ref{prp:symp_trop}, there is a corresponding
symplectic curve in $\fN$. These symplectic curves will have the same
intersection with the anti-canonical divisor as the rational curves highlighted
in Figure \ref{fig:IntroTable}. Their homology classes can differ from the rational curves
listed in $\SH$ (Figure \ref{fig:IntroTable}), by the classes of the Lagrangian spheres projecting in
between two nodes inside $\fN$. In order to obtain them in the desired homology
class, we will need to modify our curves in the pre-image of the neighborhood
of the cuts, containing the Lagrangian spheres. This correction is done in the next
Subsection \ref{subsec:def_symptrop_curves}.

The collapsing cycle corresponding to the node associated with the cut $(l,q)$,
respectively $(m,r)$, is represented by the orthogonal vector $(q,-l)$,
respectively $(-r, m)$. Consistent with Definition \ref{dfn:Symp_comp_graph}
(vii), consider a symplectic-tropical curve with: one interior vertex; 
one leaf going towards one of the nodes with $(q,-l)$ collapsing cycle, 
with multiplicity $r$; another leaf going towards one of the nodes with $(-r, m)$ collapsing cycle,
with multiplicity $q$; and the third going towards the bottom edge, with multiplicity
$\frac{Gp}{n_2n_3}$.  
 
These choices satisfy the balancing condition \eqref{eq:BalCond} of Definition
\ref{dfn:Symp_comp_graph} (ix), since from \eqref{eq:lm}:

\begin{equation} \label{eq:Trop_line}
  r(q,-l) + q(-r,m) + \frac{Gp}{n_2n_3}(0,1) = (0,0).
\end{equation}

Hence we get a symplectic-tropical curve in $\fN$ by Theorem \ref{prp:symp_trop}.

Let us understand the behaviour of a family of such curves as we 
deform our Del Pezzo surface towards the corresponding limit orbifold.
From the proof of \cite[Theorem 4.5]{Vianna3}, and considering both the limit orbifold 
and the limit orbiline $\mA$ corresponding to the limit of the horizontal line in Figure 
\ref{fig:CornATBD}, we get that the intersection of $\mA$ with the anticanonical 
divisor $[\mA] + [\mB] + [\mC]$, where $[\mB]$ and $[\mC]$ are the classes of the other
orbilines, is:

\begin{equation} \label{eq:AntiCan_int}
    [\mA] \cdot ([\mA] + [\mB] + [\mC]) = \frac{n_1p^2(n_3r^2 + n_2q^2 + n_1p^2)}{n_1n_2n_3p^2q^2r^2} 
    = \frac{Gp}{n_2n_3qr}    
\end{equation}

 Note that $qr$ times this number is the one found in Equation \eqref{eq:lm}. Hence, 
 the symplectic-tropical curve we construct limits to an orbi-curve in the 
 class $qr\mA$. In particular, for the symplectic-tropical curve 
 to be a smoothing of that orbiline in the class $\mA$, we must have
 $q=1$ and $r=1$. 
 
 
 \subsubsection{Required Symplectic-Tropical Curves in the $\fN$ for Theorem \ref{thm:main}}\label{subsec:RequiredSTC}

Consider the triangular-shaped ATFs of $\CP^2$, $\PxP$, $\BlIII$ and $\BlIV$, 
\cite{Vianna3}
with one smooth corner, i.e., satisfying a Markov type equation of the form

\begin{equation} \label{eq:Markov_1corner}
  1 + n_2q^2 + n_3r^2 = Gqr.
\end{equation}

Let $\fN$ be a neighborhood of the edge opposite the frozen smooth vertex, introduced in Section \ref{subsubsec:Neigh_Edge}. The results of 
the previous subsections yield the following

\begin{thm}\label{wSTCs}
  The homology class of the symplectic divisors highlighted in Figure \ref{fig:IntroTable} can be realized as a symplectic-tropical curve in $\fN$.  
\end{thm}

\begin{proof}
  First note that for the equations associated with $\CP^2$, $\PxP$, $\BlIII$ and $\BlIV$, 
   the quantity $G/n_2n_3$ is respectively, $3$, $2$, $1$, $1$, which is the intersection 
   number of the symplectic divisors in the corresponding spaces with the 
   anti-canonical divisor. Thus, Theorem \ref{prp:symp_trop} and Section \ref{subsubsec:Neigh_Edge}
suffice for the case of $\CP^2$. For the remaining cases, 
   one needs to ensure the correct intersection with the Lagrangian 
   spheres. This can be achieved case by case using that 
   $$1 + n_2q^2 \equiv 0 \mod n_3,\qquad 1 + n_3r^2 \equiv 0 \mod n_2.$$
   For instance, in the case of $\BlIV$, we get 
   $q^2 \equiv -1 \mod 5$, so $q \equiv 2$, or $3 \mod 5$. Take the divisor of Figure \ref{fig:IntroTable} that 
   intersects two Lagrangian spheres. Following the mutations 
   in \cite[Figure~17]{Vianna3}, at the triangular-shape ATBD \cite[Figure~17 ($A_4$)]{Vianna3}, 
   these spheres become the top and bottom spheres, of the 4-chain of Lagrangian spheres.
   If $q \equiv 2 \mod 5$, apply Proposition \ref{prp:Disjoint_S_i} taking the 2
   intersections to the top and bottom node. If $q \equiv 3 \mod 5$, we arrive 
   at the three interior nodes instead, having the same intersection number with 
   the Lagrangian spheres. One can check that the sign is correct by looking at 
   the first instance when $q=2$ and observing that the mutation 
   $q \to 5r - q$ switches between $q \equiv 2 \mod 5$ and $ q \equiv 3 \mod 5$ 
   (the first case $q =2$ is depicted as $A_1$ in Figure \ref{fig:PxP_3Blup}).\end{proof}

 

\subsection{Deforming Symplectic-Tropical Curves} \label{subsec:def_symptrop_curves}

The previous subsection and Theorem \ref{prp:symp_trop} provides a symplectic
sphere in the desired homology class, for each of the examples. Nevertheless, the actual intersection number with the
Lagrangian 2-spheres, or between two representatives, is off in the majority of the cases\footnote{Namely, there is an excess
of pairs of intersections with opposite signs.}. In this section we further deform the symplectic-tropical
curves to take into account the required intersection with the Lagrangian 2-spheres. In the next section we develop technical
results to deal with intersections of chains of symplectic spheres, keeping the same intersection with the
Lagrangian spheres.

\color{black}
So, we need to be able to have control of the intersection number of our
symplectic-tropical curve $\SC$ with the Lagrangian 2-spheres that appear naturally for a pair of nodes lying
inside the same cut. For that, let us prove Proposition \ref{prp:Disjoint_S_i}, with the following notation.

Let $\SC: \Gamma \lr P_X$ be a symplectic-tropical curve in an ATF
of $X$, represented by the ATBD $P_X$ and $\mN$ a neighborhood of
$\SC(\Gamma)$, as before. Consider a class of $n$ nodes of the ATBD inside the
same cut, so that at least one of the nodes is in $\SC(\Gamma)$ and let $\mM$
be a neighborhood of the cut. Let $S_1, \dots, S_{n-1}$ be Lagrangian spheres
projecting inside the cut to consecutive segments between the nodes. Also name
$S_0 = S_n =\emptyset$. Let $m$ be the sum of the multiplicities of the leaves
arriving at these $n$ nodes, and choose $d < n$ nodes, where $m \equiv d \mod n$.

\begin{proposition} \label{prp:Disjoint_S_i} {\it 
There is a symplectic curve projecting to $\mN \cup \mM$ with the property
that its intersection with $S_1 \cup \dots \cup S_{n-1}$ consists of exactly $d$ points, each projecting to one of the
$d$ chosen nodes.}\hfill$\Box$
\end{proposition}
 \color{black}
The reminder of this subsection is devoted to the proof of Proposition \ref{prp:Disjoint_S_i}. The idea is that, since $m -d = kn$, we can place disks $\sigma_j$ as in Proposition 
\ref{prp:Local_node}, for $2\leq j\leq k+1$ at all the $n$ nodes.                                                    
Naming $n_1, \dots, n_n$ be the $n$ nodes in a cut, in order, and
 $S_i$ a Lagrangian sphere projecting to the cut between the consecutive nodes
 $n_i$, $n_{i+1}$, we see that the signed intersection of $S_i$ with the
 collection of $kn$ disks is zero, since the sphere $S_i$ intersects the $k$ disks only around the $n_i$ and $n_{i+1}$ nodes, with opposite signs. Thus, it is clear that, 
 at least smoothly, we can pairwise cancel these intersections. Nevertheless, before that, we need
 to connect the disks.   
 
Let us focus now on how to construct in Lemma \ref{lem:Sigma} a symplectic pair of pants $\mP$ that we can glue to the boundary of a pair of symplectic 2-disks $\sigma_j$ near the $n_i$ and $n_{i+1}$ nodes.
 
\begin{figure}[h!]   
  \begin{center} 
 \centerline{\includegraphics[scale= 0.8]{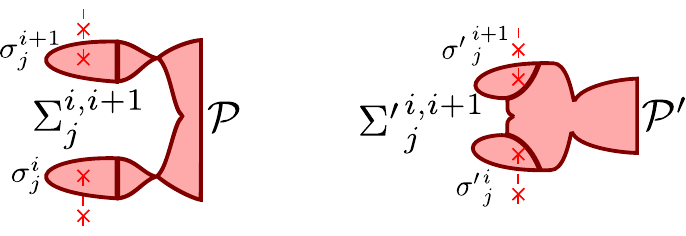}}

\caption{Amoebas of the surfaces $\Sigma_j^{i,i+1}$, ${\Sigma'}_j^{i,i+1}$.}
\label{fig:Sigma^i,i+1} 
\end{center} 
\end{figure} 

\begin{lem} \label{lem:Sigma}
There is a symplectic disk $\Sigma_j^{i,i+1}$ in the cut neighborhood $\mM$, containing the two symplectic 2-disks $\sigma^i_j$ 
  and $\sigma^{i+1}_j$ as a subset, with boundary in a thickened 2-torus $\mathbb{T}$, and whose class is twice the collapsing class of the nodes
  via the identification $H_1(\mathbb{T};\Z) \cong H_1(T^2;\Z)$.\\
  
  Denote by $\mP$ the pair of pants 
  which is the closure of $\Sigma_j^{i,i+1} \setminus(\sigma^i_j \cup \sigma^{i+1}_j)$.\hfill$\Box$\end{lem}

Figure \ref{fig:Sigma^i,i+1} shows on the left the amoeba corresponding to the $\Sigma_j^{i,i+1}$ surface. 
The idea is to then isotope it to a different surface  ${\Sigma'}_j^{i,i+1}$ with boundary still in 
$I\times T^2$, such that the intersection with the Lagrangian 2-sphere $S_i$ is empty.  


\begin{figure}[h!]   
  \begin{center} 
 \begin{minipage}{0.6\textwidth}
 \centerline{ \hspace{0.3cm} \includegraphics[scale= 0.7]{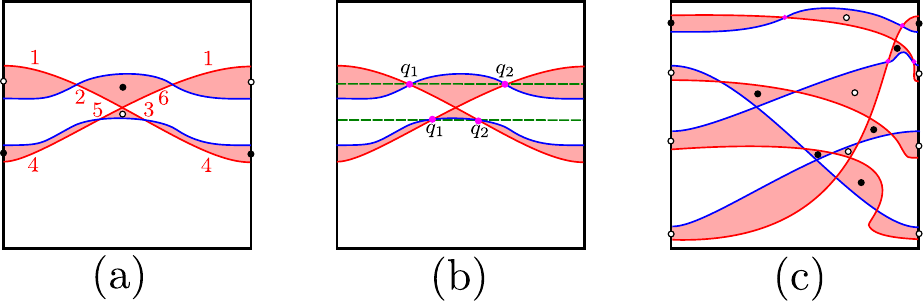}}
\end{minipage}
\hspace{0.4cm}
\begin{minipage}{0.3\textwidth}
\centerline{\hspace{0.8cm} \includegraphics[scale=0.7]{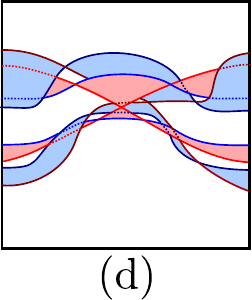}}
\end{minipage}

\caption{(a) A dimer model for the symplectic pair of pants $\mP$.
(b) Dimer model for $\mP'$, illustrating the cycles in the Lagrangian 
sphere $S_i$, over the intesection between the boundary of the amoeba of $\mP'$
and the amoeba of $S_i$ (a segment between the nodes). (c)
The dimer model for the part of the surface ${\Sigma'}_j^{i,i+3}$ connecting the 
boundary of the $\sigma_j$ disk in the
$n_{i+3}$ node with the boundary of ${\Sigma'}_j^{i,i+2}$. (d) Coameaba of 2 non-intersecting 
$\mP'$s. The suggested height is with respect to the $p_2$ coordinate.}
\label{fig:Dimer_Sigma} 
\end{center} 
\end{figure}

\begin{proof}[Proof of Lemma \ref{lem:Sigma}]
  Let us assume the cut is vertical, and thus the collapsing cycle is given by a 
  $\theta_2$-constant curve in the $(\theta_1,\theta_2)$ coordinates of $T^2$.  
  Making a change of action-angle coordinates, if needed, we may assume that 
  the collapsing cycles corresponding to the $n_i$ nodes are slightly phased-out. We thus draw the $T^2$ projection of the boundary of the 2-disks $\sigma_j$, 
  recalling that each of them links once the horizontal cycle in a thickened torus that they 
  live in (recall the construction before Proposition \ref{prp:Local_node}). 
  We color these boundaries blue, and draw a dimer model that indicates the 
  coamoeba of the pair of pants $\mP$ we will construct, as illustrated in the
  first picture of Figure \ref{fig:Dimer_Sigma}. We will color the
  other boundary of $\mP$ red. We number the components of the red curve
  in the dimer model of Figure \ref{fig:Dimer_Sigma} (Left)
  as indicated, and sketch the profile of its $p_2$ coordinate as indicated in Figure
  \ref{fig:p2_profile}.  
  
  The curves $\del \sigma_j$ can be taken sufficiently close
  to the collapsing cycle $\del \sigma_1$, so we assume its $\theta_2$ variation is
  small enough with respect to the difference $p_1$ between the coordinate 
  of the red curve and the $p_1$ coordinate of the blue curve. (Essentially,
  we take them small with respect to the size of $\mM$.) 
  
  \begin{figure}[h!]   
  \begin{center} 
 \centerline{\includegraphics[scale= 0.6]{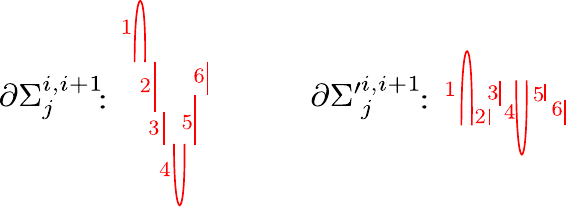}}

\caption{The profile of the $p_2$ coordinate as we move along $\del \Sigma_j^{i,i+1}$, $\del {\Sigma'}_j^{i,i+1}$.
The numbering corresponds to the one on the leftmost picture of Figure \ref{fig:Dimer_Sigma}.}
\label{fig:p2_profile} 
\end{center} 
\end{figure}
   
The dimer model we are considering consists of two bi-gons and two 
  tri-gons. We carefully analyze the $(p_1,p_2)$-coordinates we associate 
  to these pieces, to ensure that we get a symplectic pair of paints $\mP$
  connecting the blue boundaries to the red one. The corresponding amoebas, 
  together with the $\theta_1$- and $\theta_2$-level sets, are indicated in Figure \ref{fig:Def_Amoeba} (Middle). There is a curve in the 
  amoeba corresponding to the vertices of the bi-gon that will be crucial 
  in the further analysis. This curve shall be named the
  {\it pink} curve, and in this case it is a horizontal segment.

  \begin{figure}[h!]   
  \begin{center} 
 \centerline{\includegraphics[scale= 0.6]{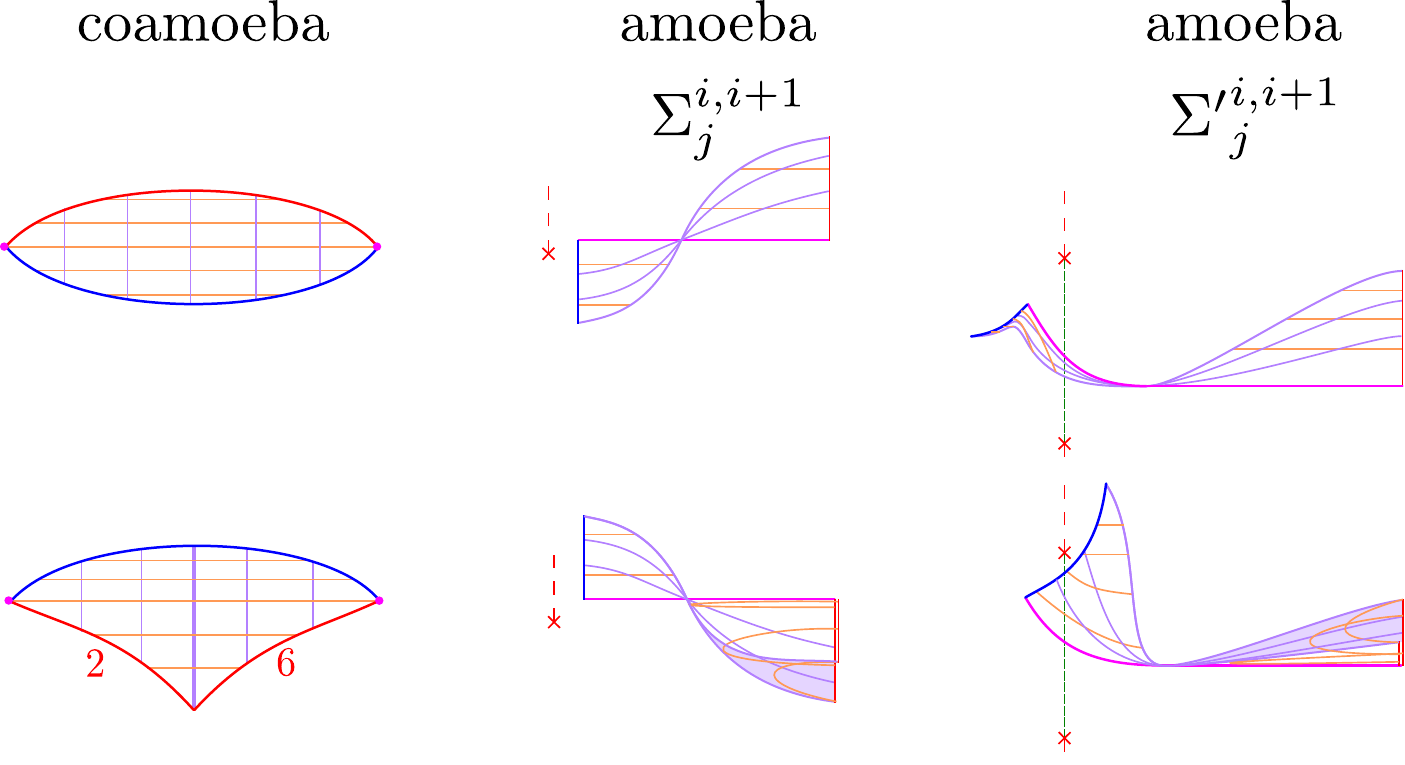}}

\caption{Amoebas of the  pieces of the pair of pants $\mP$ and $\mP'$
for the surfaces $\Sigma_j^{i,i+1}$, ${\Sigma'}_j^{i,i+1}$. }
\label{fig:Def_Amoeba} 
\end{center} 
\end{figure}
  
  Following analogous conventions for the analysis of amoebas and coamoebas as in
  the proof of Proposition \ref{prp:local interior}, with $\rho_i$ corresponding
  to $p_i$ and $\phi_i$ to $\theta_i$, we take the first curve $\xi^1$
  corresponding to a $\theta_2$-level set, oriented with $d\theta_1^1 >0$ and
  the second curve $\xi^2$ corresponding to a $\theta_1$-level set, with the
  orientation such that the $p_1$ coordinate is decreasing, i.e., $dp_1^2 < 0$. 

  First, let us study the bi-gon pieces. We assume that the $\theta_2$ variation is sufficiently small for the red curve as well, so if we fix the norm of $dp^2$ as we
  travel, we then have that $d\theta_2^2$ is small enough away from the vertices. As
  indicated in Figure \ref{fig:Def_Amoeba}, we choose $\xi^1$ so that $dp_2^1 =
  0$ and $d\theta_2^1 = 0$. As we move along $\xi^1$ at the bottom of the
  bi-gon, the $p_1$ coordinate increases up to the middle of the horizontal
  curve in the amoeba and decreases from the middle until the end. The logic is reversed as we move to the top of the bi-gon. So, $\omega( \del \xi^1, \del
  \xi^2) = - d \theta_1^1 dp_1^2 > 0$ (recall equation \eqref{eq:Symp_Amoeba}),
  as desired. The limit case where $\theta_2$ attains the maximum and the
  minimum in the bi-gon can be analyzed by replacing the $\xi^1$ curve by the
  respective red and blue curves. Since it also corresponds to maximum and
  minimum of the $p_2$ coordinate, we still have $dp_2^1 = 0$ and $d\theta_2^1 =
  0$, ensuring positivity of $\omega( \del \xi^1, \del \xi^2)$. 
   
  Second, we now turn our attention to the tri-gon, as depicted in Figure
  \ref{fig:Def_Amoeba}. We first notice that the $p_2$ coordinates associated
  with the bottom vertex of the tri-gon, must be different for the edge labeled
  $2$ and the edge labeled $6$, the latter being greater than the former, recall
  Figures \ref{fig:Dimer_Sigma} and \ref{fig:p2_profile}. This implies that 
  we must not only consider one curve with that corresponding $\theta_2$-coordinate constant, 
  but rather a family of curves. The shaded region in the bottom-middle picture
  of Figure \ref{fig:Def_Amoeba} indicates the image of these curves under
  the $(p_1,p_2)$-projection. The analysis at the top of this part of 
  the coamoeba is similar to before, with $dp_2^1 =
  0$ and $d\theta_2^1 = 0$ and $\omega(
  \del \xi^1, \del \xi^2) = - d \theta_1^1 dp_1^2 > 0$. (Note that we do not remain 
  stationary with $d\theta_1^1 = 0$ at the top part of the coamoeba.) 
  At the bottom part, we do have $dp_2^1 > 0$ as $\xi^1$ travels 
  from the $2$ curve to the $6$ curve and we also have $d\theta_2^2 > 0$, 
  as we are traveling from the red curve to the blue curve along $\xi^2$, 
  with $dp_1^2 < 0$. That way, since $d\theta_1^1 \ge 0$, we have
  
  \[\omega(
  \del \xi^1, \del \xi^2) = - d \theta_1^1 dp_1^2 + d \theta_2^2 dp_2^1 > 0.\]
   
  Gluing analogous models for the bottom part of the diagram in a 
  consistent  way provides our symplectic pair of pants.\end{proof}

We now move towards the construction of the surface ${\Sigma'}_j^{i,i+1}$ that
does not intersect the Lagrangian sphere $S_i$. Given our choice of coordinates,
the Lagrangian sphere projects to a vertical segment in the $(p_1,p_2)$
factor. In order to keep working with the same coordinates in the region we want
to do the modification of our curve, we modify the ATBD, without changing our
ATF, simply by reversing the direction of the cuts associated with the nodes
$n_{i+1}, \dots, n_n$ (this operation was named \emph{transferring the cut} in
\cite{Vianna2}). Note that we already made this operation in Figures
\ref{fig:Sigma^i,i+1}, \ref{fig:Def_Amoeba}. The rightmost diagrams of
Figure \ref{fig:Def_Amoeba} represent the projection of $S_i$ by a vertical dashed segment,
which is different than the one we use to represent the cuts. Let $(p_1,p_2) = (0,0)$ be the midpoint of this dashed segment.

From Remark \ref{rmk:disjoint}, we see that $\sigma^i_j$ and $\sigma^{i+1}_j$
both intersect $S_i$ once, and with opposite signs. We
consider disks ${\sigma'}^i_j$ and ${\sigma'}^{i+1}_j$, obtained from $\sigma^i_j$
and $\sigma^{i+1}_j$ by carving out a neighborhood of the intersection point with
$S_i$, with their $(p_1,p_2)$ projection as illustrated in Figure \ref{fig:Sigma^i,i+1} (Right). We keep coloring their boundary blue, and their
coamoeba projection is similar to the ones we just analyzed, as it is sufficiently close to
the corresponding collapsing cycle with constant $\theta_2$-coordinate. Thus, we
can also build a dimer model as in Figure
\ref{fig:Dimer_Sigma} (Left), to build a pair of paints $\mP'$, and we also color the
other boundary of $\mP'$ red.

Recall that we named a pink curve, and the $(p_1,p_2)$ image of the vertices corresponding to the intersection
of the blue and red curves in the coamoeba coordinates $(\theta_1,\theta_2)$, which we will refer to as pink
vertices. The pink curves and the pink vertices will play the following role in our construction. As before,
we focus on the top part of the coamoeba, the bottom part being symmetric under the reflection around the $p_2
= 0$ coordinate (or at least having a symmetric behaviour, since our coamoeba picture is not symmetric.) Then
the pink curve corresponding to the top bi-gon will be the graph of a non-increasing convex function
$p_2(p_1)$, starting at a point with $p_1 <0$, $p_2 >0$, becoming negative before $p_1$ becomes $0$, and
eventually becoming constant at some point where $p_1>0$. Let us call $x$ the endpoint in the $p_2$-constant
segment, with smallest $p_1$-coordinate. Since the pink curve of the bottom bi-gon is the
reflection around the $p_2 = 0$ of the pink curve for the top bi-gon, these curves will intersect at some point with
negative $p_1$ coordinate. Also, the pink curve of the bottom bi-gon ends at a segment with positive constant
$p_2$ coordinate. For that to happen, the $p_2$-coordinate of the red curve needs to move
different as we move along the different parts of the coamoeba, as indicated in Figure \ref{fig:p2_profile}. Note that, as illustrated in Figure \ref{fig:p2_profile}, we maintain the property
that the $p_2$-coordinate at the common point of the segments labeled $2$ and $3$ is smaller than the one
corresponding to the segments $5$ and $6$. We also need to ensure the following. Consider the point of
intersection between the $(p_1,p_2)$ projection of the top blue curve and the projection of $S_i$. Look at the
$\theta_2$-coordinate of the circle of $S_i$ over this point, and consider its intersections $q_1$, $q_2$ with
the blue curve. We draw the coamoeba profile of the red curve so that $q_1$, $q_2$ are precisely the pink
vertices, as illustrated in Figure \ref{fig:Dimer_Sigma} (Center). We make analogous choices for drawing the
red curves at the bottom of the coamoeba.

 We are now in shape to prove the following lemma.

\begin{lem} \label{lem:Sigma_prime}There is a symplectic disk ${\Sigma'}_j^{i,i+1}$ in the cut neighborhood $\mM$,
disjoint from the Lagrangian 2-sphere $S_i$, containing the symplectic 2-disks ${\sigma'}^i_j$ and ${\sigma'}^{i+1}_j$ as a subset, with boundary in a thickened torus $\mathbb{T}$, whose homology class is twice the collapsing class of the nodes via the identification $H_1(\mathbb{T};\Z) \cong
H_1(T^2;\Z)$.

Denote by $\mP'$ the pair of pants which is the closure of
${\Sigma'}_j^{i,i+1} \setminus({\sigma'}^i_j \cup {\sigma'}^{i+1}_j)$. 
\end{lem}

\begin{proof} As in the previous proof, we will draw in the $(p_1,p_2)$-coordinates the level sets of the $\theta_2$ and $\theta_1$ coordinates, naming
the former $\xi^1$ and the latter $\xi^2$. Let us start looking at the bi-gon, and
describe the projection of the $\xi^1$ curve in the amoeba. Each curve starting
close to the $\theta_2$ minimum, up to a certain height $b$ (to be specified),
will have a horizontal projection, with the maximum of the $p_1$ coordinate
corresponding to half of the $\theta_1$ coordinate of $\xi^1$. This last
property is preserved, even when we start at a height greater than $b$, but
then, the image becomes a graph of a non-increasing function $p_2(p_1)$,
eventually limiting to the part of the pink curve that stops at $x$, as
illustrated in the top-right picture of Figure \ref{fig:Def_Amoeba}. 
We assume that the derivative is smaller in norm than the derivative of the
graph giving the pink curve. For the top
part of the coamoeba, the $\xi^1$ curves have constant $p_2$ coordinates, with
the minimum of the $p_1$ coordinate attained at the middle of the $\theta_1$
coordinate. The amoeba projection of the $\xi^2$ curves are also indicated in the top-right picture of Figure \ref{fig:Def_Amoeba}. In particular, analyzing the
symplectic condition for the top of the coamoeba part, is essentially done as in the proof of Lemma \ref{lem:Sigma}, compare the top-middle and top-right
pictures of Figure \ref{fig:Def_Amoeba}.

To ensure the symplectic condition at
the bottom of the coamoeba, we need to carefully choose the point $b$, recalling
that we can choose the red curve so that the $\theta_2$ variation is small enough
compared with the $\theta_1$ variation. Away form the points in $\xi^1$ with
maximum $p_1$ coordinate, let us move along the $\xi^i$ curves with the normalised
condition $|dp_1^i| = 1$. The variation $d\theta_2^2$ will then be 
bounded by an extremely small constant  (w.r.t. the $\theta_1$ diameter of the coamoeba). We take a constant $b$, large enough to ensure that
$\omega(\del \xi^1, \del \xi^2) = d\theta_1^1 -  dp_2^1 d\theta_2^2 > 0$,
recalling that we forced $dp_2^1$ to be zero for points at $\theta_2$
heights smaller than $b$ and $|dp_2^1|$ is bounded by the maximum slope
of the pink curve. In the points on $\xi^1$ with
maximal $p_1$ coordinate, we simply have $dp_2^1 = 0$ and 
$\omega(\del \xi^1, \del \xi^2) = - d\theta_1^1 dp_1^2  > 0$.
We let the reader check the positivity for the limiting points at the pink 
curve.

Let us now move to the tri-gon part of the coamoeba. For the top part we choose a
height $c$ in the coamoeba, playing a similar role as $b$ in the above
paragraph. If the $\theta_2$ coordinate of $\xi^1$ is bigger than $c$, we take
the amoeba part to have constant $p_2$ coordinate. So if the $\theta_2$
coordinate is not smaller than $c$, we have $dp_2^1 = 0$ and $\omega(\del \xi^1,
\del \xi^2) = - d\theta_1^1 dp_1^2 > 0$. The analysis regarding the symplectic
condition is done as in the last paragraph for the part corresponding to the
$\theta_2$ coordinate smaller than $c$. For the bottom part, 
recall that the $p_2$ coordinate corresponding to the part $2$ of the 
coamoeba of the red curve is smaller than the one corresponding to the 
part $6$. So we can take $dp^1_2 > 0$. The analysis now is similar to the 
analogous part in the proof of Lemma \ref{lem:Sigma}. We have
$d\theta^1_1 \ge 0$ (being zero at the middle of the $\xi^1$ curves
whose $\theta_2$ coordinate is not bigger than the pink vertices), 
$dp_2^1 > 0$, $d\theta^1_2 = 0$, $dp^2_1 < 0$, $d\theta_1^2 = 0$,
$d\theta^2_2 > 0$, since we move on $\xi^2$ from the red curve
to the blue curve.
  
The fact that we chose the pink curve to cross $p_1 = 0$ with 
negative $p_2$ coordinate promptly ensures that the 
coamoeba region of $S_i$ corresponding to the segment given 
by the intersection of $p_1 = 0$ with the amoeba of the bi-gon,
does not intersect the bi-gon itself. For the tri-gon part, 
looking at the amoeba projection, we see that $S_i$ does
not intersect the region of the surface corresponding to the
bottom of the tri-gon. The fact that we chose the 
coamoeba of our red curve so that the pink vertices
coincide with the points $q_1$, $q_2$, ensures 
that the part of the surface whose coamoeba corresponds
to the top of the tri-gon does not intersect the Lagrangian 
$S_i$ as well, as required.

\end{proof}

\begin{rmk} \label{rmk:Sigma_isotopy}
  There is a smooth way to isotope the boundary of $\mP$ to the boundary
  of $\mP'$, and following that, a smooth way to isotope the amoebas
  of $\Sigma_j^{i,i+1}$ to the ones of $\left(\Sigma_j^{i,i+1}\right)'$, see again 
  Figure \ref{fig:Def_Amoeba}, as well as the small differences on the coamoebas. 
  Hence, we can isotope from $\Sigma_j^{i,i+1}$ to $\left(\Sigma_j^{i,i+1}\right)'$, with 
  their boundaries restricted to a thickened torus $\mathbb{T}=I\times T^2$.\hfill$\Box$
\end{rmk}

Now we see that the symplectic surface ${\Sigma'}_j^{i,i+1}$ intersects the Lagrangian 2-sphere $S_{i+1}$ once
in a point belonging to ${\sigma'}_j^{i+1}$. In an analogous fashion to Lemma \ref{lem:Sigma_prime}, we can
chop out of ${\Sigma'}_j^{i,i+1}$ the intersection point with $S_{i+1}$ and then glue its boundary, the
boundary of ${\sigma'}_j^{i+2}$ (where we remove the intersection of $\sigma_j^{i+2}$ with $S_{i+1}$), and the
boundary of a newly chosen ``red curve'' in a thickened torus that has three times the homology of the
collapsing cycle in this thickened torus, with a new pair of pants. Naming the former two curves blue, the
first step would be to construct a dimer model between the blue curves and the red curve, so that the
behaviour of this new surface obtained from the dimer model, on the region where it could
intersect $S_{i+1}$, is the same as the one analyzed in Lemma \ref{lem:Sigma_prime}. Denote this
surface by ${\Sigma'}_j^{i,i+2}$. We can iterate this process and consider symplectic surfaces
${\Sigma'}_j^{i,i+k}$ in the cut neighborhood $\mM$, that do not intersect $S_i, \dots, S_{i+k-1}$, and has
boundary on a thickened torus, whose homology class is $k+1$ times the collapsing cycle. This process leads to
the following:

\begin{lem} \label{lem:Sigma_1n}
  There exists a symplectic disk ${\Sigma'}_j^{1,n}$ inside $\mM$, not intersecting the Lagrangian set $\bigcup_{i=1}^{n-1} S_i$, and whose boundary lies on a thickened 
  torus $I \times T^2$, with boundary class being $n$ times the class of the collapsing cycle.
\end{lem}

\begin{proof} We build this surface iteratively as indicated above, starting with
${\Sigma'}_j^{1,2}$. The algorithm to build the red curve and the corresponding
dimer model is illustrated for going from ${\Sigma'}_j^{1,3} \cup \sigma_j^4$ to
${\Sigma'}_j^{1,4}$ in the third diagram of Figure \ref{fig:Dimer_Sigma}. We chose
the red curve to pass through consecutive chambers of the complement of the curves
given by the coamoeba projection of the blue curves, which we recall is the
boundary of the disconnected surface obtained by chopping off the intersections
of ${\Sigma'}_j^{1,i} \cup \sigma_j^{i+1}$ with $S_i$. We do it so that the top
part of the dimer, corresponding to one bi-gon and two tri-gons, has the same
configuration as in Lemma \ref{lem:Sigma_prime}, including the points analogous
to $q_1,q_2$. Recalling Remark \ref{rmk:Sigma_isotopy}, we can think that we
first glue a pair of pants $\mP$ as before to the boundaries of
${\Sigma'}_j^{1,i} \cup \sigma_j^{i+1}$ and then isotope to our desired surface
${\Sigma'}_j^{1,i +1}$, with the modifications happening in the same framework as
in the proof of Lemma \ref{lem:Sigma_prime}. 
\end{proof}

  \begin{figure}[h!]   
  \begin{center} 
 \centerline{\includegraphics[scale= 0.6]{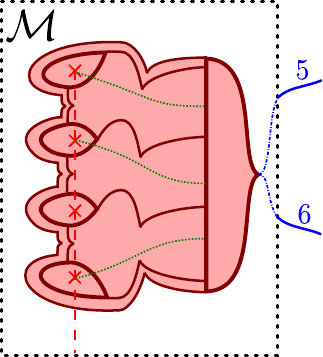}}

\caption{Final amoeba for the symplectic disk in the case of a symplectic-tropical curve
arriving at $n = 4$ nodes with multiplicity $ m = 5 + 6 = 11$, so $k=2$ and $d = 3$.
For visual purposes, we illustrate the amoeba of $k=2$ surfaces as if one 
envelopes the other, while in fact both have the same amoebas.}
\label{fig:Sigma_final_amoeba} 
\end{center} 
\end{figure}

Now we should inductively build the surfaces ${\Sigma'}_j^{1,n}$, for $j$ going
from $1$ to $k$, making sure these surfaces do not intersect. Recall that the boundaries 
of $\sigma_j$ and $\sigma_{j+1}$ are linked in the thickened 3-dimensional neighborhood. We can achieve non-intersection
by adjusting the crossings of blue and red curves between different amoebas, 
and the $p_2$ coordinate. Thus we can ensure the required non-intersection just by looking at 
the $(\theta_1, \theta_2, p_2)$ projection of the surfaces. The case $n =2$ is
illustrated in Figure \ref{fig:Dimer_Sigma} (d).  
We can then get the other $d$
disks, carrying the collapsing cycle from the respective nodes to the boundary of
the same thickened torus. They project to curves inside the amoeba of the
${\Sigma'}_j^{1,n}$ disks, see Figure \ref{fig:Sigma_final_amoeba}. After doing that, still within the cut neighborhood $\mM$, we can connect the
boundaries of the ${\Sigma'}_j^{1,n}$ surfaces, $j=1,\dots,k$ and of the $d$ disks
by $m = kn + d$ straight cycles and redistribute these cycles over curves
connecting to our symplectic-tropical curve $\SC(\Gamma)$, as we did in
the Section \ref{subsec:Conn_local_models}. Figure \ref{fig:Sigma_final_amoeba}
illustrates the amoeba image of this local model of a deformed 
symplectic-tropical curve, when $n=4$, $m=11$, so $k=2$ and $d=3$. 
This finishes the proof
of Proposition \ref{prp:Disjoint_S_i}.  \qed

We can glue all these local models now, as 
we did in Section \ref{subsec:Conn_local_models}, to get deformed symplectic-tropical curves, intersecting Lagrangian 
spheres only at prescribed nodes, with the number
determined by the total multiplicity and the number of nodes
at a given cut.


\subsection{Further Deformations of Symplectic Tropical Curves} \label{subsec:Chains_STC_Prelim}

In this subection we introduce a series of additional techniques regarding symplectic-tropical 
curves, that will allow us to visualize chains of them inside an ATF. 
When we say a chain of symplectic curves, we imply that the
total intersection between them is equal to the geometric intersection.
Thus, it is not enough to simply construct an STC for each curve in the chain,
as we did in the previous sections, as we want to {\it geometrically} realize the homological intersection. We start with a simple observation:

\begin{rmk} \label{rmk:Crossing_STC}

 First, for $i=1,2$, let $\SC_i$ denote two STCs as 
  in Definition \ref{dfn:Symp_comp_graph}. Let $C_i$ denote a STC
  in $X$ represented by $\SC_i$ as in Theorem \ref{prp:symp_trop}, and let $\gamma_i$ be an edge of $\SC_i$. If the homology classes in the Lagrangian torus fibers associated with the edges $\gamma_1$ and $\gamma_2$ are the same, then any intersection between
  $\gamma_1$ and $\gamma_2$ can be taken to be empty as an intersection of the symplectic surface $C_1$ and $C_2$, since we can just assume we carry disjoint cycles 
  in the same homology class.  \hfill$\Box$
  \end{rmk}

\begin{figure}[h!]   
  \begin{center} 
 \centerline{\includegraphics[scale=0.5]{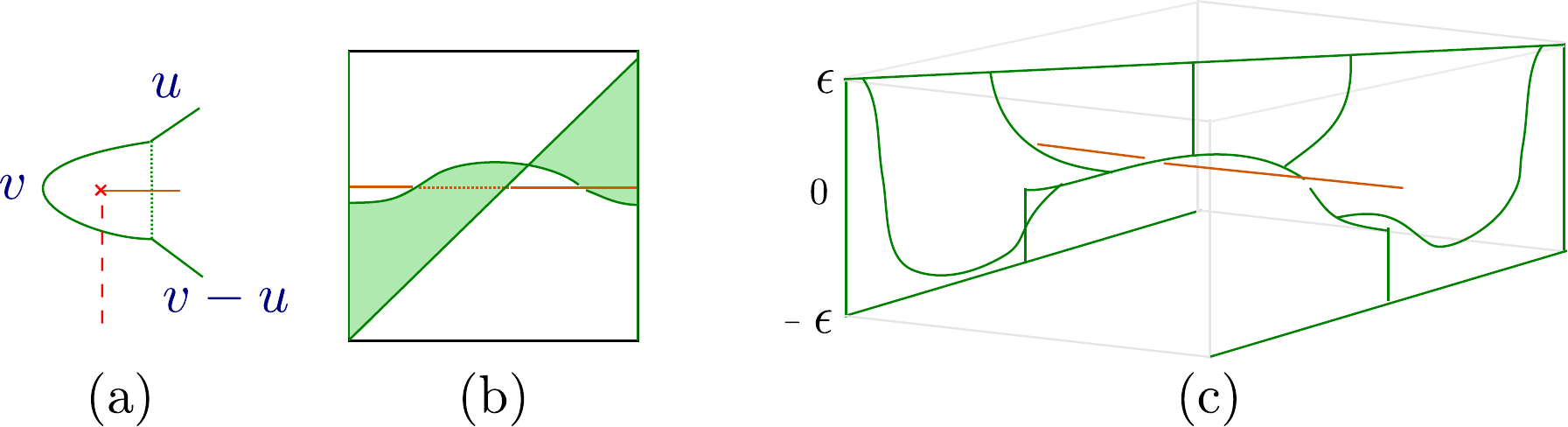}}
\caption{(a) The amoeba of a STC given by a disk $\sigma_j$ (from Proposition \ref{prp:Local_node}), glued along the boundary with 
two cycles intersecting once $\del \sigma_j$; (b) The coamoeba 
of the gluing surface; (c) A depiction of how the surface looks like inside a thickened torus.}
\label{fig:Esc_Cap} 
\end{center} 
\end{figure}

We will also need the following:
\begin{proposition} \label{prp:Esc_Cap}
  
 Consider a thickened torus $\bT = [-\epsilon , \epsilon ]\times T^2$, as the
 pre-image of a segment in the regular part of a base of an ATF. Let $\alpha$ be
 a straight cycle in $\{0\} \times T^2$ represented by $v \in H_1(\bT;\Z) \cong
 H_1(T^2;\Z) \cong \Z^2$, $\beta$ be a cycle in the class $v$ that wraps around $\alpha$ once, and $\gamma_\pm$ be a straight cycle in  $\{\pm \epsilon\} \times T^2$,
 represented by $u_\pm$, with $\det u_{\pm} \wedge v = \pm 1$, and $u_- = u_+ - v$.
 
 Then there exists a symplectic pair of pants in the complement in $\bT \setminus \alpha$, with boundary
 $\gamma_- \cup \beta \cup \gamma_+$.
  
\end{proposition}

\begin{proof}
  Use the dimer model represented in Figure \ref{fig:Esc_Cap} (b) to build a 
  symplectic surface as in Section \ref{subsec:Local_Int}, making sure that the
  $0$ level set of the height function ($\rho_2$ in Proposition \ref{prp:local interior})
  is disjoint from the straight cycle $\alpha$. The end result is depicted in 
  Figure \ref{fig:Esc_Cap} (c). 
\end{proof}

\begin{rmk} \label{rmk:Esc_Cap_STC}
  Applying this result for $\del \sigma_j$, the boundary of a
disk $\sigma_j$ as in Proposition \ref{prp:Local_node}, we see that we can pass
with all the $\sigma_l$, with $l < j$ 
(the ones with boundary closer to $\alpha = \del \sigma_1$) through the middle
of the surface constructed in Proposition \ref{prp:Esc_Cap}. The amoeba of this process is depicted
in Figure \ref{fig:Esc_Cap}.(a). \hfill$\Box$
\end{rmk}

\begin{figure}[h!]   
  \begin{center} 

       \centerline{\includegraphics[scale=0.6]{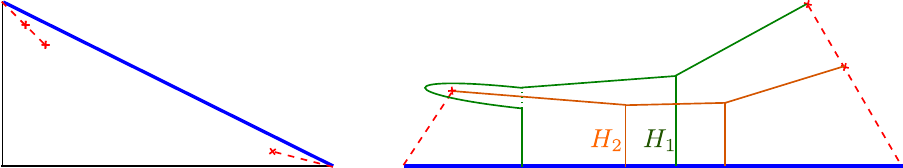}}

\vspace{0.5cm}

\centerline{\hspace{0.8cm} \includegraphics[scale=0.6]{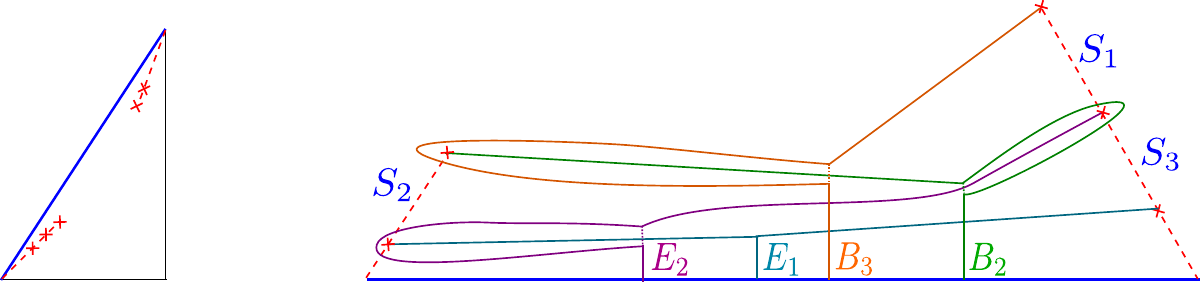}}

\vspace{0.5cm}

\centerline{\hspace{0.8cm} \includegraphics[scale=0.6]{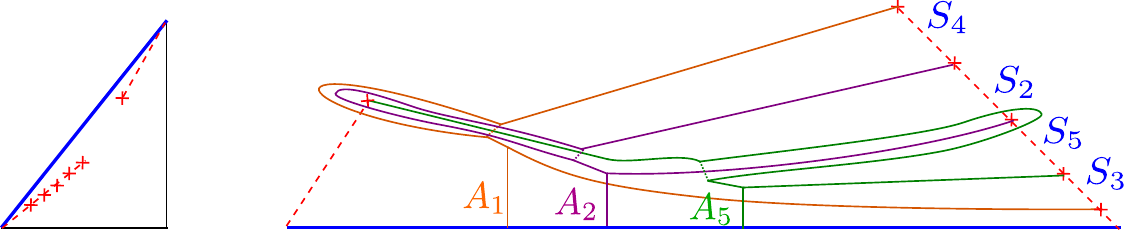}}

\caption{STCs in classes $H_1$, $H_2$ in a specific ATBD of $\PxP$ (in the first row), in the exceptional classes 
$E_1$, $E_2$, $B_2 = H - E_1 - E_3$ and $B_3 = H - E_1 - E_2$ in an ATBD of $\BlIII$ (in the second row), and in
the exceptional classes $A_1 = H - E_1 - E_4$, $A_2 = E_4$, and  
$A_5 = E_1$ in an ATBD of $\BlIV$ (in the third row). In this figure, the ATBDs are depicted on the left of each column on their own, and (a piece of) the ATBDs with the corresponding STCs are depicted on the right.}
\label{fig:PxP_3Blup} 
\end{center} 
\end{figure}

Now, the top left diagram in Figure \ref{fig:PxP_3Blup}, on the left of the first row, is an ATBD diagram for $\PxP$. It can be obtained from the top right diagram of Figure
\ref{fig:IntroTable}, after we apply nodal trades. For this diagram in Figure \ref{fig:PxP_3Blup}, we apply the above Remark
\ref{rmk:Esc_Cap_STC} to visualise STCs in the neighborhood $\fN$ of the highlighted edge of the ATF of
$\PxP$. They are representatives of the classes $H_1 = [\C\P^1 \times \{\pt \}]$ and $H_2= [\{\pt\} \times
\C\P^1 ]$. Remarks \ref{rmk:Esc_Cap_STC}, \ref{rmk:Crossing_STC}, are used to visualize a 4-chain of STCs in
the neighborhood $\fN$ of the highlighted edge of the ATF of $\BlIII$ of triangular shape depicted in Figure
\ref{fig:IntroTable}. The homology classes for the spheres in this 4-chain are the ones corresponding to the
highlighted edges of the toric diagram for $\BlIII$ in Figure \ref{fig:IntroTable}. Figure \ref{fig:PxP_3Blup}
also shows a 3-chain of symplectic spheres in the ATBD of $\BlIV$ of \cite[Diagram~$(A_4)$]{Vianna3}. Their
classes corresponds to the highlighted 3-chain in the first diagram of $\BlIV$ in Figure \ref{fig:IntroTable}.

Now, we can iteratively apply Remark \ref{rmk:Esc_Cap_STC} in the
neighborhood of one or more nodes. We are going to use simplified pictures, for
visual purposes. For instance, Figure \ref{fig:Simp_Pictures} (a) shows a
simplified depiction of two nonintersecting STCs near two nodes of a cut in an
ATF, with associated vector $v$. Surrounding each node, we see $a$ $\sigma_j$
type curve, where we applied Remark \ref{rmk:Esc_Cap_STC} $a$ times around each
node and then unite their amoebas using Remark \ref{rmk:Crossing_STC}. Figure
\ref{fig:Simp_Pictures} (b) shows an alternative version, where we took $2b$
curves around a unique node associated with the vector $w$, and applied
Remarks \ref{rmk:Esc_Cap_STC}, \ref{rmk:Crossing_STC} to
get two disjoint symplectic curves.\\

\begin{figure}[h!]   
  \begin{center} 

 \centerline{\includegraphics[scale=0.6]{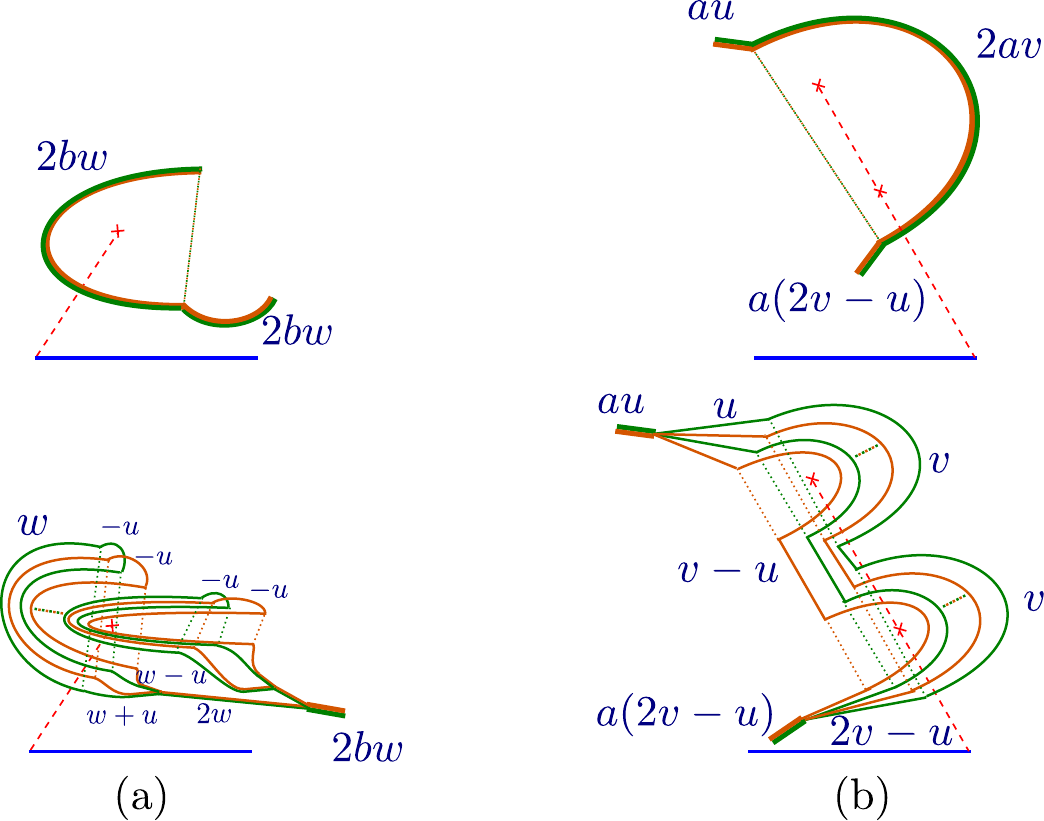}}

\caption{Simplified pictures for the STC discussed above.}
\label{fig:Simp_Pictures} 
\end{center} 
\end{figure}

Let us now shift our focus to what happens as we approach the local model nearby a trivalent 
vertex of an STC, with two (or more) sets of curves satisfying the balancing 
condition \eqref{eq:BalCond}. The first observation is that if you arrive 
with two sets of cycles with total homology represented by $\bbv \in \R^2$ and want to glue them
to sets of cycles with homoloy $\bbw$ and $-\bbw - \bbv$, using two non-intersecting
surfaces, the price you pay is that the cycles corresponding to $\bbw$ and $-\bbw - \bbv$
in the boundary of the second surface, must link the corresponding boundaries on the 
first surface a total amount\footnote{For us $\bbw \wedge \bbv$
	denotes the determinant of the two vectors in $\R^2$.} of $\bbw \wedge \bbv$. We define precisely what 
we mean by cycles linking within the regular part of a Lagrangian torus 
fibration:

\begin{definition} \label{dfn:Linking} Let $\DD$ be a 2-disk, fix a point $o \in \DD$ and take two disjoint 1-cycles $\alpha$, $\sigma$ in $\DD
\times T^2$ away from $\{o\}\times T^2$. We view $\sigma$ as a cycle in $H_1(\DD
\times T^2 \setminus \alpha; \Z) \cong H_1(T^2; \Z) \oplus \Z$, where the 
first summand corresponds to $\{o\}\times T^2 \hookrightarrow \DD
\times T^2 \setminus \alpha$.

By definition, the \emph{linking} between $\sigma$ and $\alpha$
is the projection of the class $[\sigma] \in H_1(\DD
\times T^2 \setminus \alpha; \Z)$ onto the (rightmost) $\Z$-factor. 
The sign involves a choice of generator for the $\Z$ cycle, 
that we assume the same, when dealing with more than one 1-cycle relative
to $\alpha$.\hfill$\Box$
\end{definition}

Let us summarize the above discussion into the following statement:

\begin{proposition} \label{prp:Link_dimer} Consider the local model for a
symplectic surface near the interior vertex constructed in Proposition \ref{prp:local
interior}, associated to the balancing condition $m_1\bbw_1 + m_2\bbw_2
+m_3\bbw_3 = 0$. Consider the number $d = m_1m_2 |\bbw_1 \wedge \bbw_2|$ and $d= \delta_1 + 
\delta_2$, a two partition $\delta_1, \delta_2 \in \Z_{\ge 0}$.

Then there exists another
disjoint symplectic surface in the same local neighborhood such that the
boundaries satisfies the following two conditions:

\begin{itemize}
	\item[-] The $m_3$ boundaries associated to $\bbw_3$ are parallel 
	copies of the corresponding boundaries of the original curve inside the same torus 
	fibre;
	\item[-] The boundaries associated to $\bbw_1$ and $\bbw_2$ link the corresponding 
	boundaries of the original curve, in the sense of Definition \ref{dfn:Linking},
	$\delta_1$ and $\delta_2$ times, respectively. 
\end{itemize} 
\end{proposition}

\begin{proof}
  
  We start with the dimer model for the original surface constructed in
  Proposition \ref{prp:local interior}. Recall that we used $(\rho_1,\rho_2)$
  coordinates for the amoeba description, with $\rho_2 \in [-\epsilon,
  \epsilon]$. We will construct another dimer model for our second surface, and
  build the amoeba as in Sections \ref{subsec:Local_Int},
  \ref{subsec:Conn_local_models}, \ref{subsec:def_symptrop_curves}.
  
  	\begin{figure}[h!]   
  	\begin{center} 
  		
  		\centerline{\includegraphics[scale=0.4]{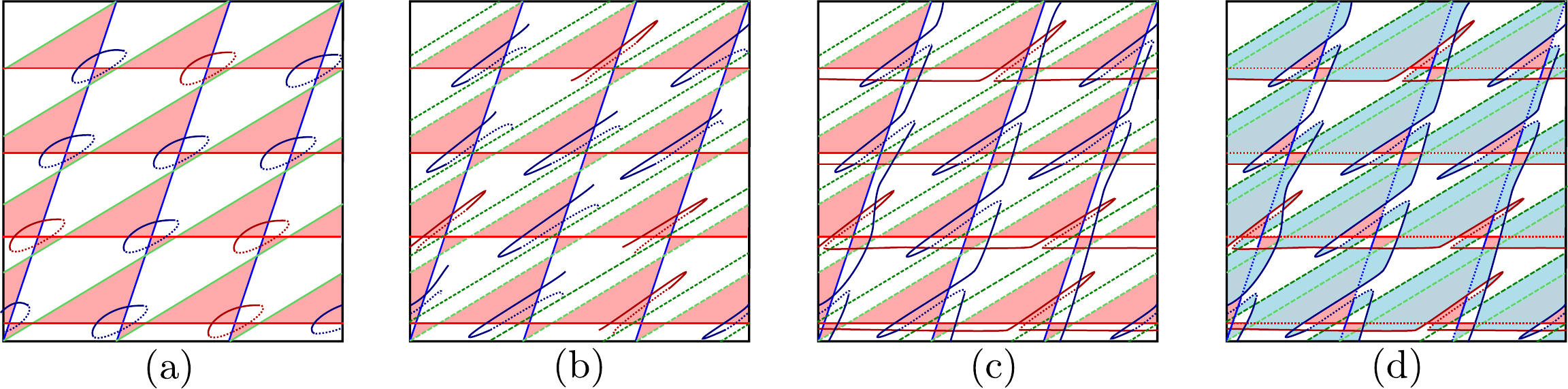}}
  		
  		\caption{Near an interior vertice of determinant $d$, 
  			there are disjoint symplectic curves with the total linking of their boundaries equal 
  			to $d$. This series of diagrams illustrate the algorithm to get the disjoint surfaces.}
  		\label{fig:DimerLink} 
  	\end{center} 
  \end{figure}
  
  In the
  intersection of these two dimer models we will record the $\rho_2$-coordinate
  of the new curve in the dimer model, to be less than $-\epsilon$ or greater than $\epsilon$. This will be indicated in the same diagram, as follows. If
  coamoeba regions of the two surfaces intersect, the boundary of the lower
  region will be denoted by a dotted segment in the intersection. Figure
  \ref{fig:DimerLink} (d) illustrates two non-intersecting surfaces via their
  coamoeba projection. By the work developed in Sections \ref{subsec:Local_Int},
  \ref{subsec:Conn_local_models}, \ref{subsec:def_symptrop_curves}, 
  the construction of such dimer model, with the additional $\rho_2$ 
  information, will be enough to ensure that we obtain two disjoint symplectic surfaces. 
  
  As before we take $\bbw_1 = (1,0)$, and we use the following algorithm to 
  construct the dimer model, which concludes Proposition \ref{prp:Link_dimer}.

  {\it Algorithm for the Dimer Model:} 
  
  \begin{enumerate}[label= Step \arabic*]
    
    \item Color the $\bbw_1$ cycles red, $\bbw_2$ cycles blue, $\bbw_3$ cycles green.
    
    \item \label{item:Links} For each of the $|d| = m_1m_2 |\bbw_1 \wedge \bbw_2|$ intersections between 
  the red and blue cycles of the original curve, draw cycles linking both
  blue and red cycles, as illustrated in Figure \ref{fig:DimerLink} (a), 
  in the same pattern. Color $\delta_1$ of them red and the other $\delta_2$ of them
  blue.   

    \item \label{item:Parallel} Consider a new green cycle, parallel to the original green, constructed as a positive shift in, say, the $\phi_2$ coordinate of the amoeba (the $\phi_i$-coordinates being
    the coamoeba coordinates as in the notation of Proposition \ref{prp:local interior}). 
    
  \item Replace the red/blue links of \ref{item:Links}, by chains of the
             same color, ``linking'' the corresponding original curve, with one
             end on the new green cycle and crossing the original blue and red
             cycle twice, to the left of the new green cycle. The first crossing
             is above and the second below the original dimer, with respect to
             the $\rho_2$ coordinate, as illustrated in Figure
             \ref{fig:DimerLink} (b).
    
    \item \label{item:connect} For each red/blue linking chain, we connect its ``tail'' with the ``head'' 
    of the adjacent chain of the same color, using a red/blue 1-chain 
    parallel to the original chain of the same color, forming
    the new red/blue cycles as illustrated in Figure \ref{fig:DimerLink} (c).     
      
    \item \label{item:shade} Paint the regions that were created by the new green, red and blue cycles,
    passing below or above the original dimer accordingly, as illustrated in 
    Figure \ref{fig:DimerLink} (d).
	\end{enumerate}
\end{proof}

\begin{remark} \label{rmk:Linking}
	We could allow linkings for the $\bbw_3$ cycles, 
	provided the total linking is still $d$. This is not needed 
	for our purposes, and it would make the construction algorithm more intricate.\hfill$\Box$
\end{remark}

\begin{remark} \label{rmk:n_dimerLink}
  After getting the two surfaces of Proposition \ref{prp:Link_dimer},
  one can actually run an analogous algorithm to get yet a third, fourth and $n$-th surfaces, 
  disjoint from the previous ones, with boundary so that the green cycle 
  is parallel to the previous green cycles, and the red/blue cycles links
  each red/blue cycles of the previous surfaces $\delta_1$/$\delta_2$ times.
  For that, one just needs to replicate the intersection pattern of the $n$-th red/blue curve, 
  with the previous $(n-1)$ ones.\hfill$\Box$
\end{remark}

Now assume that we arrive at the surface near the interior node, constructed in
Proposition \ref{prp:local interior}, with $m_3$ straight (green) $\bbw_3$
cycles, parallel to the original one, and $m_1$ (red) $\bbw_1$ cycle, linking the original curve   
$c$ times. We can adjust the new red curves, so that the curves arrive linking 
the dimer model over one red/blue vertex as in Figure \ref{fig:Local_linking} 
(a).

\begin{proposition} \label{prp:link+d}
  Consider the above setup and $d\in\N$ as in Proposition \ref{prp:Link_dimer}. Then there exists a symplectic surface connecting the $m_3$ $\bbw_3$-cycles (green) and $m_1$ $\bbw_1$-cycles (red), 
  with $m_2$ $\bbw_2$-cycles (blue), linking the original $\bbw_2$-cycles $(c + d)$ times. 
\end{proposition}

\begin{figure}[h!]   
  \begin{center} 

 \centerline{\includegraphics[scale=0.6]{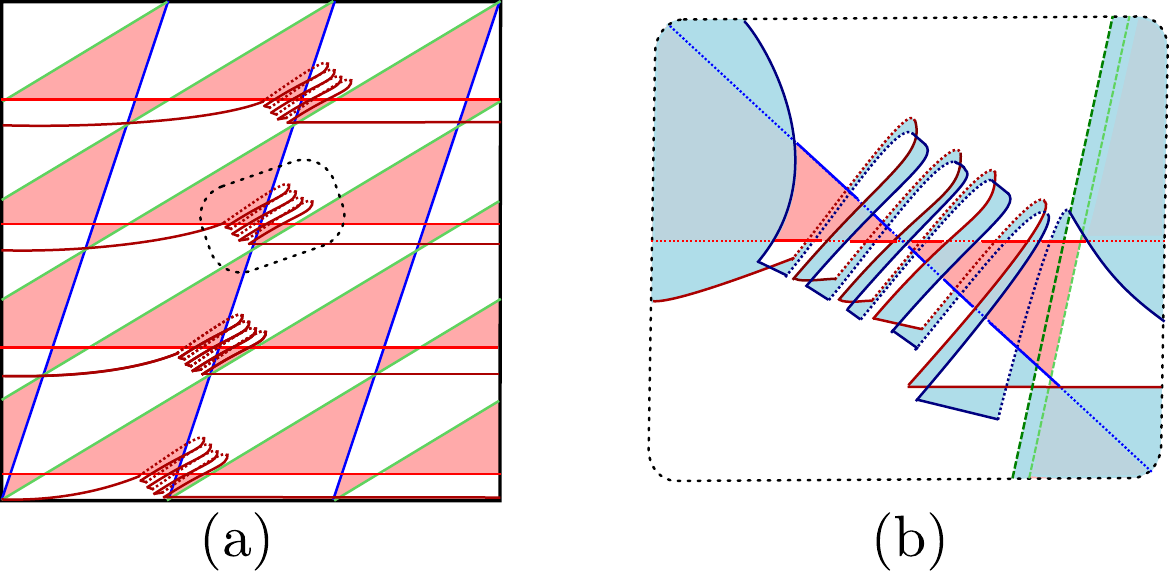}}

\caption{How linking can change for non-intersecting symplectic surfaces near the interior vertex of a STC.
Picture (b) is a zoomed version around the node with red links. The blue curve is a bit rotated for visual purposes.}
\label{fig:Local_linking} 
\end{center} 
\end{figure}

\begin{proof}
  At each vertex having the linking of the red curves, we construct a local model as illustrated by Figure 
  \ref{fig:Local_linking} (a) and (b). We note that the new blue chain links the 
  original blue at the red linking number plus 1. In Figure \ref{fig:Local_linking}, 
  the red cycles are linking 4 times and the new blue cycle links 5 times the 
  original blue cycle. (A generalization of the picture is clear.) Now we add 
  blue links to the remaining intersections of the originals blue and red curves,
  as in \ref{item:Links} of the algorithm of Proposition \ref{prp:Link_dimer}. 
  Then we can run \ref{item:connect}, \ref{item:shade} of the above algorithm in 
  an analogous fashion, noting that the local model of Figure \ref{fig:Local_linking} 
  (b) is well adapted for that.
\end{proof}

\begin{rmk} \label{rmk:n_link+d} As in Remark \ref{rmk:n_dimerLink}, let us assume that we have arrived to a 3rd set of $m_1$ red cycles (correspondingly 4th,\ldots,$n$th), each linking $c'$ times one red cycle for each of the previous
surfaces, and a 3rd set of $m_3$ green cycles (correspondingly 4th,\ldots,$n$th), parallel to the green cycles of the
previous surfaces. By inspecting Figure \ref{fig:Local_linking} (a) and (b), now imagining that the red
surface represents the two previously constructed surfaces (correspondingly three,\ldots, and $n-1$ surfaces), close to each other, and the third surface is represented by the blue surface in Figure \ref{fig:Local_linking}.(b) (correspondingly by the 4th,\ldots, and $n$th surface). As in
the proof of Proposition \ref{prp:link+d}, after the initial adjustment of concentrating the several red links in
one node as before, we can locally construct the third blue link (correspondingly 4th,\ldots,$n$th) in a manner that locally links all other
previous blue cycles with linking number being the local red linking number plus one, i.e. with linking number $c' + 1$. Close to the boundary of the local
region, the third local surface will lay above the previous two surfaces (similarly above the previous three,\ldots, and $n-1$ surfaces). 
Hence we can globally construct the third surface as in the proof of Proposition \ref{prp:link+d} (and similarly with the 4th,\ldots, and $n$th surfaces). \hfill$\Box$
  

\end{rmk}


\subsection{Getting chains of symplectic-tropical curves} \label{subsec:Chains_STC}

Let us apply the results of Section \ref{subsec:Chains_STC_Prelim} to construct the required chains of symplectic-tropical curves used in Theorem \ref{sec:proofmain}. There are three cases, corresponding to $\CP^1\times\CP^1$, $\BlIII$ and $\BlIV$, which we now analyze.

\subsubsection{The case of $\CP^1\times\CP^1$.} \label{subsubsec:Chains_STC_PxP}
Consider the triangular-shaped ATF of the symplectic monotone $\PxP$, with a smooth corner, 
associated to a solution of the Diophantine equation
$$1 + q^2 + 2r^2 = 4qr.$$
Let $\fN$ be a neighborhood of the edge opposite the smooth corner -- where the frozen vertex is located -- and consider its associated cuts, as in Section \ref{subsubsec:Neigh_Edge}.

\begin{proposition} \label{prp:PxP_H1H2} There is a 2-chain of symplectic-tropical
curves inside the edge neighborhood $\fN$, such that the associated symplectic curves belong to the classes $H_1$ and $H_2$, and have total
intersection number one, i.e., equal to their topological intersection. \end{proposition}

\begin{proof} We need to revisit the specific combinatorics in this situation. The associated Markov type equation of interest is $1 + q^2 + 2r^2 = 4qr.$ 
From Equation \eqref{eq:lm}, we see that the determinant between the associated vectors 
$v = (-r,m)$ and $w = (q, -l)$ is $v \wedge w = 2$. It follows from the corresponding Vieta jumping (Proposition \ref{prop:arithmetic}) that $q = 2a + 1$ and $r = 2b +1$ are odd. Hence, we can rewrite the balancing condition as: 
\begin{equation} \label{eq:vwu}
  qv + rw + 2(0,1) = v + w + 2u = 0
\end{equation}  
where $u = av + bw + (0,1)$. We readily see that $w \wedge u =  u \wedge v 
= 1$. In consequence, we are allowed to use the vectors $\pm u$, with each of the vectors $v$ and $w$ 
as in Proposition \ref{prp:Esc_Cap}. Finally, we look at 
\[ u \wedge (0,1) =  -\frac{q-1}{2}r + \frac{r-1}{2}q = \frac{r - q}{2}. \]

Assume that $q > r$, so that $u$ points to the left. Then we deduce that 
$$ 2v - u \wedge (0,1) = \frac{q -5r}{2} < 0,$$
because $1 + 2r^2 = q(4r - q)$, so $q < 4r$.

\begin{figure}[h!]   
  \begin{center} 

 \centerline{\includegraphics[scale=0.5]{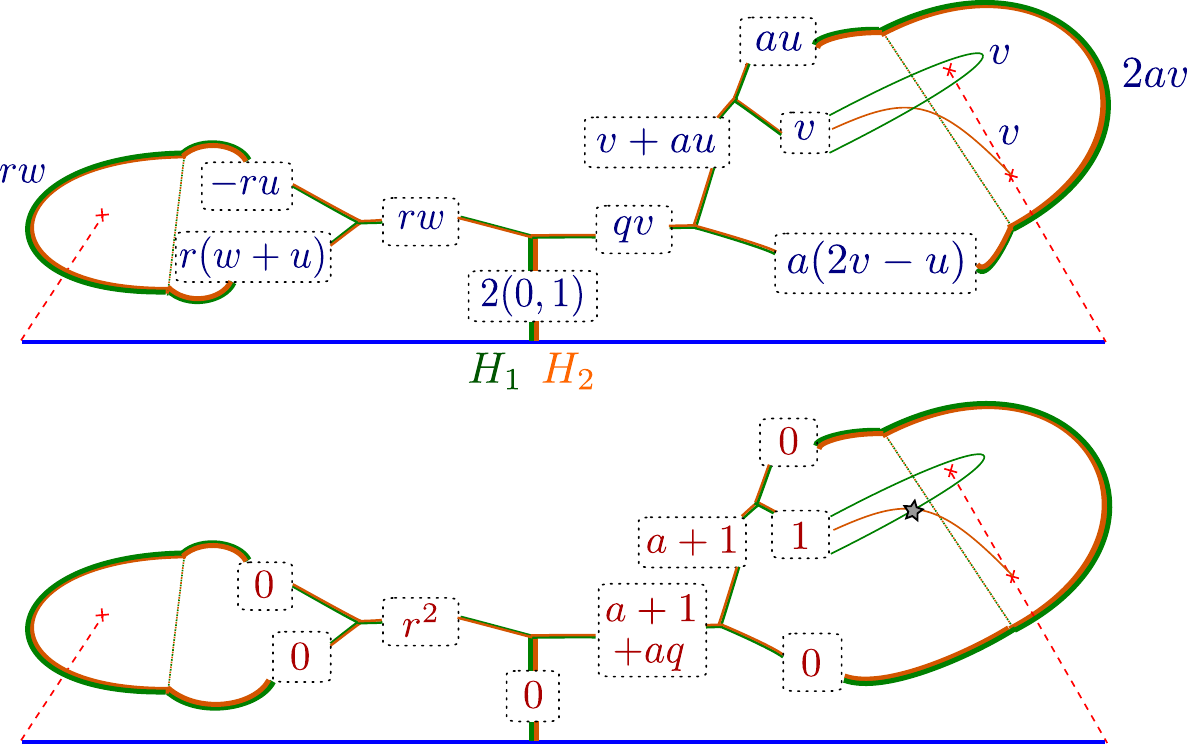}}

\caption{STCs in classes $H_1$, $H_2$ in a neighborhood $\fN$ of the edge 
opposite to the smooth corner in an arbitrary ATF of $\PxP$ associated with a 
solution of $1 + q^2 + 2r^2 = 4qr$.}
\label{fig:PxP_H1H2} 
\end{center} 
\end{figure}
 
Now, Figure \ref{fig:PxP_H1H2} is a depiction of the chain of symplectic spheres in $\fN$
in classes $H_1$ and $H_2$. The intersection is depicted by a star. The top picture records the homology classes of the 
cycles corresponding to each edge. The bottom picture records the linking 
between the two cycles obtained by applying Propositions \ref{prp:Link_dimer} and \ref{prp:link+d}. 
We note that we get the Markov equation as a compatibility equation for 
the interior vertice associated with the $2(0,1)$ cycles. Indeed, 
since $qv\wedge rw = 2qr$, we must have
\[ 2qr = r^2 + a + 1 + aq = \frac{2r^2 + q - 1 + 2 + (q-1)q}{2} =  \frac{2r^2 + q^2 + 1}{2}. 
\]

In the case $r > q$, we replace the vector $u$ by $-u$ in Figure
\ref{fig:PxP_H1H2}, and note that $u \wedge (0,1) > 0$ and
$$ 2v + u \wedge (0,1)
= \frac{-q -3r}{2} < 0.$$
The case $r = q$ is depicted in Figure \ref{fig:PxP_3Blup}.
\end{proof}

\begin{remark}
The construction obtained by the above picture in Figure \ref{fig:PxP_H1H2} is equivalent to the construction of two (geometrically) disjoint copies of the $H_1$ 
	class, to which we can apply a Dehn twist with respect to the visible Lagrangian 2-sphere.\hfill$\Box$
\end{remark}


\subsubsection{The case of $\BlIII$.} \label{subsubsec:Chains_STC_BlIII}

Let us now consider a triangular-shaped ATF for the symplectic 4-manifold $\BlIII$, with a smooth corner, 
associated to a solution of the Diophantine equation
$$1 + 2q^2 + 3r^2 = 6qr,$$
$\fN$ a neighborhood of the edge opposite the smooth corner, and its associated cuts, as in Section \ref{subsubsec:Neigh_Edge}. In this case, the required chain of symplectic surfaces reads:

\begin{proposition} \label{prp:3Blup_A1A2A5} There exists a 4-chain of symplectic-tropical
curves inside the edge neighborhood $\fN$ whose associated symplectic curves lie in the exceptional classes $E_1$, $E_2$, $B_2 = H - E_1 - E_3$ and $B_3 = H - E_1 - E_2$, and the intersection between two of them equals their geometric intersection. \end{proposition} 

\begin{proof}
  
From Equation \eqref{eq:lm}, we deduce that the determinant between associated vectors 
$v = (-r,m)$ and $w = (q, -l)$ is $v \wedge w = 1$. From the corresponding Vieta jumping (Proposition \ref{prop:arithmetic}), we obtain that $q^2 \equiv -1 \mod 3$, 
and thus $q \equiv 1$ or $2 \mod 3$, and $r \equiv 1 \mod 2$.

\begin{figure}[h!]   
  \begin{center} 

 \centerline{\includegraphics[scale=0.5]{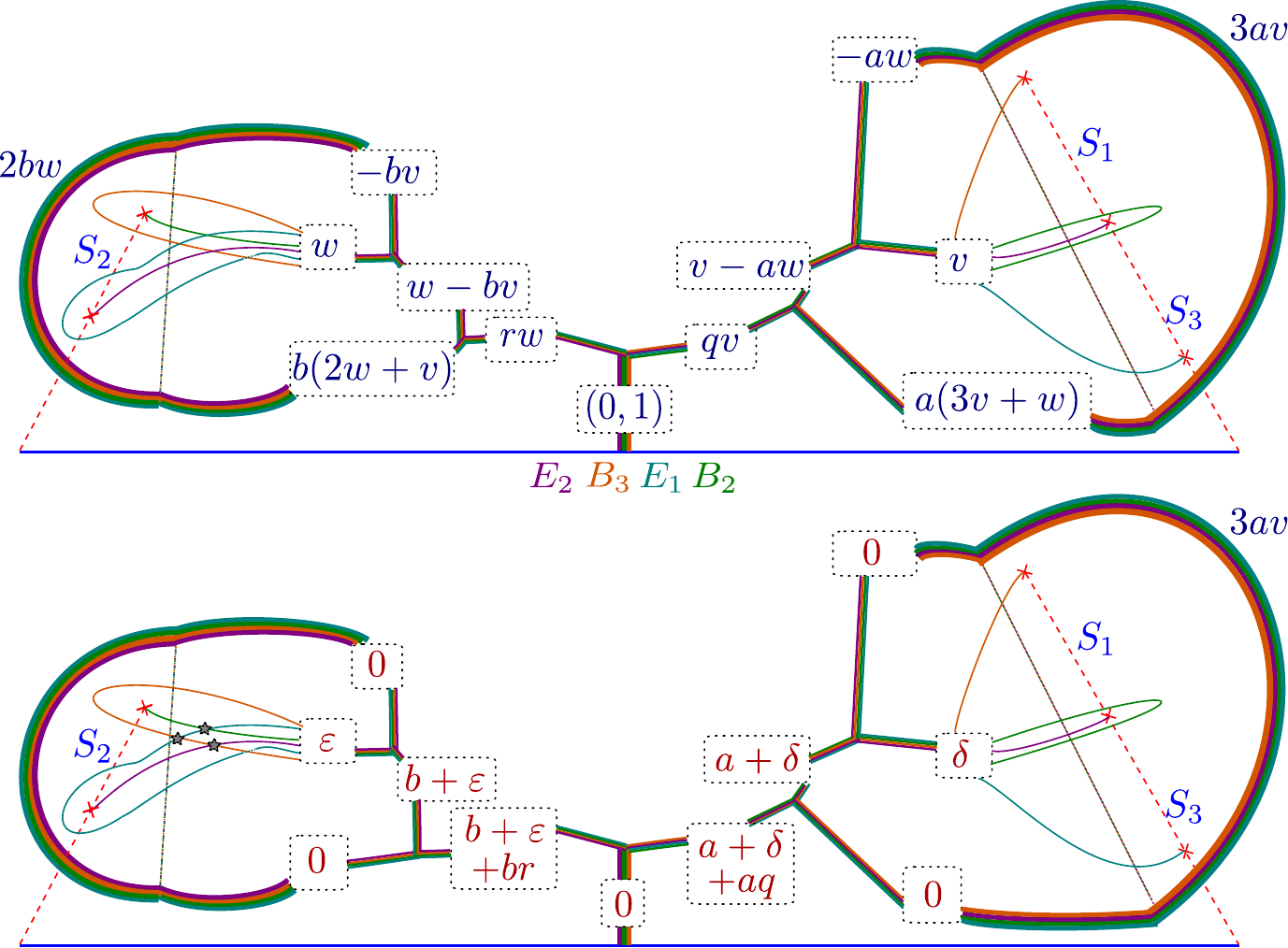}}

\caption{STCs in classes $E_1$, $E_2$, $B_2$ and $B_3$ in a neighborhood $\fN$ of the edge 
opposite to the smooth corner in a arbitrary ATF of $\BlIII$ associated with a 
solution of $1 + 2q^2 + 3r^2 = 6qr$.}
\label{fig:3Blup_E1E2B2B3_q1} 
\end{center} 
\end{figure}
 
 Figure \ref{fig:3Blup_E1E2B2B3_q1} illustrates the case $q = 3a + 1$, $r=2b +1$.  As before, 
the intersection is depicted by a star in the bottom picture, and occurs exactly as we change the linking number between cycles. 
The top picture records the homology classes of the 
cycles corresponding to each edge. The bottom picture records the linking 
between the two cycles by applying Propositions \ref{prp:Link_dimer} and \ref{prp:link+d}.
In this case, the linking numbers $\delta$ and $\varepsilon$, will depend on the curves
we are taking into account. We have that $\varepsilon,\delta \in \{0,1\}$ and $\delta 
+ \varepsilon = 1$. The compatibility condition becomes:
\[ qr = b + \varepsilon +br + a + \delta +aq = \frac{ 3(r-1)(r+1) + 6 + 2(q - 1)(q+1)}{6} =  \frac{3r^2 + 2q^2 + 1}{6} 
\]

Figure \ref{fig:3Blup_E1E2B2B3_q2} illustrates the second case $q = 3a + 2$, $r=2b +1$. 
In this case, we have $\varepsilon \in \{0,1\}$, $\delta \in \{1,2\}$ and $\delta 
+ \varepsilon = 2$, where the compatibility becomes:

\[ qr = b + \varepsilon +br + 2a + \delta +aq = \frac{ 3(r-1)(r+1) + 12 + 2(q-2)(q+2)}{6} =  \frac{3r^2 + 2q^2 + 1}{6},
\]
as required. This concludes the construction for the case of $\BlIII$.
\begin{figure}[h!]   
  \begin{center} 

 \centerline{\includegraphics[scale=0.5]{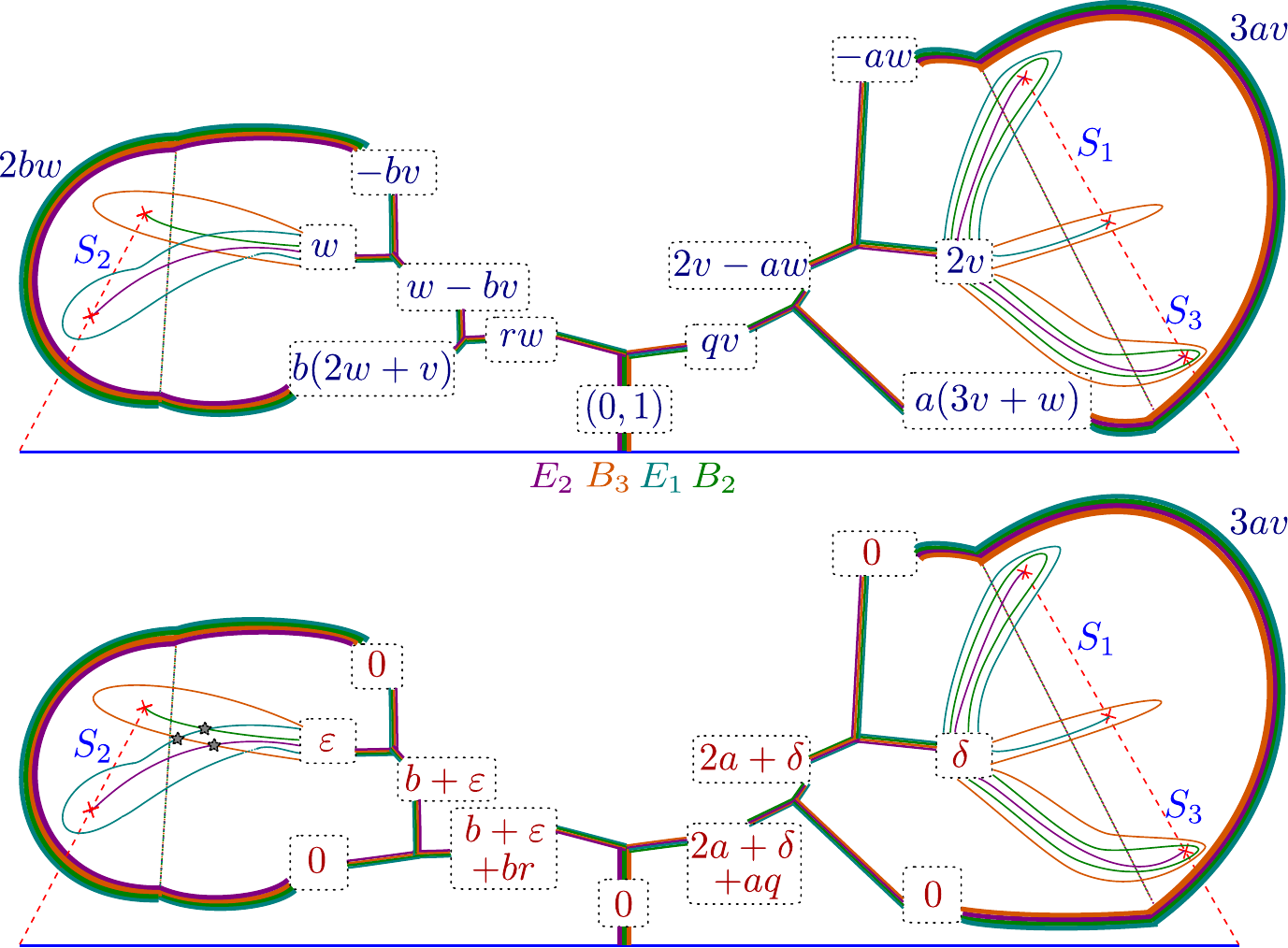}}

\caption{STCs in classes $E_1$, $E_2$, $B_2$ and $B_3$ in a neighborhood $\fN$ of the edge 
opposite to the smooth corner in a arbitrary ATF of $\BlIII$ associated with a 
solution of $1 + 2q^2 + 3r^2 = 6qr$.}
\label{fig:3Blup_E1E2B2B3_q2} 
\end{center} 
\end{figure}

\end{proof}

\subsubsection{The case of $\BlIV$.} \label{subsubsec:Chains_STC_BlIV}

Finally, we consider a triangular-shaped ATF for the symplectic surface $\BlIV$, with a smooth corner, 
associated to a solution of the Diophantine equation
$$1 + q^2 + 5r^2 = 5qr,$$
$\fN$ a neighborhood of the edge opposite the smooth corner, together with the associated cuts, as in Section \ref{subsubsec:Neigh_Edge}. The required chain of symplectic curves is obtained in the following

\begin{figure}[h!]   
  \begin{center} 

 \centerline{\includegraphics[scale=0.5]{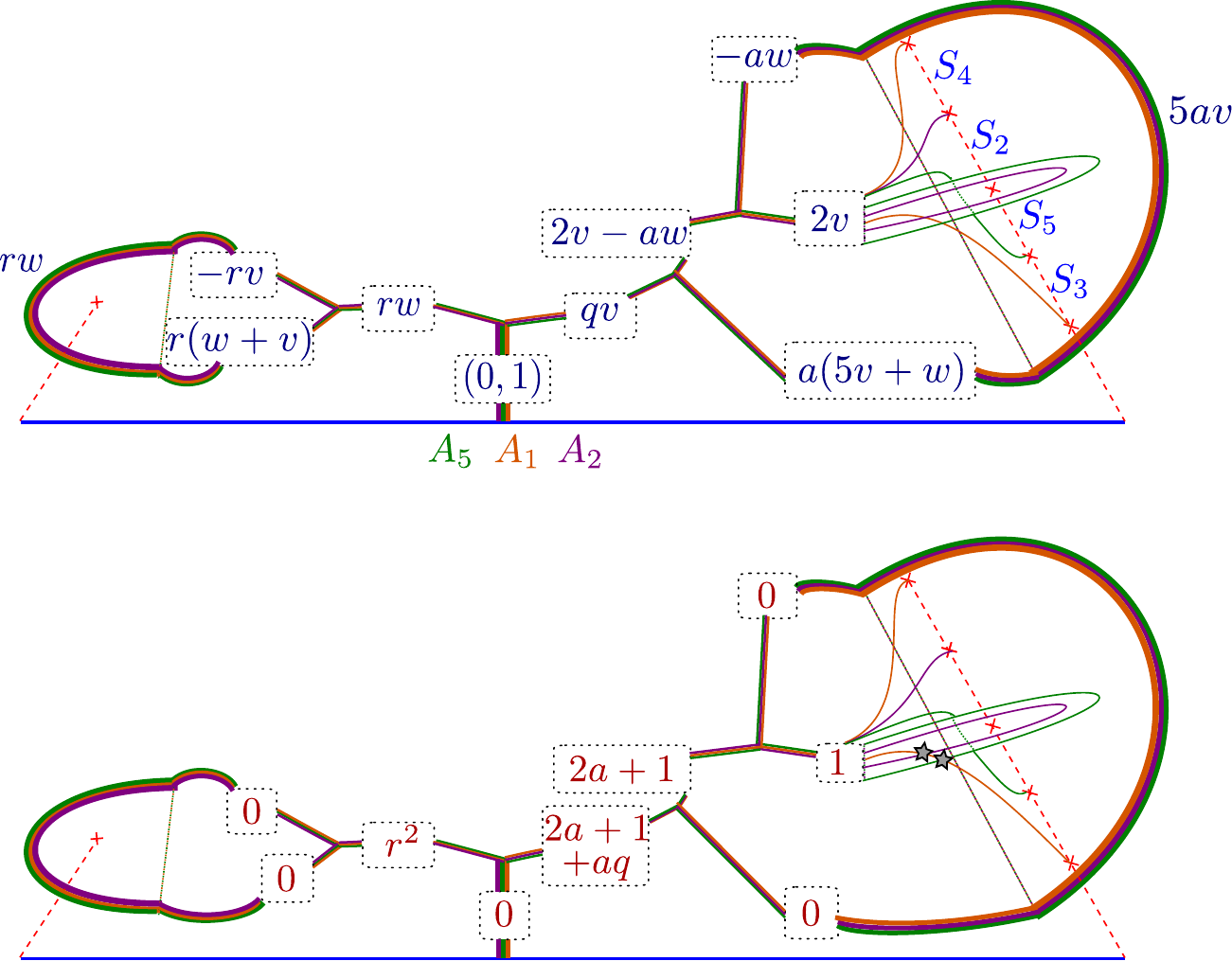}}

\caption{STCs in classes $A_1$, $A_2$ and $A_5$ in a neighborhood $\fN$ of the edge 
opposite to the smooth corner in a arbitrary ATF of $\BlIV$ associated with a 
solution of $1 + q^2 + 5r^2 = 5qr$.}
\label{fig:4Blup_A1A2A5} 
\end{center} 
\end{figure}

\begin{proposition} \label{prp:4Blup_A1A2A5} There exists a 3-chain of symplectic-tropical
curves inside the edge neighborhood $\fN$ whose associated symplectic curves belong to the exceptional classes $A_1 = H - E_1 - E_4$, $A_2 = E_4$, and  
$A_5 = E_1$, and their pairwise intersections equal their geometric intersections. \end{proposition} 

\begin{proof} Let us revisit the specific combinatorics of the situation: the associated Markov type equation of interest is $1 + q^2 + 5r^2 = 5qr$, 
and from Equation \eqref{eq:lm}, the determinant between the associated vectors 
$v = (-r,m)$ and $w = (q, -l)$ is $v \wedge w = 1$. As above, Proposition \ref{prop:arithmetic} shows that $q^2 \equiv -1 \mod 5$, 
and thus $q \equiv 2$ or $3 \mod 5$. Figure \ref{fig:4Blup_A1A2A5} shows the case $q = 2 + 5a$. 

\begin{figure}[h!]   
  \begin{center} 

 \centerline{\includegraphics[scale=0.5]{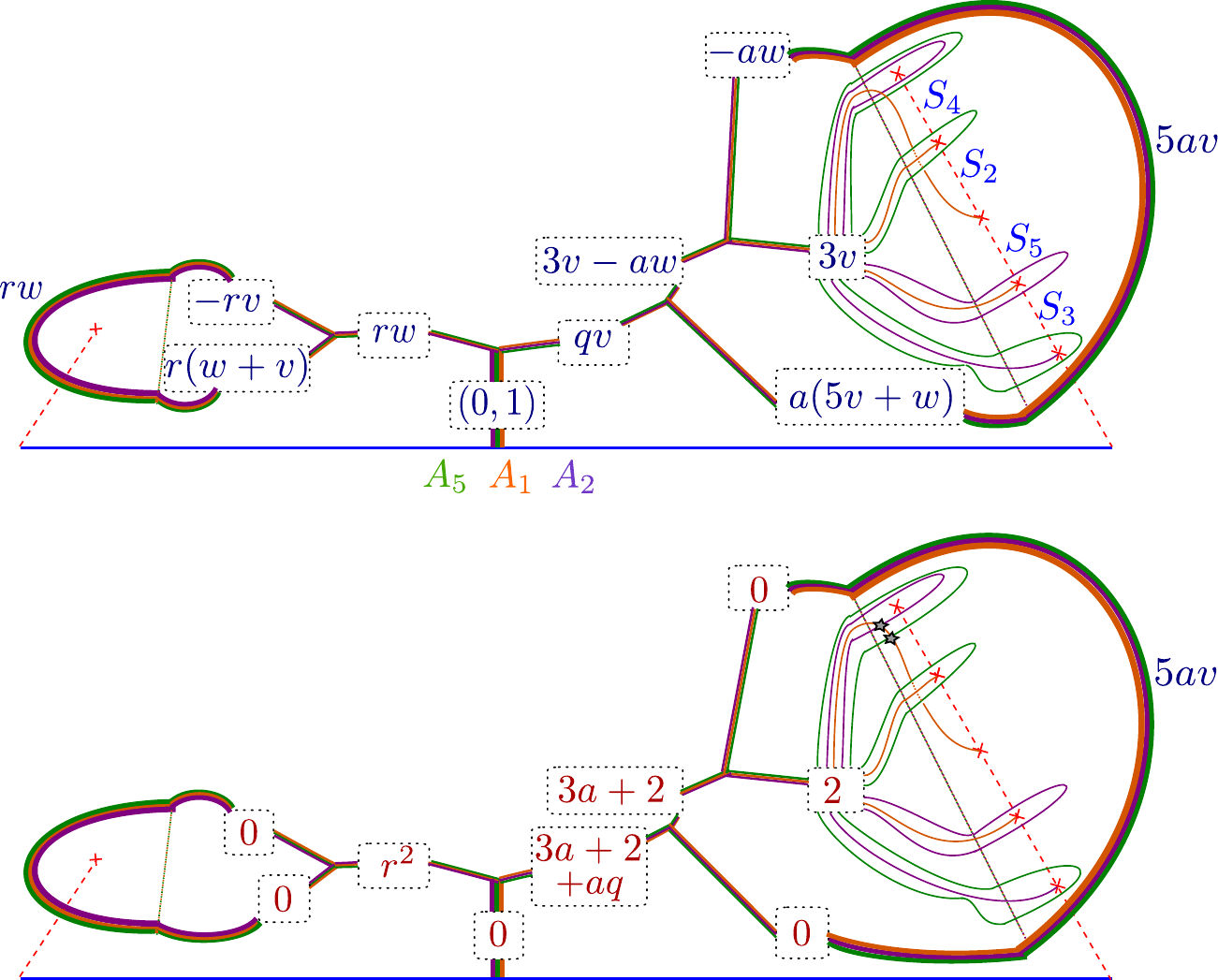}}

\caption{STCs in classes $A_1$, $A_2$ and $A_5$ in a neighborhood $\fN$ of the edge 
opposite to the smooth corner in a arbitrary ATF of $\BlIV$ associated with a 
solution of $1 + q^2 + 5r^2 = 5qr$.}
\label{fig:4Blup_A1A2A5_q3} 
\end{center} 
\end{figure}
 
As before, we have the compatibility associated with the $(0,1)$ cycles, using $qv\wedge rw = 
qr$, and giving the associated Diophantine equation:

\[ qr = r^2 + 2a + 1 + aq = \frac{5r^2 + (q - 2)(q+2) + 5}{5} =  \frac{5r^2 + q^2 + 1}{5}.
\]

For the second case, Figure \ref{fig:4Blup_A1A2A5_q3} shows the case $q = 3 + 5a$, and the compatibility becomes: 

\[ qr = r^2 + 3a + 2 + aq = \frac{5r^2 + (q - 3)(q+3) + 10}{5} =  \frac{5r^2 + q^2 + 1}{5}. 
\]

This concludes the verification for the case of $\BlIV$.
\end{proof}

\bibliographystyle{plain}
\bibliography{CasalsVianna_SharpEllipsoids}
\end{document}